\documentclass[11pt, reqno]{amsart}
\usepackage{amsfonts}
\usepackage{amscd}
\usepackage{amssymb}
\usepackage{amsthm}
\usepackage{amsmath}
\usepackage{stmaryrd}
\usepackage{amsfonts}
\usepackage{mathrsfs}
\usepackage{color}
\usepackage{tikz}
\usetikzlibrary{arrows}
\usepackage{mathdots}
\usepackage{multicol}

\usepackage{mathtools}
\usepackage{pdfsync}%leaves makers for tex searching
\usepackage{enumerate}
\usepackage{pdflscape}

\theoremstyle{plain}
	\newtheorem{thm}{Theorem}[section]
	
	\newtheorem{prop}[thm]{Proposition}
	\newtheorem{cor}[thm]{Corollary}
\theoremstyle{definition}
	
	\newtheorem{remark}[thm]{Remark}
\theoremstyle{example}
	\newtheorem*{example}{Example}
	
\theoremstyle{remark}
	
\numberwithin{equation}{section}

\def\cB{\mathcal{B}}\def\cD{\mathcal{D}}\def\cF{\mathcal{F}}\def\cP{\mathcal{P}}\def\cW{\mathcal{W}}

  \def\CC{\mathbb{C}}                       \def\ZZ{\mathbb{Z}}

\def\fgl{\mathfrak{gl}}  \def\fsl{\mathfrak{sl}}

\def\cc{\mathbf{c}}

\DeclareMathAlphabet{\mathpzc}{OT1}{pzc}{m}{it} %defines Zapf Chancery font, for calligraphic lower case letters
\def\ze{\mathpzc{e}}\def\zg{\mathpzc{g}}

         \def\zJ{\mathpzc{J}}               \def\zZ{\mathpzc{Z}}
\def\zTL{\mathpzc{TL}}

\def\dim{\mathrm{dim}}

\def\ext{\textrm{ext}}

\def\Res{\mathrm{Res}}

\def\vep{\varepsilon}

\def\half{{\hbox{$\frac12$}}}
\def\<{\langle}	\def\>{\rangle}

\def\({[\![}
\def\){]\!]}

\usepackage{color}
\definecolor{dred}{rgb}{.65, 0, 0.15}
\definecolor{arun}{rgb}{.8,0,0}
\definecolor{zajj}{rgb}{.65, 0, .85}
\definecolor{gray}{rgb}{0.6, .6, .6}

\usepackage{tikz}
	\usepgflibrary[patterns] % ConTEXt and pure pgf 
	\usetikzlibrary{patterns} % LATEX and plain TEX when using TikZ 
	\usepgflibrary{shapes.geometric}

\tikzstyle over=[draw=white,double=black,line width=2pt, double distance=.5pt]
\tikzstyle{B}=[draw, fill=black, circle, inner sep=0pt, outer sep=0pt, minimum size=5pt]
\tikzstyle{V}=[draw, fill =black, circle, inner sep=0pt, minimum size=1.5pt]
\tikzstyle{bV}=[draw, fill =black, circle, inner sep=0pt, minimum size=3.5pt]
\tikzstyle{cV}=[draw, fill =white, circle, inner sep=0pt, minimum size=3.5pt]
\tikzstyle{BoxArr}=[xscale = .2, yscale=-.2]
\tikzstyle{WhiteCircle}=[draw,circle,white, fill=white]

\newcommand{\TikZ}[1]{
\begin{matrix}\begin{tikzpicture}#1\end{tikzpicture}\end{matrix}
}

\def\ShiftY{\pgftransformyshift}

\newcommand\OneTLNode[2]{ 
\TikZ{[xscale=.75, yscale=-.75]
	\node[#1] at (0,0) {\scriptsize$#2$};
	}}

\def\Over[#1,#2][#3,#4]{ %1,2=start position; 3,4=end position
	\draw[style=over]   (#1,#2) .. controls ++(0,#4*.5-#2*.5) and ++(0,-#4*.5+#2*.5) .. (#3,#4);}
\def\Under[#1,#2][#3,#4]{ %1,2=start position; 3,4=end position
	\draw  (#1,#2) .. controls ++(0,#4*.5-#2*.5) and ++(0,-#4*.5+#2*.5) .. (#3,#4);}
\def\Cross[#1,#2][#3,#4]{%Mimic over, under follows
	\Under[#3,#2][#1,#4]\Over[#1,#2][#3,#4]}

\def\Ez[#1]{\draw [over, bend left=75] (1,#1+1) to (1-.4,#1+.7)  (1-.4,#1+.3)  to (1,#1) ;
		\draw[densely dotted]  (1-.4,#1)--(1-.4,#1+1) ; \node[V] at (1-.4,#1+.3){}; \node[V] at  (1-.4,#1+.7){};}
\def\Ek[#1][#2]{\draw [over, bend right=75] (#2,#1+1) to (#2+.4,#1+.7)  (#2+.4,#1+.3)  to (#2,#1) ;
		\draw[densely dotted]  (#2+.4,#1)--(#2+.4,#1+1) ; \node[V] at (#2+.4,#1+.3){}; \node[V] at  (#2+.4,#1+.7){};}

\def\Tops[#1][#2][#3]{%1=pole locations, 2=top, 3=k 
	\foreach\x in {#1}{
		\draw (\x+.15,#2) -- (\x+.15,#2+.1) (\x-.15,#2) -- (\x-.15,#2+.1) ;
		\draw (\x+.15,#2+.1) arc (0:360:1.5mm and .75mm);}
	\foreach \x in {1,...,#3} {\draw (\x,#2)  to (\x,#2+.05); \node[V] at (\x,#2+.05){};}
	}
\def\Bottoms[#1][#2][#3]{%1=pole locations, 2 = bottom, 3=top 
	\foreach\x in {#1}{
		\draw (\x+.15,#2) -- (\x+.15,#2-.1) (\x-.15,#2) -- (\x-.15,#2-.1) ;
		\draw (\x+.15,#2-.1) arc (0:-180:1.5mm and .75mm);}
	\foreach \x in {1,...,#3} {\draw (\x,#2)  to (\x,#2-.05); \node[V] at (\x,#2-.05){};}
	}
\def\Caps[#1][#2,#3][#4]{%1=pole locations, 2 = bottom, 3=top, 4=k 
	\Tops[#1][#3][#4]
	\Bottoms[#1][#2][#4]
	}
\def\Pole[#1][#2,#3]{%1=horizontal location, 2 = bottom, 3=top
	\shade[left color=white,right color=white] (#1+.15,#2) rectangle (#1-.15,#3);
	\draw[over] (#1+.15,#2) to (#1+.15,#3) (#1-.15,#2) to (#1-.15,#3) ;}
\def\Label[#1,#2][#3][#4]{%1,2 = top/bot position, 3=i, 4=label
	\node[above] at (#3,#2+.1) {#4};
	\node[below] at (#3,#1-.1) {#4};		}
\def\Nodes[#1][#2]{
	 \foreach \x in {1,...,#2} {\node[V] at (\x,#1){};	}
	}
\def\PoleCaps[#1][#2,#3]{%1=pole location, 2 = bottom, 3=top, 4=k 
	\foreach\x in {#1}{
		\draw (\x+.15,#2) -- (\x+.15,#2-.1) (\x-.15,#2) -- (\x-.15,#2-.1) ;
		\draw (\x+.15,#2-.1) arc (0:-180:1.5mm and .75mm);}
	\foreach\x in {#1}{
		\draw (\x+.15,#3) -- (\x+.15,#3+.1) (\x-.15,#3) -- (\x-.15,#3+.1) ;
		\draw (\x+.15,#3+.1) arc (0:360:1.5mm and .75mm);}
	}
\def\PoleTwist[#1,#2]{%1 = bottom, 2=top
	\foreach \x/\y in {-1/1L, -.7/1R, 0/2L, .3/2R}{\coordinate(T\y) at (\x,#2); \coordinate(B\y) at (\x,#1);}
	\draw[thin] (B1R) .. controls ++(0,#2*.5-#1*.5-.1) and ++(0,-#2*.5+#1*.5-.1) ..  (T2R)
			(B1L)   .. controls ++(0,#2*.5-#1*.5+.1) and ++(0,-#2*.5+#1*.5+.1) ..    (T2L) ;
	\draw[line width=2pt, white]
			(.15,#1)  .. controls +(0,#2*.5-#1*.5) and +(0,-#2*.5+#1*.5) ..   (-.85,#2) ;
	\draw[thin,over] 
		(B2R) .. controls ++(0,#2*.5-#1*.5+.1) and ++(0,-#2*.5+#1*.5+.1) ..  (T1R) 
			(B2L)  .. controls +(0,#2*.5-#1*.5-.1) and +(0,-#2*.5+#1*.5-.1) ..   (T1L) ;
			}

\def\SymPolesCaps[#1,#2][#3]{%1 = Vertical position, 2=k
	\draw (.3,#1) -- (.3,#1-.1) (.15,#1) -- (.15,#1-.1) ;
	\draw (.3,#1-.1) arc (0:-180:2pt and 1.5pt);
	\draw (#3+.7,#1) -- (#3+.7,#1-.1) (#3+.85,#1) -- (#3+.85,#1-.1) ;
	\draw (#3+.85,#1-.1)  arc (0:-180:2pt and 1.5pt);
	\draw (.3,#2) -- (.3,#2+.1) (.15,#2) -- (.15,#2+.1) ;
	\draw (.3,#2+.1) arc (0:360:2pt and 1.5pt);
	\draw (#3+.7,#2) -- (#3+.7,#2+.1) (#3+.85,#2) -- (#3+.85,#2+.1) ;
	\draw (#3+.85,#2+.1) arc (0:360:2pt and 1.5pt);}

\newcounter{r}
\newcommand\Part[1]{
        \setcounter{r}{1}
	 \foreach \x in {#1}{
 	{\ifnum\value{r}=1
		\draw (0,\value{r}-1)--(\x,\value{r}-1); 
		\fi}
	\draw (0,\value{r}) to (\x,\value{r});
   	\foreach \y in {0, ..., \x} {\draw (\y,\value{r})--(\y,\value{r}-1);}
	\addtocounter{r}{1}
 }}
 
  \def\PartUNIT{.175}
  \newcommand\ColTab[1]{
\TikZ{[xscale=.3, yscale=-.3] 
        \setcounter{r}{1}
	 \foreach \x in {#1}{
 	{\ifnum\value{r}=1
		\draw (0,0)--(1,0); 
		\fi}
	\draw (0,\value{r}-1) to  (0,\value{r}) to  (1,\value{r}) to  (1,\value{r}-1) ;
   	\node at (.5, \value{r}-.5) {\tiny $\x$};
	\addtocounter{r}{1}}}
 }

\newcommand{\BOX}[2]{
	\draw[fill=white] (#1) to ++(1,0) to ++(0, 1) to ++(-1, 0) to ++(0, -1);
	\begin{scope}[xshift=.5cm, yshift=.5cm]\node at (#1) {\tiny #2};\end{scope}
	}

\newcommand{\bDOT}[1]{
\filldraw [red] (#1) circle (6pt);
}

\newcommand\Cont[2]{\node at (#1) {\small #2};}

\tikzstyle{sBoxArr}=[xscale = \PartUNIT, yscale=-\PartUNIT]

\title{Calibrated representations of two boundary Temperley-Lieb algebras}

\author{Zajj Daugherty}
\address{Department of Mathematics, 
The City College of New York, 
NAC 8/133,
New York, NY 10031}
\author{Arun Ram}
\address{Department of Mathematics and Statistics, University of Melbourne, Parkville VIC 3010 Australia}
\usepackage[foot]{amsaddr}
\email{zdaugherty@gc.cuny.edu, aram@unimelb.edu.au}

\date{\today}

\begin{document}
\maketitle

\begin{center}
{\sl In memory of a friend and an inspiration, Vladimir Rittenberg 1934-2018}
\end{center}

\medskip

\begin{abstract}
The two boundary Temperley-Lieb algebra $TL_k$ arises in the transfer matrix formulation of lattice models in Statistical Mechanics, in particular in the introduction of  integrable boundary terms to the six-vertex model. In this paper, we classify and study the calibrated representations---those for which all the Murphy elements (integrals) are simultaneously diagonalizable---which, in turn, corresponds to diagonalizing the transfer matrix in the associated model.  Our approach is founded upon the realization of $TL_k$ as a quotient of the type $C_k$ affine Hecke algebra $H_k$. In previous work, we studied this Hecke algebra via its presentation by braid diagrams, tensor space operators, and related combinatorial constructions. That work is directly applied herein to give a combinatorial classification and construction of all irreducible calibrated $TL_k$-modules and explain how these modules also arise from a Schur-Weyl duality with the quantum group $U_q\fgl_2$.\end{abstract}

\tableofcontents

%\footnote{\emph{Keywords:} Temperley-Lieb, Hecke, representations, combinatorics}
%\footnote{AMS Subject Classifications: (20C08, 17B37, 05E10)}

\section{Introduction}

The paper \cite{DR} studied the calibrated representations of affine Hecke algebras of type C with unequal parameters and developed their combinatorics and their role in Schur-Weyl duality.  This paper applies that information to the study of two boundary  Temperley-Lieb algebras.  The two boundary Temperley-Lieb algebras appear in statistical mechanics for analysis of spin chains with generalized boundary conditions \cite{GP,GNPR}. T he spectrum of the Hamiltonian for these spin chains with  boundaries can be determined via the representation theory of the two boundary Temperley-Lieb algebras. In fact, the need to understand the representation theory of the two boundary Temperley-Lieb algebra better was a primary motivation for our preceding papers \cite{Da, DR} on two boundary Hecke algebras.

In the first part of this paper, Section 2, we review the definition and structure of the two boundary Hecke algebra $H_k$ (the affine Hecke algebra of type C with unequal parameters). Following this brief review we carefully analyze  certain idempotents which, as we prove in Theorem \ref{idempquotient}, generate the ideal that one must quotient by to obtain the two boundary Temperley-Lieb algebra from the two boundary Hecke algebra.  It is the expression of these idempotents in terms of the intertwiner presentation of $H_k$  (see Proposition \ref{idempexpansion}) that will eventually provide understanding  of the weights that can appear in two boundary Temperley-Lieb modules (the possible eigenvalues of the ``Murphy elements" $W_i$---see equation \eqref{BraidMurphy} and \S\ref{Bernsteingens}).

In Section 3 we define the two boundary Temperley-Lieb algebra  (or \emph{symplectic blob algebra}) $TL_k$  following \cite{GN, GMP07, GMP08, GMP12,  Re12, KMP16, GMP17},  and review the diagram algebra calculus for these algebras. Part of our contribution is to extend this calculus to make its connection to the diagrammatic calculus of the Hecke algebra $H_k$ via braids.    In Theorem \ref{IIIexpansion} we use these diagrammatics to give a proof of a result of \cite{GN} that provides an expansion of a certain central element of $H_k$ inside $TL_k$. Using the Hecke algebra point of view, this result enables us to understand that  the center of $TL_k$ is a polynomial ring in one variable $Z(TL_k) = \CC[Z]$,  and that $TL_k$ is of finite rank over this center. 
In retrospect, the algebra $H_k$ has a similar structure and so perhaps this should not be surprising but, nonetheless, it is pleasant to see it come out in such a vivid and explicit form. 

We have used a different normalization of the parameters of the two boundary Hecke and Temperley-Lieb algebra from those used in \cite{GN,GMP12}.  Our normalization will be helpful, for example, for future applications of these algebras to the theory of  Macdonald polynomials and to the study of the exotic nilpotent cone.  In both of these cases  the affine Hecke algebra of type $C_n$ plays an important role:  the Koornwinder polynomials are the Macdonald polynomials for type $(C_n^\vee, C_n)$ \cite{M03},  and the K-theory of the Steinberg variety of the exotic nilpotent cone provides a geometric construction of the representations of the two boundary Hecke and Temperley-Lieb algebras at unequal parameters (see \cite{Kat}).

The calibrated representations are the irreducible representations of the two boundary Hecke algebra for which a large family of commuting operators (integrals, or Murphy elements) have a simple (joint) spectrum.  This property makes these representations particularly attractive, and the detailed combinatorics of these representations has been worked out in \cite{DR}. In Section 4 we use the detailed analysis of the idempotents done in Section 2 to determine exactly which calibrated irreducible representations of the two boundary Hecke  algebra are representations of the two boundary Temperley-Lieb algebra  (Theorem \ref{TLconfigurations}).  In consequence, we obtain a full classification of the calibrated irreducible representations of the two boundary  Temperley-Lieb algebras.

As explained in \cite{DR}, there is a Schur-Weyl type duality between the two boundary  Hecke algebra and the quantum group $U_q\fgl_n$.  The classical  Schur-Weyl duality between $U_q\fgl_n$ and the finite Hecke algebra of type A becomes a  Schur-Weyl duality for the finite Temperley-Lieb algebra when $n=2$. In Theorem \ref{thm:modules-in-tensor-space} we show that at $n=2$  the Schur-Weyl duality of \cite{DR} gives a Schur-Weyl duality for the two boundary Temperley-Lieb algebra.  This method (coming from R-matrices for the quantum group $U_q\fgl_2$) provides many many irreducible calibrated representations of the two boundary Temperley-Lieb algebra $TL_k$.  Using our results from Section 4, we determine exactly which irreducible calibrated representations of $TL_k$ occur in the Schur-Weyl duality context.

The seeds of this work were sown in a conversation between Pavel Pyatov, Arun Ram  and Vladimir Rittenberg at the Max Planck Institut in Bonn in 2006.    Vladimir was the leader and provided the inspiration by introducing us to spin chains with boundaries.  The seed has now grown from a concept into fully formed and fruitful mathematics.   We thank all the institutions which have supported our work on this paper,  particularly the University of Melbourne,  the  Australian Research Council (grants DP1201001942 and DP130100674), the National Science Foundation (grant DMS-1162010), the Simons Foundation (grant  \#586728), ICERM (Institute for Computational and Experimental Research in Mathematics, 2013 semester on Automorphic Forms, Combinatorial Representation Theory and Multiple Dirichlet Series) and the Max Planck Institut in Bonn.

\section{The two boundary Hecke algebra $H_k$}

The two boundary Hecke algebra is often called the affine Hecke algebra of type $(C, C^\vee)$.  In this  section we review the definitions of $H_k^{\mathrm{ext}}$ following our previous paper \cite{DR}.  In particular, we will need the basic diagrammatics and the ``Bernstein'' presentation with a Laurent polynomial ring $\CC[W_1^{\pm}, \ldots, W_k^{\pm}]$ and intertwiners $\tau_1, \ldots, \tau_k$.  After this review we define the idempotent elements $p_i^{(1^3)}$, $p_0^{(\emptyset, 1^2)}$, $p_0^{(1^2, \emptyset)}$, $p_{0^\vee}^{(\emptyset, 1^2)}$, $p_{0^\vee}^{(1^2,\emptyset)}$, which we will need to quotient by in order to obtain the two boundary Temperley-Lieb algebra.  We derive expressions of these elements in terms of  the different choices of generators: the braid generators $T_i$, the cap/cup generators $e_i$, and the intertwiner generators $\tau_i$ and $W_j$.

\subsection{Graph notation for braid relations}
For generators $g, h$, encode relations graphically by 
\begin{equation}\label{braidlengths}
\begin{array}{cl}
\begin{tikzpicture}
	\draw[fill=white] (0,0) circle (2.5pt) node[above=1pt] {\small $\phantom{h}g\phantom{h}$};
	\draw[fill=white] (1,0) circle (2.5pt) node[above=1pt] {\small $\phantom{g}h\phantom{g}$}; 
\end{tikzpicture}
&\hbox{means $gh=hg$,} 
\\ \\
\begin{tikzpicture}
	\draw (0,0)--(1,0);
	\draw[fill=white] (0,0) circle (2.5pt) node[above=1pt] {\small $\phantom{h}g\phantom{h}$};
	\draw[fill=white] (1,0) circle (2.5pt) node[above=1pt] {\small $\phantom{g}h\phantom{g}$}; 
\end{tikzpicture}
&\hbox{means $ghg=hgh$, and} \\ \\
\begin{tikzpicture}
	\draw[double distance = 2pt] (0,0)--(1,0);
	\draw[fill=white] (0,0) circle (2.5pt) node[above=1pt] {\small $\phantom{h}g\phantom{h}$};
	\draw[fill=white] (1,0) circle (2.5pt) node[above=1pt] {\small $\phantom{g}h\phantom{g}$}; 
\end{tikzpicture}
&\hbox{means $ghgh=hghg$.} 
\end{array}
\end{equation}
For example, the group of signed permutations,
\begin{equation}
\label{eq:Weyl-group}
\cW_0 = \left\{ 
\begin{matrix}
\hbox{bijections $w\colon \{-k, \ldots, -1, 1, \ldots, k\} \to \{ -k, \ldots, -1, 1, \ldots, k\}$} \\
\hbox{such that $w(-i) = -w(i)$ for $i=1, \ldots, k$}
\end{matrix} \right\},
\end{equation}
has a presentation by generators $s_0, s_1,\ldots, s_{k-1},$ with relations
\begin{equation}
\begin{tikzpicture}
	\foreach \x in {0,1, 2, 4,5}{
		\draw (\x,0) circle (2.5pt);
		\node(\x) at (\x,0){};
		}
	\foreach \x in {1, 2}{
		\node[label=above:$s_{\x}$] at (\x,0){};
		\node[label=above:$s_{k-\x}$] at (6-\x,0){};}
	\node[label=above:$s_0$] at (0){};
	\draw[double distance = 2pt] (0)--(1);
	\draw (1)--(2) (4)--(5);
	\draw[dashed] (2) to (4);
\end{tikzpicture}
\qquad\hbox{and}\qquad s_i^2=1\ \hbox{for $i=0,1,2,\ldots, k-1$.}
\label{W0defn}
\end{equation}

\subsection{The two boundary braid group}

The \emph{two boundary braid group} is the group $\cB_k$ generated by 
$\bar{T}_0, \bar{T}_1, \ldots, \bar{T}_k$,
with relations
\begin{equation}\label{Bdefrels}
\TikZ{
	\foreach \x in {0, 1, 2, 4,5,6}{
		\draw (\x,0) circle (2.5pt);
		\node(\x) at (\x,0){};
		}
	\foreach \x in {0, 1, 2}
		{\node[label=above:$\bar{T}_\x$] at (\x){};}
	\node[label=above:$\bar{T}_{k-2}$] at (4){};
	\node[label=above:$\bar{T}_{k-1}$] at (5){};
	\node[label=above:$\bar{T}_{k}$] at (6){};
	\draw[double distance = 2pt] (0)--(1) (5)--(6);
	\draw (1)--(2) (4)--(5);
	\draw[dashed] (2) to (4);
}\ .
\end{equation}
\noindent Pictorially, the generators of $\cB_k$ are identified with the braid diagrams
$$
{\def\TOP{2}\def\K{6}
\bar{T}_k=
\TikZ{[scale=.5]
\Pole[.15][0,2]
\Under[\K,0][\K+1.3,1]
\Pole[\K+.85][0,1][\K]
\Pole[\K+.85][1,2][\K]
\Over[\K+1.3,1][\K,2]
 \foreach \x in {1,...,5} {
	 \draw[thin] (\x,0) -- (\x,\TOP);
	 }
\Caps[.15,\K+.85][0,\TOP][\K]
},
	\qquad 
\bar{T}_0=
\TikZ{[scale=.5]
	\Pole[\K+.85][0,2][\K]
	\Pole[.15][0,1]
	\Over[1,0][-.3,1]
	\Under[-.3,1][1,2]
	\Pole[.15][1,2]
	 \foreach \x in {2,...,\K} {
	 \draw[thin] (\x,0) -- (\x,\TOP);
	 }
\Caps[.15,\K+.85][0,\TOP][\K]
},\qquad  \text{and}
}$$
\begin{equation}\label{LRpics}
{\def\TOP{2} \def\K{6}
\bar{T}_i=
\TikZ{[scale=.5]
	\Pole[\K+.85][0,2][\K]
	\Pole[.15][0,2]
	\Under[3,0][4,2]
	\Over[4,0][3,2]
	 \foreach \x in {1,2,5,\K} {
		 \draw[thin] (\x,0) -- (\x,\TOP);
		 }
	\Caps[.15,\K+.85][0,\TOP][\K]
	\Label[0,\TOP][3][\footnotesize $i$]
	\Label[0,\TOP][4][\footnotesize$i$+1]
}
\qquad \text{for $i=1, \dots, k-1$,}
}
\end{equation}
and the multiplication of braid diagrams is given by placing one diagram on top of another (multiplying generators left-to-right corresponds to stacking diagrams top-to-bottom).

In some applications (notably to the Schur-Weyl duality of \cite[\S5]{DR}), 
it is useful to move the rightmost  pole to the left by conjugating by the diagram
\begin{equation}\label{eq:sigma}
\sigma = 
\TikZ{[scale=.5]
	\draw (-.7,1) .. controls (-.7,.15) .. (0,.15) -- (6,.15) .. controls (7,.15) .. (7,-1);
	\draw (-1,1) .. controls (-1,-.15) .. (0,-.15)-- (6,-.15) .. controls (7-.3,-.15) .. (7-.3,-1);
\Pole[.15][-1,1]
 \foreach \x in {1,...,6} {\draw[style=over] (\x,-1) -- (\x,1);}
\Tops[.15, -.85][1][6]
\Bottoms[.15, 6+.85][-1][6]
}\ .
\end{equation}
Define
\begin{equation}\label{DefnTiY1}
{\def\TOP{2}\def\K{6}
T_i= \sigma \bar{T}_i \sigma^{-1}=
\TikZ{[scale=.5]
	\Pole[-.85][0,2]
	\Pole[.15][0,2]
	\Under[3,0][4,2]
	\Over[4,0][3,2]
	 \foreach \x in {1,2,5,\K} {
		 \draw[thin] (\x,0) -- (\x,\TOP);
		 }
	\Caps[.15,-.85][0,\TOP][\K]
	\Label[0,\TOP][3][\footnotesize $i$]
	\Label[0,\TOP][4][\footnotesize$i$+1]
}\ , \qquad 
Y_1= \sigma \bar{T}_0 \sigma^{-1} =
\TikZ{[scale=.5]
	\Pole[-.85][0,2][\K]
	\Pole[.15][0,1]
	\Over[1,0][-.3,1]
	\Under[-.3,1][1,2]
	\Pole[.15][1,2]
	 \foreach \x in {2,...,\K} {
	 \draw[thin] (\x,0) -- (\x,\TOP);
	 }
	\Caps[.15,-.85][0,\TOP][\K]
}\ ,}
\end{equation}
and
\begin{equation}\label{DefnX1}
{\def\TOP{2} \def\K{6}
X_1 =   T_1^{-1} T_2^{-1} \cdots T^{-1}_{k-1} \sigma \bar{T}_k \sigma^{-1} T_{k-1} \cdots T_1 =   
\TikZ{[scale=.5]
	\Pole[-.85][0,1]
	\Over[1,0][-1.3,1]
	\Under[-1.3,1][1,2]
	\Pole[-.85][1,2]
	\Pole[.15][0,2]
	 \foreach \x in {2,...,\K}  {
		 \draw[thin] (\x,0) -- (\x,\TOP);
		 }
	\Caps[.15,-.85][0,\TOP][\K]
	}\ .
}\end{equation}
\noindent 
Define 
\begin{equation}\label{BraidMurphy}{
\def\TOP{2}\def\K{6}
Z_1=X_1Y_1\quad\hbox{and}\quad
		 Z_i = T_{i-1} T_{i-2} \cdots T_1 X_1 Y_1 T_1 \cdots T_{i-1} =
\TikZ{[scale=.5]
		\Pole[-.85][0,1]
		\Pole[.15][0,1]
		 \foreach \x in {1,2} {\draw[thin] (\x,0) -- (\x,1);}
		\Over[3,0][-1.3,1]
		\Under[-1.3,1][3,2]
		\Pole[-.85][1,2]
		\Pole[.15][1,2]
		\foreach \x in {1,2} { \draw[thin, style=over] (\x,1) -- (\x,2); }
		\foreach \x in {4,...,\K} {\draw[thin] (\x,0) -- (\x,\TOP);}
		\Caps[.15,-.85][0,\TOP][\K]
		\Label[0,\TOP][3][{\footnotesize $i$}]
}\ ,}\end{equation}
for $i=2, \ldots, k$.
Let
{\def\TOP{1.75} \def\K{6}$$P= \TikZ{[scale=.5]
		\PoleTwist[0,\TOP*.5]\PoleTwist[\TOP*.5, \TOP]
		\foreach \x in {1,...,\K} {\draw[thin] (\x,0) -- (\x,\TOP);}
		\Caps[-.85,.15][0,\TOP][\K]}\ .$$}
The \emph{extended affine braid group} is the group $\cB_k^{\mathrm{ext}}$ 
generated by $\cB_k$ and $P$ with the additional relations
\begin{align}
PX_1P^{-1} = Z_1^{-1}X_1Z_1, \qquad
PY_1P^{-1} = Z_1^{-1}Y_1Z_1, 
\label{Pcomm1} %\tag{braid c5}
\\
PZ_1P^{-1} = Z_1,\qquad \text{ and } \qquad PT_iP^{-1} = T_i\ \hbox{for $i=1, \ldots, k-1$.}
\label{Pcomm2} %\tag{braid c6}
\end{align}
The element 
\begin{equation}\label{eq:Z0iscentral}
Z_0 = P Z_1 \cdots Z_k \quad\hbox{is central in $\cB_k^\mathrm{ext}$}
\tag{c0}
\end{equation}
since the group $\cB_k^{\mathrm{ext}}$ is a subgroup of the braid group on $k+2$ strands,
and
$Z_0$ is the generator of the center of the braid group on $k+2$ strands (see \cite[Theorem 4.2]{GM}).
So
\begin{equation}
\hbox{if}\ \cD = \{ Z_0^{j}\ |\ j\in \ZZ\}
\qquad\hbox{then}\qquad
\cB_k^{\mathrm{ext}} = \cD \times \cB_k,
\quad\hbox{with $\cD \cong \ZZ$.}
\label{bdgpproduct}
\end{equation}

\subsection{The extended affine Hecke algebra $H_k^{\mathrm{ext}}$ of type $C_k$}

Fix $a_1, a_2, b_1, b_2, t ^{\frac12}\in \CC^\times$ and let
\begin{equation}\label{eq:ab-to-t0tk}
t_k^{\frac12} =  a_1^{\frac12}(-a_2)^{-\frac12},
\qquad
t_0^{\frac12} = b_1^{\frac12}(-b_2)^{-\frac12}.
\end{equation}
The \emph{extended two boundary Hecke algebra $H_k^\ext$ with parameters $t^{\frac12}$,
$t_0^{\frac12}$ and $t_k^{\frac12}$} is the quotient of $\cB_k^\ext$ by the relations 
\begin{equation}
(X_1 - a_1)(X_1-a_2) = 0,
\quad
(Y_1 - b_1)(Y_1-b_2) = 0,
\quad\hbox{and}\quad 
(T_i - t^{\half})(T_i+ t^{-\half}) = 0,
\label{Heckedefn} \tag{H}
\end{equation}
for $i = 1, \dots, k-1$. 
Let
\begin{equation}
T_0 = b_1^{-\frac12}(-b_2)^{-\frac12} Y_1,
\qquad
T_k = a_1^{-\frac12}(-a_2)^{-\frac12} T_{k-1}\cdots T_2T_1X_1T_1^{-1}T_2^{-1}\cdots T_{k-1}^{-1}.
\end{equation}
Then $
%\begin{equation*}
\TikZ{
	\foreach \x in {0, 1, 2, 4,5,6}{
		\draw (\x,0) circle (2.5pt);
		\node(\x) at (\x,0){};
		}
	\foreach \x in {0, 1, 2}
		{\node[label=above:$T_\x$] at (\x){};}
	\node[label=above:$T_{k-2}$] at (4){};
	\node[label=above:$T_{k-1}$] at (5){};
	\node[label=above:$T_{k}$] at (6){};
	\draw[double distance = 2pt] (0)--(1) (5)--(6);
	\draw (1)--(2) (4)--(5);
	\draw[dashed] (2) to (4);
}
$
and
\begin{equation}
(T_0- t_0^{\frac12})(T_0 + t_0^{-\frac12}) = 0,
\quad 
(T_i - t^{\frac12})(T_i +  t^{-\frac12})=0,
\quad %\hbox{and}\quad 
(T_k- t_k^{\frac12})(T_k + t_k^{-\frac12}) = 0,
\label{DRTBHeckerelations}
\end{equation}
for $i\in \{1, \ldots, k-1\}$. 

Let $a, a_0, a_k \in \CC^\times$ 
and define
\begin{equation}
a_0e_0 = T_0-t_0^{\frac12},
\qquad
ae_i = T_i-t^{\frac12},
\qquad
a_ke_k = T_k-t_k^{\frac12},
\label{edefin}
\end{equation}
for $i\in \{1, \ldots, k-1\}$.
The relations in \eqref{DRTBHeckerelations}
are equivalent to
\begin{equation}
T_0e_0 = -t_0^{-\frac12}e_0,
\qquad
T_i e_i = -t^{-\frac12}e_i,
\qquad
T_k e_k = -t_k^{-\frac12}e_k,
\label{signrep}
\end{equation}
and to
\begin{equation}
e_0^2 = \frac{-(t_0^{\frac12}+t_0^{-\frac12})}{a_0} e_0,
\qquad
e_i^2 = \frac{-(t^{\frac12}+t^{-\frac12})}{a}e_i,
\qquad
e_k^2 = \frac{-(t_k^{\frac12}+t_k^{-\frac12})}{a_k} e_k,
\label{eisquared}
\end{equation}
for $i \in \{1, \dots, k-1\}$. 

\begin{remark}
\label{bdequivi}
For $i\in \{1, \ldots, k-2\}$, using $T_i = ae_i+t^{\frac12}$ to expand
$T_iT_{i+1}T_i$ and $T_{i+1}T_iT_{i+1}$ in terms of the $e_i$ shows that 
in the presence of the relations \eqref{quadraticT0andTi},
$$
T_iT_{i+1}T_i = T_{i+1}T_iT_{i+1}
\quad\hbox{is equivalent to}\quad
a^3 e_ie_{i+1}e_i - a e_i = a^3 e_{i+1}e_ie_{i+1} - a e_{i+1}.
%\label{bdequivi}
$$
Similarly, 
$T_0T_1T_0T_1=T_1T_0T_1T_0$ is equivalent to
$$
a_0^2a^2 e_0e_1e_0e_1-a_0 a (t_0^{-\frac12}t^{\frac12}+t_0^{\frac12}t^{-\frac12} )e_0e_1
=a_0^2a^2 e_1e_0e_1e_0
- a_0 a (t_0^{-\frac12}t^{\frac12}+t_0^{\frac12}t^{-\frac12} )e_1e_0.
$$
In the case that
$a_0^2a^2 =a_0 a(t_0^{-\frac12}t^{\frac12}+t_0^{\frac12}t^{-\frac12})$
then 
$$T_0T_1T_0T_1=T_1T_0T_1T_0
\quad\hbox{is equivalent to}\quad
e_0e_1e_0e_1-e_0e_1 = e_1e_0e_1e_0-e_1e_0.
$$
In the case that $a^3=a$ then
$$T_iT_{i+1}T_i = T_{i+1}T_iT_{i+1}
\quad\hbox{is equivalent to}\quad
e_ie_{i+1}e_i - e_i = e_{i+1}e_ie_{i+1}-e_{i+1}.
$$
This is the explanation for why the favorite choices of $a$, $a_0$ and $a_k$ satisfy
$$a = \pm1,
\qquad
a_0a = t_0^{-\frac12}t^{\frac12}+t_0^{\frac12}t^{-\frac12} = [\![t_0t^{-1}]\!]
\qquad\hbox{and}\qquad
a_ka = t_k^{-\frac12}t^{\frac12}+t_k^{\frac12}t^{-\frac12} = [\![t_kt^{-1}]\!],
$$
where we use the notation
\begin{equation}
[\![t^s]\!] = (t^{\frac{s}{2}} + t^{-\frac{s}{2}}) 
= \Big(\frac{t^s-t^{-s}}{t^{\frac12}-t^{-\frac12}} \Big)
\Big( \frac{t^{\frac12}-t^{-\frac12}}{t^{\frac{s}{2}}-t^{-\frac{s}{2}}} \Big)
= \frac{[2s]}{[s]}.
\label{notation}
\end{equation}
\end{remark}

\subsection{The Bernstein presentation of $H_k^{\mathrm{ext}}$}\label{Bernsteingens}

The \emph{Murphy elements for $H^{\mathrm{ext}}_k$} are 
$$
W_1 = T_1^{-1} T_2^{-1} \cdots T_{k-1}^{-1} T_k T_{k-1} \cdots T_2 T_1 T_0
\qquad\hbox{and}\qquad
W_j = T_j W_{j-1} T_j, 
$$
for $j\in \{2, \dots, k\}$.  
Let
$$W_0 = PW_1\cdots W_k.$$

\begin{thm}\label{thm:2bdryHeckePresentation} (See, for example, \cite[Theorem 2.2]{DR}.)
Fix $t_0, t_k, t \in \CC^\times$ and use notations for relations as defined in \eqref{braidlengths}. 
The extended affine Hecke algebra $H^\ext_k$ defined in 
\eqref{Heckedefn}
is presented by generators, 
$T_0$, $T_1$, \dots, $T_{k-1}$, $W_0$, $W_1$, \dots, $W_k$ and relations
\begin{equation}
W_0\in Z(H_k^{\mathrm{ext}}), \qquad\qquad
\TikZ{
	\foreach \x in {0,1, 2, 4,5}{
		\draw (\x,0) circle (2.5pt);
		\node (\x) at (\x,0){};
		}
	\foreach \x in {1, 2}{
		\node[label=above:$T_\x$] at (\x,0){};
		\node[label=above:$T_{k-\x}$] at (6-\x,0){};}
	\node[label=above:$T_0$] at (0){};
	\draw[double distance = 2pt] (0)--(1);
	\draw (1)--(2) (4)--(5);
	\draw[dashed] (2) to (4);
\label{PresBraidRelsB1}\tag{B1}
};
\end{equation}
\begin{equation}
W_iW_j = W_j W_i, \qquad\hbox{for $i,j= 0, 1,\ldots, k$;}
\label{WcommuteB2}\tag{B2}
\end{equation}
\begin{equation}
T_0W_j = W_j T_0,\quad\hbox{for $j\ne 1$;}
\label{T0WcommuteB3}\tag{B3}
\end{equation}
\begin{equation}
T_iW_j = W_j T_i\ \hbox{for $i=1, \ldots, k-1$ and $j=1,\ldots, k$ with $j\ne i, i+1$;}
\label{TiWcommuteB4}\tag{B4}
\end{equation}
\begin{equation}
(T_0- t_0^{\frac12})(T_0 + t_0^{-\frac12}) = 0, \ 
\quad \hbox{and}\quad 
(T_i - t^{\frac12})(T_i +  t^{-\frac12})=0 \ \hbox{for $i = 1, \dots, k-1$;}
\label{quadraticT0andTi}\tag{H}
\end{equation}
for $i=1,\ldots, k-1$,  
\begin{equation}
T_iW_i = W_{i+1}T_i + (t^{\frac12} - t^{-\frac12})\frac{W_i - W_{i+1}}{1-W_iW_{i+1}^{-1}},
\qquad
T_iW_{i+1} = W_iT_i +(t^{\frac12}-t^{-\frac12})\frac{W_{i+1}-W_i}{1-W_iW_{i+1}^{-1}}, 
\label{TipastWi}\tag{C1}
\end{equation}
\begin{align}\text{and}\qquad
T_0W_1 &= W_1^{-1}T_0 + \left((t_0^{\frac12}-t_0^{-\frac12})  + (t_k^{\frac12} -t_k^{-\frac12})W_1^{-1}\right) \frac{W_1 - W_1^{-1}}{1-W_1^{-2}}.
\label{T0pastW1}\tag{C2}
\end{align}
\end{thm}

The \emph{two boundary Hecke algebra $H_k$ with parameters $t^{\frac12}$,
$t_0^{\frac12}$ and $t_k^{\frac12}$} is the subalgebra of $H_k^{\mathrm{ext}}$
generated by $T_0, T_1, \ldots, T_k$.  Then
\begin{equation}
H_k^{\mathrm{ext}} = H_k \otimes \CC[W_0^{\pm1}]
\qquad\hbox{as algebras,}
\label{Hksplitting}
\end{equation}
and, 
as proved for example in \cite[Theorem 2.3]{DR}, the element
\begin{equation}
Z = W_1+W_1^{-1}+W_2+W_2^{-1}+\cdots+W_k+W_k^{-1}
\qquad\hbox{is central in $H^{\mathrm{ext}}_k$.}
\label{Zdefn}
\end{equation}

\subsection{The elements $\tau_i$}

Define
\begin{equation}
\tau_0 
= T_0 - \frac{ (t_0^{\frac12} - t_0^{-\frac12}) + (t_k^{\frac12} - t_k^{-\frac12})W_1^{-1} }
{1-W_1^{-2}},
\qquad\hbox{and}\qquad
\tau_i = T_i - \frac{t^{\frac12} - t^{-\frac12} }{1-W_iW_{i+1}^{-1} },
\label{intertwinerdefs}
\end{equation}
for $i\in \{1, \ldots, k-1\}$.  
Evoking the notation of \cite[\S 3]{DR}, reviewed later in \S\ref{sec:Calibrated H reps}, 
let
\begin{equation}
\begin{array}{lll}
f_{2\varepsilon_i} = (1-W_i^{-1})(1+W_i^{-1})=1-W_i^{-2}, \\
f_{\varepsilon_i-r_2}
=(1-t_0^{\frac12}t_k^{\frac12}W_i^{-1}), \quad
&f_{\varepsilon_i-r_1}
=(1+t_0^{\frac12}t_k^{-\frac12}W_i^{-1}), \\
f_{-\varepsilon_i-r_2}
=(1-t_0^{\frac12}t_k^{\frac12}W_i), \quad
&f_{-\varepsilon_i-r_1}
=(1+t_0^{\frac12}t_k^{-\frac12}W_i), \\
f_{\varepsilon_i-\varepsilon_j} = 1 - W_iW_j^{-1}, 
&f_{\varepsilon_i-\varepsilon_j+1} = 1 - tW_iW_j^{-1},
\end{array}
\end{equation}
for $i,j\in \{1, \ldots, k\}$.
Then 
\begin{equation}
a_0e_0 = \tau_0
- t_0^{-\frac12}\frac{f_{\varepsilon_1-r_1} f_{\varepsilon_1-r_2} }{f_{2\varepsilon_1}} %c_{\alpha_0}, 
\qquad\hbox{and}\qquad
a_ie_i = \tau_i
- t^{-\frac12}\frac{f_{\varepsilon_i-\varepsilon_{i+1}+1}}{f_{\varepsilon_i-\varepsilon_{i+1}}}, %c_{\alpha_i}, 
\label{eiintaui}
\end{equation}
and,  as proved in \cite[Proposition 2.4]{DR},
$\begin{tikzpicture}
	\foreach \x in {0,1, 2, 4,5}{
		\draw (\x,0) circle (2.5pt);
		}
	\foreach \x in {1, 2}{
		\node[label=above:$\tau_{\x}$] at (\x,0){};
		\node[label=above:$\tau_{k-\x}$] at (6-\x,0){};}
	\node[label=above:$\tau_0$] at (0){};
	\draw[double distance = 2pt] (0)--(1);
	\draw (1)--(2) (4)--(5);
	\draw[dashed] (2) to (4);
\end{tikzpicture}
$
%\label{taubraidrels}
and
\begin{equation*}
\begin{array}{ccc}
\displaystyle{ 
\tau_0^2 = W_1^{-2}t_0^{-1} 
\frac{f_{\varepsilon_1-r_1}f_{-\varepsilon_1-r_1}f_{\varepsilon_1-r_2}f_{-\varepsilon_1-r_2}}
{f_{2\varepsilon_1}^2}  ,
}%c_{\alpha_0}c_{-\alpha_0},
&W_1\tau_0 = \tau_0W_1^{-1},
&W_r \tau_0 = \tau_0 W_r, \\ \\
\displaystyle{
\tau_i^2 = t^{-1} 
\frac{f_{\varepsilon_i-\varepsilon_{i+1}+1}f_{\varepsilon_{i+1}-\varepsilon_i+1}}
{f_{\varepsilon_i-\varepsilon_{i+1}}f_{\varepsilon_{i+1}-\varepsilon_i}} ,
}%c_{\alpha_i}c_{-\alpha_i},
&W_i \tau_i = \tau_i W_{i+1},\quad
W_{i+1} \tau_i = \tau_i W_i, 
&W_j\tau_i = \tau_i W_j, 
\end{array} 
\end{equation*}
for $r, j\in \{1, \ldots, k\}$ with $r\ne 1$ and $j\ne i,i+1$.

\subsection{The elements $p_i^{(1^3)}$, $p_0^{(\emptyset, 1^2)}$ and $p_0^{(1^2,\emptyset)}$}

Fix $i\in \{1, \ldots, k-2\}$.
Let
\begin{align*}
HS_3 &\hbox{\quad be the subalgebra of $H^{\mathrm{ext}}_k$ generated by $T_i$ and $T_{i+1}$, and let} \\
HB_2 &\hbox{\quad be the subalgebra of $H^{\mathrm{ext}}_k$ generated by $T_0$
and $T_1$.}
\end{align*}
The idempotent $p_i^{(1^3)}$ in $HS_3$ 
and the idempotents $p_0^{(\emptyset, 1^2)}$ and $p_0^{(1^2, \emptyset)}$
in $HB_2$ are uniquely determined by the equations
\begin{equation}
(p_i^{(1^3)})^2=p_i^{(1^3)}, \qquad
(p_0^{(\emptyset, 1^2)})^2=p_0^{(\emptyset, 1^2)}
\qquad
(p_0^{(1^2, \emptyset)})^2=p_0^{(1^2, \emptyset)},
\label{idempconds}
\end{equation}
and
\begin{equation}
\begin{array}{lll}
T_i p_i^{(1^3)} = -t^{-\frac12}p_i^{(1^3)},
&\quad &T_{i+1} p_i^{(1^3)} = -t^{-\frac12}p_i^{(1^3)}, \\
T_0 p_0^{(\emptyset, 1^2)} = -t_0^{-\frac12} p_0^{(\emptyset, 1^2)}
&\quad &T_1 p_0^{(\emptyset, 1^2)} = -t^{-\frac12} p_0^{(\emptyset, 1^2)}, \\
T_0 p_0^{(1^2, \emptyset)} = t_0^{\frac12} p_0^{(1^2, \emptyset)},
&\quad &T_1 p_0^{(1^2, \emptyset)} = -t^{-\frac12} p_0^{(1^2, \emptyset)}.
\end{array}
\label{piconds}
\end{equation}
The conditions in \eqref{piconds} are equivalent to 
\begin{equation}
\begin{array}{lll}
ae_i p_i^{(1^3)} = -(t^{\frac12}+t^{-\frac12}) p_i^{(1^3)},
&\quad &ae_{i+1} p_i^{(1^3)} = -(t^{\frac12}+t^{-\frac12}) p_i^{(1^3)}, \\
a_0e_0 p_0^{(\emptyset, 1^2)} = -(t_0^{\frac12}+t_0^{-\frac12}) p_0^{(\emptyset, 1^2)}, 
&\quad &ae_1 p_0^{(\emptyset, 1^2)} = -(t^{\frac12}+t^{-\frac12}) p_0^{(\emptyset, 1^2)}, \\
a_0e_0 p_0^{(1^2, \emptyset)} = 0,
&\quad &ae_1 p_0^{(1^2, \emptyset)} = -(t^{\frac12}+t^{-\frac12}) p_0^{(1^2, \emptyset)}.
\end{array}
\label{pieiconds}
\end{equation}

\begin{prop} \label{idempexpansion} Let $p_i^{(1^3)}$, $p_0^{(\emptyset, 1^2)}$ and $p_0^{(1^2,\emptyset)}$ be as defined in \eqref{idempconds} and \eqref{piconds}
and let
$$N = t^{-\frac12}(1+t)(1+t+t^2)
\quad\hbox{and}\quad
N_0 =N_0'=t_0^{-1}t^{-1}(1+t_0)(1+t)(1+t_0t).
$$
Then the expansions of these idempotents in terms of the three favored generating sets is given by
\begin{align*}
N p_i^{(1^3)}
&= T_iT_{i+1}T_i - t^{\frac12}T_iT_{i+1} - t^{\frac12}T_{i+1}T_i + t T_i +t T_{i+1} - t^{\frac32} \\
&= a^3e_ie_{i+1}e_i - a e_i = a^3e_{i+1}e_ie_{i+1}-ae_{i+1} \\
%&=\tau_i\tau_{i+1}\tau_i 
%+ \tau_{i+1}\tau_i c_{\alpha_{i+1}}
%+ \tau_i\tau_{i+1} c_{\alpha_i} 
%+ \tau_i c_{\alpha_i+\alpha_{i+1}}c_{\alpha_{i+1}}
%+ \tau_{i+1}  c_{\alpha_{i+1}+\alpha_i}c_{\alpha_i}
%- c_{\alpha_i}c_{\alpha_{i+1}} c_{\alpha_i+\alpha_{i+1}} 
%\\
&=\tau_i\tau_{i+1}\tau_i 
-  t^{-\frac12} \tau_{i+1}\tau_i 
\frac{f_{\vep_{i+1}-\vep_{i+2}+1}}{f_{\vep_{i+1}-\vep_{i+2}}} 
%c_{\alpha_{i+1}}
- t^{-\frac12} \tau_i\tau_{i+1}
\frac{f_{\vep_{i+1}-\vep_i+1}}{f_{\vep_{i+1}-\vep_i}}
%c_{\alpha_i} 
\\
&\qquad
+ t^{-1} \tau_i 
\frac{f_{\vep_{i+1}-\vep_{i+2}+1} f_{\vep_{i+2}-\vep_i+1} } 
{f_{\vep_{i+1}-\vep_{i+2}} f_{\vep_{i+2}-\vep_i} } 
%c_{\alpha_i+\alpha_{i+1}}c_{\alpha_{i+1}}
+ t^{-1} \tau_{i+1}  
\frac{f_{\vep_{i+2}-\vep_i+1} f_{\vep_{i+1}-\vep_i+1} } 
{f_{\vep_{i+2}-\vep_i} f_{\vep_{i+1}-\vep_i}}  
%c_{\alpha_{i+1}+\alpha_i}c_{\alpha_i}
\\
&\qquad
- t^{-\frac32} 
\frac{f_{\vep_{i+1}-\vep_{i+2}+1} f_{\vep_{i+2}-\vep_i+1} f_{\vep_{i+1}-\vep_i+1} } 
{f_{\vep_{i+1}-\vep_{i+2}} f_{\vep_{i+2}-\vep_i} f_{\vep_{i+1}-\vep_i+1} } 
%- c_{\alpha_i}c_{\alpha_{i+1}} c_{\alpha_i+\alpha_{i+1}} 
,
\end{align*}
\begin{align*}
&N_0 p_0^{(\emptyset, 1^2)}
=T_0T_1T_0T_1 - t_0^{\frac12}T_1T_0T_1 - t^{\frac12}T_0T_1T_0 
+t_0^{\frac12}t^{\frac12}T_0T_1 + t_0^{\frac12}t^{\frac12}T_1T_0
- t_0t^{\frac12}T_1 - t_0^{\frac12}tT_0 + t_0t \\
&= a_0^2a^2 e_0e_1e_0e_1 - a_0a(t_0^{-\frac12}t^{\frac12}+t_0^{\frac12}t^{-\frac12} ) e_0e_1 
= a_0^2a^2 e_1e_0e_1e_0 - a_0a(t_0^{-\frac12}t^{\frac12}+t_0^{\frac12}t^{-\frac12} ) e_1e_0 \\
%&= \tau_0\tau_1\tau_0\tau_1
%- \tau_1\tau_0\tau_1 c_{\alpha_0}
%- \tau_0\tau_1\tau_0 c_{\alpha_1}
%+ \tau_0\tau_1 c_{s_0\alpha_1} c_{\alpha_0}
%+ \tau_1\tau_0 c_{s_1\alpha_0} c_{\alpha_1} 
%\\
%&\qquad
%- \tau_1 c_{s_1\alpha_0} c_{s_0\alpha_1} c_{\alpha_0}
%- \tau_0 c_{s_0\alpha_1} c_{s_1\alpha_0} c_{\alpha_1} 
%+ c_{\alpha_0} c_{s_0\alpha_1} c_{s_1\alpha_0} c_{\alpha_1}
%\\
&= \tau_0\tau_1\tau_0\tau_1
- t_0^{\frac12} \tau_1\tau_0\tau_1 
\frac{f_{\vep_1-r_2}f_{\vep_1-r_1} }{f_{2\vep_1}}
%-f_{\alpha_0}
- t^{-\frac12} \tau_0\tau_1\tau_0 
\frac{f_{\vep_2-\vep_1+1}}{f_{\vep_2-\vep_1}}
%-f_{\alpha_1}
\\
&\qquad
+ t_0^{\frac12}t^{-\frac12} \tau_0\tau_1 
\frac{f_{-\vep_2-\vep_1+1}}{f_{-\vep_2-\vep_1}}
\frac{f_{\vep_1-r_2}f_{\vep_1-r_1} }{f_{2\vep_1}}
%+f_{s_0\alpha_1} f_{\alpha_0}
- t_0^{\frac12}t^{-\frac12} \tau_1\tau_0 
\frac{f_{\vep_2-r_2}f_{\vep_2-r_1} }{f_{2\vep_2}}
\frac{f_{\vep_2-\vep_1+1}}{f_{\vep_2-\vep_1}}
%+c_{s_1\alpha_0} f_{\alpha_1} 
\\
&\qquad
- t_0 t^{-\frac12} \tau_1 
\frac{f_{\vep_2-r_2}f_{\vep_2-r_1} }{f_{2\vep_2}}
\frac{f_{-\vep_2-\vep_1+1}}{f_{-\vep_2-\vep_1}}
\frac{f_{\vep_1-r_2}f_{\vep_1-r_1} }{f_{2\vep_1}}
%-c_{s_1\alpha_0} c_{s_0\alpha_1} c_{\alpha_0}
- t_0^{\frac12} t^{-1} \tau_0 
\frac{f_{-\vep_2-\vep_1+1}}{f_{-\vep_2-\vep_1}}
\frac{f_{\vep_2-r_2}f_{\vep_2-r_1} }{f_{2\vep_2}}
\frac{f_{\vep_2-\vep_1+1}}{f_{\vep_2-\vep_1}}
\\
&\qquad
+ t_0 t^{-1}
\frac{f_{\vep_1-r_2}f_{\vep_1-r_1} }{f_{2\vep_1}}
\frac{f_{-\vep_2-\vep_1+1}}{f_{-\vep_2-\vep_1}}
\frac{f_{\vep_2-r_2}f_{\vep_2-r_1} }{f_{2\vep_2}}
\frac{f_{\vep_2-\vep_1+1}}{f_{\vep_2-\vep_1}}, 
\end{align*}
and
\begin{align*}
N_0' p_0^{(1^2,\emptyset)}
&=T_0T_1T_0T_1 + t_0^{-\frac12}T_1T_0T_1 - t^{\frac12}T_0T_1T_0 - t_0^{-\frac12}t^{\frac12}T_0T_1
- t_0^{-\frac12}t^{\frac12}T_1T_0 - t_0t^{\frac12}T_1 + t_0^{-\frac12}tT_0 + t_0t \\
&= (a_0^2a^2 e_0e_1e_0e_1 - a_0a(t_0^{-\frac12}t^{\frac12}+t_0^{\frac12}t^{-\frac12} ) e_0e_1)
- (a_0a^2e_1e_0e_1-a(t_0^{-\frac12}t^{\frac12}+t_0^{\frac12}t^{-\frac12} )e_1) \\
&= \tau_0\tau_1\tau_0\tau_1
- t_0^{\frac12} \tau_1\tau_0\tau_1 W_1^{-2}
\frac{f_{-\vep_1-r_2}f_{-\vep_1-r_1} }{f_{2\vep_1}}
%+c_{-\alpha_0}
- t^{-\frac12} \tau_0\tau_1\tau_0 
\frac{f_{\vep_2-\vep_1+1}}{f_{\vep_2-\vep_1}}
%-c_{\alpha_1}
\\
&\qquad
+ t_0^{\frac12}t^{-\frac12} \tau_0\tau_1 W_1^{-2}
\frac{f_{-\vep_2-\vep_1+1}}{f_{-\vep_2-\vep_1}}
\frac{f_{-\vep_1-r_2}f_{-\vep_1-r_1} }{f_{2\vep_1}}
%-c_{s_0\alpha_1} c_{-\alpha_0}
+ t_0^{\frac12}t^{-\frac12} \tau_1\tau_0 W_2^{-2}
\frac{f_{-\vep_2-r_2}f_{-\vep_2-r_1} }{f_{2\vep_2}}
\frac{f_{\vep_2-\vep_1+1}}{f_{\vep_2-\vep_1}}
%-c_{-s_1\alpha_0} c_{\alpha_1}
\\
&\qquad
- t_0 t^{-\frac12} \tau_1 W_1^{-2}W_2^{-2}
\frac{f_{-\vep_2-r_2}f_{-\vep_2-r_1} }{f_{2\vep_2}}
\frac{f_{-\vep_2-\vep_1+1}}{f_{-\vep_2-\vep_1}}
\frac{f_{-\vep_1-r_2}f_{-\vep_1-r_1} }{f_{2\vep_1}}
%-c_{-s_1\alpha_0} c_{s_0\alpha_1} c_{-\alpha_0}
\\
&\qquad
- t_0^{\frac12} t^{-1} \tau_0 W_2^{-2}
\frac{f_{-\vep_2-\vep_1+1}}{f_{-\vep_2-\vep_1}}
\frac{f_{-\vep_2-r_2}f_{-\vep_2-r_1} }{f_{2\vep_2}}
\frac{f_{\vep_2-\vep_1+1}}{f_{\vep_2-\vep_1}}
%+c_{s_0\alpha_1} c_{-s_1\alpha_0} c_{\alpha_1} 
\\
&\qquad
+ t_0 t^{-1} W_1^{-2}W_2^{-2}
\frac{f_{-\vep_1-r_2}f_{-\vep_1-r_1} }{f_{2\vep_1}}
\frac{f_{-\vep_2-\vep_1+1}}{f_{-\vep_2-\vep_1}}
 \frac{f_{-\vep_2-r_2}f_{-\vep_2-r_1} }{f_{2\vep_2}}
\frac{f_{\vep_2-\vep_1+1}}{f_{\vep_2-\vep_1}}.
%+c_{-\alpha_0} c_{s_0\alpha_1} c_{-s_1\alpha_0} c_{\alpha_1}.
\end{align*}
\end{prop}
\begin{proof}
The expressions in terms of $T_i$ are proved by using the relations
$T_i^2 = (t^{\frac12}-t^{-\frac12})T_i+1$ and $T_0^2 = (t_0^{\frac12}-t_0^{-\frac12})T_0+1$
to show that the equations in \eqref{piconds} are satisfied.
%\begin{align*}
%T_i&T_{i+1}T_i - t^{\frac12}T_iT_{i+1} - t^{\frac12}T_{i+1}T_i + t T_i +t T_{i+1} - t^{\frac32} \\
%&=(T_i - t^{\frac12})(T_{i+1}T_i -t^{\frac12}T_{i+1}+t) \\
%&=(T_{i+1}-t^{\frac12})(T_iT_{i+1}-t^{-\frac12}T_i+t), \\
%T_0&T_1T_0T_1 - t_0^{\frac12}T_1T_0T_1 - t^{\frac12}T_0T_1T_0 +t_0^{\frac12}t^{\frac12}T_0T_1
%+t_0^{\frac12}t^{\frac12}T_1T_0-t_0t^{\frac12}T_1 - t_0^{\frac12}tT_0+t_0t \\
%&=(T_0-t_0^{\frac12})(T_1T_0T_1 -t^{\frac12}T_1T_0+t_0^{\frac12}t^{\frac12}T_1 - t_0^{\frac12}t) \\
%&=(T_1-t^{\frac12})(T_0T_1T_0 - t_0^{\frac12}T_0T_1 + t_0^{\frac12}t^{\frac12}T_0 - t_0t^{\frac12}), 
%\\
%T_0&T_1T_0T_1 + t_0^{-\frac12}T_1T_0T_1 - t^{\frac12}T_0T_1T_0 - t_0^{-\frac12}t^{\frac12}T_0T_1
%- t_0^{-\frac12}t^{\frac12}T_1T_0 - t_0t^{\frac12}T_1 + t_0^{-\frac12}tT_0 + t_0t \\
%&=(T_0+t_0^{-\frac12})(T_1T_0T_1 - t^{\frac12}T_1T_0 - t_0^{-\frac12}t^{\frac12}T_1 + t_0^{\frac12}t) \\
%&=(T_1-t^{\frac12})(T_0T_1T_0 + t_0^{\frac12}T_0T_1 - t_0^{\frac12}t^{\frac12}T_0 - t_0t^{\frac12})
%\end{align*}
In view of the conditions \eqref{idempconds}, using the equations \eqref{piconds}
to compute the product of the expansion in terms of the $T_i$ with each element
$p_i^{(1^3)}$, $p_0^{(\emptyset, 1^2)}$ and $p_0^{(1^2,\emptyset)}$ respectively,
determines the normalizing constants
\begin{align*}
N &= -t^{-\frac32}-t^{-\frac12}-t^{-\frac12}-t^{\frac12}-t^{\frac12}-t^{\frac32} % \\
%&= -t^{-\frac12}(1+t+t+t^2+t^2+t^3) 
= t^{-\frac12}(1+t)(1+t+t^2), \quad\hbox{and} \\
N_0 =N_0'
&= t_0^{-1}t^{-1}+t^{-1}+t_0^{-1}+1+1+t_0+t+t_0t % \\
%&=t_0^{-1}t^{-1}(1+t_0+t+2t_0t+t_0^2t+t_0t^2+t_0^2t^2) \\
=t_0^{-1}t^{-1}(1+t_0)(1+t)(1+t_0t).
\end{align*}

Checking the conditions \eqref{pieiconds} verifies that the expressions in terms of the $e_i$ for
the elements  
$N p_i^{(1^3)}$, $N_0 p_0^{(\emptyset, 1^2)}$ and 
$N_0' p_0^{(1^2, \emptyset)}$ are correct.  
Similarly, using the expressions for $a_0e_0$ and $ae_i$ in terms of $\tau_i$
given in \eqref{eiintaui} to check these same conditions verifies that the
expressions for the elements $N p_i^{(1^3)}$, $N_0 p_0^{(\emptyset, 1^2)}$ and 
$N_0' p_0^{(1^2, \emptyset)}$ 
in terms of the $\tau_i$ are correct.  
%Alternatively, these expansions in terms of $\tau_i$
%can be obtained easily by expanding and simplifying
%\begin{align*}
%a^3e_ie_{i+1}&e_i - ae_i
%=(\tau_i-c_{\alpha_i})(\tau_{i+1}-c_{\alpha_{i+1}})(\tau_i-c_{\alpha_i})
%- (\tau_i-c_{\alpha_i}) \\
%&=\tau_i\tau_{i+1}\tau_i - \tau_{i+1}\tau_i c_{s_{i+1}s_i\alpha_i}
%- \tau_i\tau_{i+1} c_{\alpha_i} - \tau_i^2 c_{s_i\alpha_{i+1}} \\
%&\qquad 
%+\tau_i c_{\alpha_{i+1}}c_{\alpha_i} + \tau_i c_{s_i\alpha_{i+1}}c_{s_i\alpha_i} 
%+\tau_{i+1} c_{s_{i+1}\alpha_i}c_{\alpha_i} - c_{\alpha_i}^2c_{\alpha_{i+1}}
%- \tau_i +c_{\alpha_i}.
%\end{align*}
\end{proof}

\subsection{Setting up the relation 
$a_k a^2e_{k-1}e_k e_{k-1} - a(t_k^{-\frac12}t^{\frac12}+t_k^{\frac12}t^{-\frac12} ) e_{k-1}=0$}

As in \cite[Remark 2.3]{DR}, let $w_A$ be the longest element of $WA_k = \langle s_1, \ldots, s_{k-1}\rangle$.
Let
$$T_{0^\vee} = T_{w_A}^{-1} T_k T_{w_A} = 
a_1^{-\frac12}(-a_2)^{-\frac12}
{\def\TOP{8}\def\K{5}\def\Left{-.85}\def\Right{.15}
\TikZ{[scale=.4]
		\Pole[\Left][0,.5*\TOP] 
		\Over[\K,2.5][-1.3,4]
		\Over[-1.3,4][\K,5.5]
		\Pole[\Left][.5*\TOP,\TOP] 
		\Pole[\Right][0,\TOP] 
		\draw[thin, rounded corners] (1,\TOP) to (1,\TOP-.25) to [bend left=10] (5,.5*\TOP+2) to  (5,.5*\TOP+1.5)
								(1,0) to (1,.25) to [bend right=10] (5,.5*\TOP-2) to  (5,.5*\TOP-1.5);
		\draw[over, rounded corners] (2,\TOP) to (1.75,\TOP-.75) to (4,.5*\TOP+1.5) to  (4,.5*\TOP-1.5) to (2,.75)to  (2,0);
		\draw[over, rounded corners] (3,\TOP) to (3,\TOP-.75) to (1.5,.5*\TOP+2.5) to (3,.5*\TOP+1) to  (3,.5*\TOP-1) 
									 to (1.5,.5*\TOP-2.5)to (3,.75) to (3,0);
		\draw[over, rounded corners] (4,\TOP) to (4,\TOP-1) to (1.25,.5*\TOP+1.75) to (2,.5*\TOP+1)  
										to (2,.5*\TOP-1) to (1.25,.5*\TOP-1.75) to (4,1) to (4,0);
		\draw[over, rounded corners] (5,\TOP) to (5,\TOP-1.25) to (1,.5*\TOP+1) to (1,.5*\TOP-1) to  (5,1.25) to (5,0);
		\Caps[.15,-.85][0,\TOP][\K]
}}  = a_1^{-\frac12}(-a_2)^{-\frac12} 
{\def\TOP{3} \def\K{5}
\TikZ{[scale=.4]
	\Pole[-.85][0,.5*\TOP]
	\Over[1,0][-1.3,.5*\TOP]
	\Under[-1.3,.5*\TOP][1,\TOP]
	\Pole[-.85][.5*\TOP,\TOP]
	\Pole[.15][0,\TOP]
	 \foreach \x in {2,...,\K}  {
		 \draw[thin] (\x,0) -- (\x,\TOP);
		 }
	\Caps[.15,-.85][0,\TOP][\K]
	}} = W_1T_0^{-1},
$$
and note that
$T_{w_A}^{-1} T_{k-1} T_{w_A} = T_1$.
Then
$$(T_{0^\vee}-t_k^{\frac12})(T_{0^\vee}+t_k^{-\frac12}) = 0
\qquad\hbox{and}\qquad
T_{0^\vee}T_1T_{0^\vee}T_1 = T_1T_{0^\vee}T_1T_{0^\vee}.$$
Let $HB^\vee_2$ be the subalgebra of $H^{\mathrm{ext}}_k$ 
generated by $T_{0^\vee}$ and $T_1$ and define
idempotents $p_{0^\vee}^{(\emptyset, 1^2)}$ and $p_{0^\vee}^{(1^2, \emptyset)}$
in $HB^\vee_2$ by the equations
\begin{equation}
(p_{0^\vee}^{(\emptyset, 1^2)})^2=p_{0^\vee}^{(\emptyset, 1^2)},
\qquad
(p_{0^\vee}^{(1^2, \emptyset)})^2=p_{0^\vee}^{(1^2, \emptyset)};
\label{idempveeconds}
\end{equation}
and
\begin{equation}
\begin{array}{lcl}
T_{0^\vee} p_{0^\vee}^{(\emptyset, 1^2)} = -t_k^{-\frac12} p_{0^\vee}^{(\emptyset, 1^2)},
&\quad &T_1 p_{0^\vee}^{(\emptyset, 1^2)} = -t^{-\frac12} p_{0^\vee}^{(\emptyset, 1^2)}, \\
%(p_0^{(1^2, \emptyset)})^2 = p_0^{(1^2, \emptyset)}, &\quad &
T_{0^\vee} p_{0^\vee}^{(1^2, \emptyset)} = t_k^{\frac12} p_{0^\vee}^{(1^2, \emptyset)},  
&\quad \text{ and } \quad &T_1 p_{0^\vee}^{(1^2, \emptyset)} = -t^{-\frac12} p_{0^\vee}^{(1^2, \emptyset)}.
\end{array}
\label{p0veeconds}
\end{equation}
Let $a_k\in \CC^\times$ and define
\begin{equation}
a_ke_{0^\vee} = T_{0^\vee} - t_k^{\frac12},
\qquad\hbox{so that}\quad
e_{0^\vee} = T_{w_A} e_k T_{w_A}^{-1} 
\quad\hbox{and}\quad
e_1 = T_{w_A} e_{k-1} T_{w_A}^{-1}.
\label{conjugateek}
\end{equation}
The conditions in \eqref{p0veeconds} are equivalent to 
\begin{equation}
\begin{array}{lcl}
a_ke_{0^\vee} p_{0^\vee}^{(\emptyset, 1^2)} = -(t_k^{\frac12}+t_k^{-\frac12}) p_{0^\vee}^{(\emptyset, 1^2)}, 
&\quad &ae_1 p_{0^\vee}^{(\emptyset, 1^2)} = -(t^{\frac12}+t^{-\frac12}) p_{0^\vee}^{(\emptyset, 1^2)}, \\
%(p_0^{(1^2, \emptyset)})^2 = p_0^{(1^2, \emptyset)}, &\quad &
a_ke_{0^\vee} p_{0^\vee}^{(1^2, \emptyset)} = 0,
&\quad\text{ and } \quad  &ae_1 p_{0^\vee}^{(1^2, \emptyset)} = -(t^{\frac12}+t^{-\frac12}) p_{0^\vee}^{(1^2, \emptyset)}.
\end{array}
\label{p0e0veeconds}
\end{equation}
Using $a_ke_{0^\vee} %= T_{0^\vee} - t_k^{\frac12}
=W_1 T_0^{-1} - t_k^{\frac12}
= W_1 (T_0 - (t_0^{\frac12}-t_0^{-\frac12})) - t_k^{\frac12} 
= W_1(\tau_0+t_0^{\frac12}-c_{\alpha_0}- (t_0^{\frac12}-t_0^{-\frac12})) - t_k^{\frac12}$,
a short computation gives
$$a_ke_{0^\vee} 
%= \Big(\tau_0  - t_0^{-\frac12} 
%\frac{f_{\varepsilon_1-r_2}f_{-\varepsilon_1-r_1} } {f_{2\varepsilon_1}}
%\Big)W_1^{-1}
= \tau_0W_1^{-1} - t_0^{-\frac12} W_1^{-1}
\frac{f_{\varepsilon_1-r_2}f_{-\varepsilon_1-r_1} } {f_{2\varepsilon_1}}.
 $$
And, with 
$N_k = t_k^{-1}t^{-1}(1+t_k)(1+t)(1+t_k t)$, we have
\begin{align*}
N_k p_{0^\vee}^{(\emptyset, 1^2)}
%&=T_0T_1T_0T_1 - t_0^{\frac12}T_1T_0T_1 - t^{\frac12}T_0T_1T_0 +t_0^{\frac12}t^{\frac12}T_0T_1
%+t_0^{\frac12}t^{\frac12}T_1T_0-t_0t^{\frac12}T_1 - t_0^{\frac12}tT_0+t_0t \\
&= a_k^2a^2 e_{0^\vee}e_1e_{0^\vee}e_1 - a_ka(t_k^{-\frac12}t^{\frac12}+t_k^{\frac12}t^{-\frac12} ) e_{0^\vee}e_1 
= a_k^2a^2 e_1e_{0^\vee}e_1e_{0^\vee} - a_ka(t_k^{-\frac12}t^{\frac12}+t_k^{\frac12}t^{-\frac12} ) e_1e_{0^\vee} \\
&= \tau_0\tau_1\tau_0\tau_1 (W_1W_2)^{-1}
- t_0^{-\frac12}  \tau_1\tau_0  \tau_1 (W_1W_2)^{-1} 
\frac{f_{\vep_1-r_2}f_{-\vep_1-r_1} }{f_{2\vep_1}} 
%-f_{\alpha_{0^\vee}}
+ t^{-\frac12} \tau_0 \tau_1\tau_0 (W_1W_2)^{-1}
 \frac{f_{\vep_2-\vep_1+1}}{f_{\vep_2-\vep_1}}
%-f_{\alpha_1}
\\
&\qquad
- t_0^{-\frac12} t^{-\frac12}  \tau_0 \tau_1 (W_1W_2)^{-1}  
\frac{f_{-\vep_2-\vep_1+1}}{f_{-\vep_2-\vep_1}}
\frac{f_{\vep_1-r_2}f_{-\vep_1-r_1} }{f_{2\vep_1}} 
%+f_{s_0\alpha_1} f_{\alpha_{0^\vee}}
\\
&\qquad
- t_0^{-\frac12} t^{-\frac12} \tau_1\tau_0 (W_1W_2)^{-1}
 \frac{f_{\vep_2-r_2}f_{-\vep_2-r_1} }{f_{2\vep_2}}
 \frac{f_{\vep_2-\vep_1+1}}{f_{\vep_2-\vep_1}}
%+c_{s_1\alpha_{0^\vee}} f_{\alpha_1} 
\\
&\qquad
+ t_0^{-1}t^{-\frac12} \tau_1 (W_1W_2)^{-1}
\frac{f_{\vep_2-r_2}f_{-\vep_2-r_1} }{f_{2\vep_2}} 
\frac{f_{-\vep_2-\vep_1+1}}{f_{-\vep_2-\vep_1}}
\frac{f_{\vep_1-r_2}f_{-\vep_1-r_1} }{f_{2\vep_1}} 
%-c_{s_1\alpha_{0^\vee}} c_{s_0\alpha_1} c_{\alpha_{0^\vee}}
\\
&\qquad
- t_0^{-\frac12} t^{-1} \tau_0 (W_1W_2)^{-1}
 \frac{f_{-\vep_2-\vep_1+1}}{f_{-\vep_2-\vep_1}}
\frac{f_{\vep_2-r_2}f_{-\vep_2-r_1} }{f_{2\vep_2}} 
\frac{f_{\vep_2-\vep_1+1}}{f_{\vep_2-\vep_1}}
%-c_{s_0\alpha_1} c_{s_1\alpha_{0^\vee}} c_{\alpha_1} 
\\
&\qquad
+ t_0^{-1}t^{-1}  (W_1W_2)^{-1}
\frac{f_{\vep_1-r_2}f_{-\vep_1-r_1} }{f_{2\vep_1}} 
\frac{f_{-\vep_2-\vep_1+1}}{f_{-\vep_2-\vep_1}}
\frac{f_{\vep_2-r_2}f_{-\vep_2-r_1} }{f_{2\vep_2}} 
\frac{f_{\vep_2-\vep_1+1}}{f_{\vep_2-\vep_1}},
%+c_{\alpha_0} c_{s_0\alpha_1} c_{s_1\alpha_0} c_{\alpha_1},
\end{align*}
%%%
%Next is p_{0^\vee}^{(1^2,\emptyset)}
%%%
and
\begin{align*}
&N_k p_{0^\vee}^{(1^2,\emptyset)}
%&=T_0T_1T_0T_1 + t_0^{-\frac12}T_1T_0T_1 - t^{\frac12}T_0T_1T_0 - t_0^{-\frac12}t^{\frac12}T_0T_1
%- t_0^{-\frac12}t^{\frac12}T_1T_0 - t_0t^{\frac12}T_1 + t_0^{-\frac12}tT_0 + t_0t \\
= (a_k^2a^2 e_{0^\vee}e_1e_{0^\vee}e_1 
- a_ka(t_k^{-\frac12}t^{\frac12}+t_k^{\frac12}t^{-\frac12} ) e_{0^\vee}e_1)
- (a_ka^2e_1e_{0^\vee}e_1-a(t_k^{-\frac12}t^{\frac12}+t_k^{\frac12}t^{-\frac12} )e_1) \\
&= \tau_0  \tau_1\tau_0  \tau_1 (W_1W_2)^{-1}
- t_0^{-\frac12} \tau_1\tau_0 \tau_1  (W_1W_2)^{-1}
 \frac{f_{-\vep_1-r_2} f_{\vep_1-r_1} } {f_{2\vep_1}} 
%f_{-\alpha_{0^\vee}}
+ t^{-\frac12} \tau_0 \tau_1\tau_0 (W_1W_2)^{-1} 
\frac{f_{\vep_2-\vep_1+1}}{f_{\vep_2-\vep_1}}
% - f_{\alpha_1}
\\
&\qquad
- t_0^{-\frac12}t^{-\frac12} \tau_0  \tau_1 (W_1W_2)^{-1}
\frac{f_{-\vep_2-\vep_1+1}}{f_{-\vep_2-\vep_1}}
\frac{f_{-\vep_1-r_2} f_{\vep_1-r_1} } {f_{2\vep_1}} 
% - f_{s_0\alpha_1} f_{-\alpha_{0^\vee}}
\\
&\qquad
- t_0^{-\frac12} t^{-\frac12}  \tau_1\tau_0 (W_1W_2)^{-1}
\frac{f_{-\vep_2-r_2} f_{\vep_2-r_1} } {f_{2\vep_2}} 
\frac{f_{\vep_2-\vep_1+1}}{f_{\vep_2-\vep_1}}
% - f_{-s_1\alpha_{0^\vee}} f_{\alpha_1} 
\\
&\qquad
+ t_0^{-1}t^{-\frac12} \tau_1 (W_1W_2)^{-1}
\frac{f_{-\vep_2-r_2} f_{\vep_2-r_1} } {f_{2\vep_2}} 
\frac{f_{-\vep_2-\vep_1+1}}{f_{-\vep_2-\vep_1}}
\frac{f_{-\vep_1-r_2} f_{\vep_1-r_1} } {f_{2\vep_1}} 
% - f_{-s_1\alpha_{0^\vee}} f_{s_0\alpha_1} f_{-\alpha_{0^\vee}}
\\
&\qquad
- t_0^{-\frac12}t^{-1} \tau_0 (W_1^{-1}W_2)^{-1}
\frac{f_{-\vep_2-\vep_1+1}}{f_{-\vep_2-\vep_1}}
 \frac{f_{-\vep_2-r_2} f_{\vep_2-r_1} } {f_{2\vep_2}} 
\frac{f_{\vep_2-\vep_1+1}}{f_{\vep_2-\vep_1}}
%f_{s_0\alpha_1} f_{-s_1\alpha_{0^\vee}} f_{\alpha_1}
\\
&\qquad 
+ t_0^{-1} t^{-1} (W_1W_2)^{-1}
 \frac{f_{-\vep_1-r_2} f_{\vep_1-r_1} } {f_{2\vep_1}} 
\frac{f_{-\vep_2-\vep_1+1}}{f_{-\vep_2-\vep_1}}
\frac{f_{-\vep_2-r_2} f_{\vep_2-r_1} } {f_{2\vep_2}} 
\frac{f_{\vep_2-\vep_1+1}}{f_{\vep_2-\vep_1}},
%f_{-\alpha_{0^\vee}} c_{s_0\alpha_1} c_{-s_1\alpha_{0^\vee}} f_{\alpha_1},
\end{align*}
in analogy with (and with the same proof as) Proposition \ref{idempexpansion}.

\section{The two boundary Temperley-Lieb algebra $TL_k$}

In this section we define the two boundary Temperley-Lieb algebra $TL_k$ (also called the symplectic
blob algebra, see \cite{GMP07, GMP08, GMP12, Re12, KMP16, GMP17}) 
and review its diagrammatic
calculus.  We extend the diagrammatic calculus to make clear the relationship to the two boundary
Hecke algebra and to set the stage for the proof of Theorem \ref{IIIexpansion}.  Although 
Theorem \ref{IIIexpansion} takes the form of a computation, it is a computation that has amazing
consequences as it determines the relationship between the center of $H_k^{\mathrm{ext}}$
and the center of $TL_k$.  The center of $H_k^{\mathrm{ext}}$ is a ring of symmetric functions
(see \cite[Theorem 2.3]{DR}) and the center of $TL_k$ turns out to be a polynomial ring $\CC[Z]$
in a single variable $Z$.  We shall see that, in the same way that  $H_k^{\mathrm{ext}}$ is finite rank over its center, the algebra $TL_k$ is finite rank over $\CC[Z]$. However, whereas the former has the easily classified rank of  $(2^k k!)^2$ over its center, the rank of  $TL_k$  is as yet unclassified combinatorially. For example, $\dim(TL_k(b)) = 5$, $19$, $84$, $335$, and $1428$, for $k=1$, $2$, $3$, $4$, and $5$, respectively. 

\subsection{The extended two boundary Temperley-Lieb algebra $TL^{\mathrm{ext}}_k$}
Let $H^{\mathrm{ext}}_k$ be the extended two boundary Hecke algebra as defined in \eqref{DRTBHeckerelations}.
The \emph{extended two boundary Temperley-Lieb algebra} $TL^{\mathrm{ext}}_k$ 
is the quotient of $H^{\mathrm{ext}}_k$ by the relations
$$p_{0^\vee}^{(\emptyset, 1^2)}=p_{0^\vee}^{(1^2, \emptyset)}, \qquad
p_0^{(\emptyset, 1^2)}=p_0^{(1^2, \emptyset)}
\quad\hbox{and}\qquad
p_i^{(1^3)}=0
\quad\hbox{for $i\in \{1, \ldots, k-2\}$.}
$$

\begin{thm} \label{idempquotient} 
The algebra $TL^{\mathrm{ext}}_k$ is the quotient of $H^{\mathrm{ext}}_k$ by the relations
$$a_k a^2e_{k-1}e_k e_{k-1} - a(t_k^{-\frac12}t^{\frac12}+t_k^{\frac12}t^{-\frac12} ) e_{k-1}=0,
\qquad
a_0a^2e_1e_0e_1-a(t_0^{-\frac12}t^{\frac12}+t_0^{\frac12}t^{-\frac12} )e_1=0,$$
and
$$a^3e_ie_{i+1}e_i - a e_i = a^3e_{i+1}e_ie_{i+1}-ae_{i+1}=0 \quad \text{
for $i\in \{1, \dots, k-2\}$.}$$
\end{thm}
\begin{proof}  
%THE EXTREMELY IMPORTANT POINT IS THAT
%$$
%T_{w_A}
%(a_k a^2e_{k-1}e_k e_{k-1} 
%- a(t_k^{-\frac12}t^{\frac12}+t_k^{\frac12}t^{-\frac12} ) e_{k-1} )
%T_{w_A}^{-1} = a_ka^2 e_1 e_{0^\vee} e_1 
%- a(t_k^{-\frac12}t^{\frac12}+t_k^{\frac12}t^{-\frac12})e_1 .$$
Let 
$F_i =a^3e_ie_{i+1}e_i - a e_i = a^3e_{i+1}e_ie_{i+1}-ae_{i+1}$ for $i\in \{1, \ldots, k-2\}$,
\begin{align*}
F_k = a_k a^2e_{k-1}e_k e_{k-1} - a(t_k^{-\frac12}t^{\frac12}+t_k^{\frac12}t^{-\frac12} ) e_{k-1}, 
\quad\hbox{and}\quad
F_0 = a_0a^2e_1e_0e_1-a(t_0^{-\frac12}t^{\frac12}+t_0^{\frac12}t^{-\frac12} )e_1.
\end{align*}
By Proposition \ref{idempexpansion},
$$N_0p_0^{(1^2,\emptyset)} = e_0F_0, 
\qquad
N_0p_0^{(\emptyset, 1^2)} = (e_0-1)F_0, \qquad
F_0 = N_0(p_0^{(1^2,\emptyset)}-p_0^{(\emptyset, 1^2)}),
\quad\text{and}\quad
Np_i^{(1^3)} = F_i;
$$
and, by \eqref{conjugateek},
$$T_{w_A}F_kT_{w_A}^{-1} = N_0^\vee (p_{0^\vee}^{(1^2,\emptyset)}-p_{0^\vee}^{(\emptyset, 1^2)}),
\qquad
T_{w_A}^{-1}p_{0^\vee}^{(1^2,\emptyset)}T_{w_A} = e_kF_k,
\quad\text{and}\quad
T_{w_A}^{-1}p_{0^\vee}^{(\emptyset,1^2)}T_{w_A} = (e_k-1)F_k.
$$
Thus, provided $N$, $N_0$ and $N_k$ are invertible, 
the ideal $H^{\mathrm{ext}}_kF_kH^{\mathrm{ext}}_k$ is the same as the ideal generated by 
$(p_{0^\vee}^{(1^2,\emptyset)}$ and $p_{0^\vee}^{(\emptyset,1^2)}$;
the ideal $H^{\mathrm{ext}}_k F_0 H^{\mathrm{ext}}_k$ is the same as the ideal generated by 
$p_0^{(1^2,\emptyset)}$ and $p_0^{(\emptyset,1^2)}$;
and $H^{\mathrm{ext}}_k p_i^{(1^3)} H^{\mathrm{ext}}_k 
= H^{\mathrm{ext}}_k F_i H^{\mathrm{ext}}_k$.
\end{proof}

\subsection{The two boundary Temperley-Lieb algebra $TL_k$}

The \emph{two boundary Temperley-Lieb algebra} $TL_k$ is the subalgebra of 
$TL_k^{\mathrm{ext}}$ generated by $a_0e_0, ae_1, \ldots, ae_{k-1}, a_ke_k$
(as defined in \eqref{edefin}).  
As in \eqref{bdgpproduct} and \eqref{Hksplitting}, where 
$\cB_k^{\mathrm{ext}} = \cB_k \times \cD$ and
$H_k^{\mathrm{ext}} = H_k\otimes \CC[W_0^{\pm 1}]$,
the extended two boundary Temperley-Lieb algebra is 
$$TL_k^{\mathrm{ext}} = TL_k \otimes \CC[W_0^{\pm1}],
\qquad\hbox{as algebras, where}\qquad
W_0 = PW_1\cdots W_k.
$$

\subsection{Diagrammatic calculus for $TL_k$}
Pictorially, identify 
$$
{\def\TOP{2}\def\K{6}
T_k=
\TikZ{[scale=.5]
\Pole[.15][0,2]
\Under[\K,0][\K+1.3,1]
\Pole[\K+.85][0,1][\K]
\Pole[\K+.85][1,2][\K]
\Over[\K+1.3,1][\K,2]
 \foreach \x in {1,...,5} {
	 \draw[thin] (\x,0) -- (\x,\TOP);
	 }
\Caps[.15,\K+.85][0,\TOP][\K]
},
\quad 
T_0=
\TikZ{[scale=.5]
	\Pole[\K+.85][0,2][\K]
	\Pole[.15][0,1]
	\Over[1,0][-.3,1]
	\Under[-.3,1][1,2]
	\Pole[.15][1,2]
	 \foreach \x in {2,...,\K} {
	 \draw[thin] (\x,0) -- (\x,\TOP);
	 }
\Caps[.15,\K+.85][0,\TOP][\K]
},
\quad
T_i=
\TikZ{[scale=.5]
	\Pole[\K+.85][0,2][\K]
	\Pole[.15][0,2]
	\Under[3,0][4,2]
	\Over[4,0][3,2]
	 \foreach \x in {1,2,5,\K} {
		 \draw[thin] (\x,0) -- (\x,\TOP);
		 }
	\Caps[.15,\K+.85][0,\TOP][\K]
	\Label[0,\TOP][3][{\tiny{$i$}}]
	\Label[0,\TOP][4][{\tiny{$i$+1}}]
},
}$$
$$
e_0 = \TikZ{
\draw[dotted] (.5*.5,0) rectangle (.5*5+.25,.75);
\foreach \x in {1,2,3,5}{\node[V] (b\x) at (.5*\x, 0) {}; \node[V] (t\x) at (.5*\x, .75){};}
\foreach \y in {1,2}{\node[V]  (l\y) at (.25, \y*.25){}; }
	\draw[bend right=30] (l2) to (t1) (b1) to (l1);
	\foreach \x in {2,3,5}{\draw (t\x)--(b\x);}
	\node at (.5*4,.5*.75) {$\dots$};	
}, \quad  
e_k=\TikZ{
\draw[dotted] (.5*.5,0) rectangle (.5*5+.25,.75);
\foreach \x in {1,3,4,5}{\node[V] (b\x) at (.5*\x, 0) {}; \node[V] (t\x) at (.5*\x, .75){};}
\foreach \y in {1,2}{\node[V]  (r\y) at (.5*5+.25, \y*.25){};}
	\draw[bend right=30] (t5) to (r2) (r1) to (b5);
	\foreach \x in {1,3,4}{\draw (t\x)--(b\x);}
	\node at (.5*2,.5*.75) {$\dots$};	
}, \quad \text{ and } \quad  
ae_i=\TikZ{
\draw[dotted] (.5*.5,0) rectangle (.5*8+.25,.75);
\foreach \x in {1,3,4,5,6,8}{\node[V] (b\x) at (.5*\x, 0) {}; \node[V] (t\x) at (.5*\x, .75){};}
	\draw[bend right=60] (t4) to (t5) (b5) to (b4);
	\foreach \x in {1,3,6,8}{\draw (t\x)--(b\x);}
	\node at (.5*2,.5*.75) {$\dots$};	\node at (.5*7,.5*.75) {$\dots$};
	\node[above] at (t4) {\tiny $i$};\node[below] at (b4) {\tiny $i$};
},
$$
for $i\in \{1, \ldots, k-1\}$.   Recall the notation
$$\(x\) = x^{\frac12}+x^{-\frac12}$$
from \eqref{notation}.
With $i\in \{1, \dots, k-1\}$, the relations \eqref{edefin}, \eqref{signrep} and \eqref{eisquared} are 
\begin{equation*}
\begin{array}{r@{\ }lr@{\ }lr@{\ }l}
T_0 &= a_0e_0+t_0^{\frac12}, %\qquad
	&T_i &= ae_i + t^{\frac12}, \qquad
	&T_k &= a_ke_k + t_k^{\frac12},
\\
\TikZ{[scale=.3]
		\Under[1,1.7][-.5,.5]
		\Pole[.15][-.7,1.7]
			\draw (.3,1.5+.2) arc (0:360:4pt and 3pt);
			\draw (.3,-.5-.2) arc (0:-180:4pt and 3pt);
		\Over[-.5,.5][1,-.7]
}
&= a_0\ 
	\TikZ{[scale=.5]\Ez[0]}
	+  t_0^{1/2}\ 
	\TikZ{[scale=.3]
		\Pole[.15][-.7,1.7]
			\draw (.3,1.5+.2) arc (0:360:4pt and 3pt);
			\draw (.3,-.5-.2) arc (0:-180:4pt and 3pt);
		\draw (1,1.7)--(1,-.7);
	}  
%\label{T0toe0picture}\tag{T0toe0picture}
&
	\TikZ{[scale=.3]
		\Cross[1,1][2,0]
	}
&= 
\TikZ{[scale=.3]
		\draw[bend right=100] (1,1) to (2,1) (2,0) to (1,0);
	}
+ t^{1/2} \ 
	\TikZ{[scale=.3]
		\draw (1,0)--(1,1) (2,0)--(2,1);
	} 
%\label{Titoeipicture}\tag{Titoeipicture}
&
\TikZ{[xscale=-.3, yscale=.3]
		\Under[-.5,.5][1,-.7]
		\Pole[.15][-.7,1.7]
			\draw (.3,1.5+.2) arc (0:360:4pt and 3pt);
			\draw (.3,-.5-.2) arc (0:-180:4pt and 3pt);
		\Over[1,1.7][-.5,.5]
	}
&= a_k\ \TikZ{[scale=.5]\Ek[0][0]}
+ t_k^{1/2}\ 
	\TikZ{[xscale=-.3, yscale=.3]
		\Pole[.15][-.7,1.7]
			\draw (.3,1.5+.2) arc (0:360:4pt and 3pt);
			\draw (.3,-.5-.2) arc (0:-180:4pt and 3pt);
		\draw (1,1.7)--(1,-.7);
	} \\~\\
%%%%%%%%%%%%%%%%%%%%%%%%%
%%%%%%%%%%%%%%%%%%%%%%%%%
T_0e_0 &= e_0T_0 = -t_0^{-\frac12}e_0,
&T_i (ae_i) &= (ae_i)T_i = -t^{-\frac12}(ae_i),
\qquad
&T_k e_k &= e_kT_k = -t_k^{-\frac12}e_k, 
\\
\TikZ{[scale=.3]
	\PoleCaps[.15][0,3]
	\Over[1,1][-.5,2]
	\Pole[.15][0,3]
	\Ez[0]
	\Over[1,3][-.5,2]
	}
	&=\TikZ{[scale=.3]
	\PoleCaps[.15][0,3]
	\Over[1,0][-.5,1]
	\Pole[.15][0,3]
	\Ez[2]
	\Over[1,2][-.5,1]
	}
	= - t_0^{-1/2} \TikZ{[scale=.5]\Ez[0]}
&
	\TikZ{[scale=.3]
		\draw[bend right=100] (1,2) to (2,2) (2,1) to (1,1);
		\Cross[1,1][2,-.5]
	}
	&=
	\TikZ{[scale=.3]
		\draw[bend right=100] (1,1) to (2,1) (2,0) to (1,0);
		\Cross[1,2.5][2,1] 
	}
	= - t^{-1/2} \ 
	\TikZ{[scale=.3]
		\draw[bend right=100] (1,1) to (2,1) (2,0) to (1,0);
	}
&\TikZ{[xscale=-.3, yscale=.3]
	\PoleCaps[.15][0,3]
	\Over[1,1][-.5,2]
	\Pole[.15][0,3]
	\Ez[0]
	\Over[1,3][-.5,2]
	}
	&=\TikZ{[xscale=-.3, yscale=.3]
	\PoleCaps[.15][0,3]
	\Over[1,0][-.5,1]
	\Pole[.15][0,3]
	\Ez[2]
	\Over[1,2][-.5,1]
	}
	= - t_k^{-1/2} \TikZ{[scale=.5]\Ek[0][0]}\\~\\
%\label{signrep}
T_0^{-1}e_0 &= e_0T_0^{-1} = -t_0^{\frac12}e_0, 
\qquad
&T_i^{-1} (ae_i) &= (ae_i)T_i^{-1} = -t^{\frac12}(ae_i),
\qquad
&T_k^{-1} e_k &= e_kT_k^{-1} = -t_k^{\frac12}e_k,
\\
\TikZ{[scale=.3]
	\PoleCaps[.15][0,3]
	\Over[1,3][-.5,2]
	\Pole[.15][0,3]
	\Over[1,1][-.5,2]
	\Ez[0]
	}
	&=\TikZ{[scale=.3]
	\PoleCaps[.15][0,3]
	\Over[1,2][-.5,1]
	\Pole[.15][0,3]
	\Over[1,0][-.5,1]
	\Ez[2]
	}
	= - t_0^{1/2} \TikZ{[scale=.5]\Ez[0]}
&\TikZ{[scale=.3]
		\draw[bend right=100] (1,2) to (2,2) (2,1) to (1,1);
		 \Cross[1,-.5][2,1]
	}
	&=
	\TikZ{[scale=.3]
		\draw[bend right=100] (1,1) to (2,1) (2,0) to (1,0);
		\Cross[1,1][2,2.5] 
	}
	= - t^{1/2} \ 
	\TikZ{[scale=.3]
		\draw[bend right=100] (1,1) to (2,1) (2,0) to (1,0);
	}
&\TikZ{[xscale=-.3, yscale=.3]
	\PoleCaps[.15][0,3]
	\Over[1,3][-.5,2]
	\Pole[.15][0,3]
	\Over[1,1][-.5,2]
	\Ez[0]
	}
	&=\TikZ{[xscale=-.3, yscale=.3]
	\PoleCaps[.15][0,3]
	\Over[1,2][-.5,1]
	\Pole[.15][0,3]
	\Over[1,0][-.5,1]
	\Ez[2]
	}
	= - t_k^{1/2} \TikZ{[scale=.5]\Ek[0][0]}
\\~\\
%\label{signrep}
e_0^2 &= \displaystyle{\frac{-\(t_0\)}{a_0} e_0,}
&(ae_i)^2 &= -\(t\) (ae_i), \qquad \text{and}
& e_k^2 &= \displaystyle\frac{-\(t_k\)}{a_k} e_k.\\
\TikZ{[scale=.5]\Ez[0]\Ez[1]} 
	&= \frac{-\(t_0\)}{a_0} \TikZ{[scale=.5]\Ez[0]} 
&\TikZ{[scale=.3]
	 \draw (0,0) circle (20pt);}  
	 &= - [\![t]\!] 
&\TikZ{[scale=.5]\Ek[0][0]\Ek[1][0]} 
	&= \frac{-\(t_k\)}{a_k} \TikZ{[scale=.5]\Ek[0][0]}
%\label{eisquared}
\end{array}
\end{equation*}
In the quotient by $(ae_i) (ae_{i+1}) (ae_i) = a e_i$, we have
\begin{equation}
\begin{array}{cc}
a e_i T_{i+1}T_i = aT_{i+1}T_i e_{i+1} = t^{\frac12} a^2 e_i e_{i+1}, \quad
&a e_i T_{i+1}^{-1}T_i^{-1} = aT_{i+1}^{-1}T_i^{-1} e_{i+1} = t^{-\frac12} a^2 e_i e_{i+1}, \\
a e_{i+1} T_i T_{i+1} = aT_i T_{i+1} e_i = t^{\frac12} a^2 e_{i+1} e_i, \quad
&a e_{i+1} T_i^{-1} T_{i+1}^{-1} = aT_i^{-1} T_{i+1}^{-1} e_i = t^{-\frac12} a^2 e_{i+1} e_i, 
\end{array}
\end{equation}
%Pictorially these are
\begin{equation*}
	\TikZ{[scale=.3]
		\draw[bend right=100] (1,3) to (2,3) (2,2) to (1,2);
		\Cross[2,2][3,1] \Cross[1,1][2,0]  \draw(1,2)--(1,1) (3,1)--(3,0) (3,2)--(3,3);
	}
	=
	\TikZ{[scale=.3]
		\draw[bend right=100] (2,1) to (3,1) (3,0) to (2,0);
		\Cross[2,3][3,2] \Cross[1,2][2,1]  \draw(1,3)--(1,2) (3,1)--(3,2) (1,1)--(1,0);
	}
	=  t^{1/2}\ 
	\TikZ{[scale=.3]
		\draw[bend right=100] (1,2) to (2,2) (3,0) to (2,0);
		\Under[3,2][1,0]
	} 
	\qquad\qquad 
	\TikZ{[xscale=-.3, yscale=.3]
		\draw[bend right=100] (1,3) to (2,3) (2,2) to (1,2);
		\Cross[2,2][3,1] \Cross[1,1][2,0]  \draw(1,2)--(1,1) (3,1)--(3,0) (3,2)--(3,3);	
	}
	=
	\TikZ{[xscale=-.3, yscale=.3]
		\draw[bend right=100] (2,1) to (3,1) (3,0) to (2,0);
		\Cross[2,3][3,2] \Cross[1,2][2,1]  \draw(1,3)--(1,2) (3,1)--(3,2) (1,1)--(1,0);
	}
	=  t^{-1/2} \ 
	\TikZ{[xscale=-.3, yscale=.3]
		\draw[bend right=100] (1,2) to (2,2) (3,0) to (2,0);
		\Under[3,2][1,0]
	}
%	\label{eiYangBaxter1}\tag{eiYangBaxter1}
\end{equation*}
\begin{equation*}
	\TikZ{[xscale=-.3, yscale=.3]
		\draw[bend right=100] (1,3) to (2,3) (2,2) to (1,2);
		\Cross[3,2][2,1] \Cross[2,1][1,0]  \draw(1,2)--(1,1) (3,1)--(3,0) (3,2)--(3,3);
	}
	=
	\TikZ{[xscale=-.3, yscale=.3]
		\draw[bend right=100] (2,1) to (3,1) (3,0) to (2,0);
		\Cross[3,3][2,2] \Cross[2,2][1,1]  \draw(1,3)--(1,2) (3,1)--(3,2) (1,1)--(1,0);
	}
	= t^{1/2}\ 
	\TikZ{[xscale=-.3, yscale=.3]
		\draw[bend right=100] (1,2) to (2,2) (3,0) to (2,0);
		\Under[3,2][1,0]
	}
	\qquad\qquad 
	\TikZ{[scale=.3]
		\draw[bend right=100] (1,3) to (2,3) (2,2) to (1,2);
		\Cross[3,2][2,1] \Cross[2,1][1,0]  \draw(1,2)--(1,1) (3,1)--(3,0) (3,2)--(3,3);
	}
	=
	\TikZ{[scale=.3]
		\draw[bend right=100] (2,1) to (3,1) (3,0) to (2,0);
		\Cross[3,3][2,2] \Cross[2,2][1,1]  \draw(1,3)--(1,2) (3,1)--(3,2) (1,1)--(1,0);
	}
	= t^{-1/2} \ 
	\TikZ{[scale=.3]
		\draw[bend right=100] (1,2) to (2,2) (3,0) to (2,0);
		\Under[3,2][1,0]
	}
%\label{eiYangBaxter2}\tag{eiYangBaxter2}
\end{equation*}
which are proved by using $T_i^{\pm1} = ae_i+t^{\pm\frac12}$ to expand both sides in
terms of $e_i$.  

When
$a_0 (ae_1) e_0 (ae_1) - \(t_0t^{-1}\) (a e_1) = 0$
and
$a_k (ae_{k-1}) e_k (ae_{k-1}) - \(t_kt^{-1}\) a e_{k-1} = 0$, then
%\begin{align*}
%ae_1 T_0 T_1 
%&= a e_1 (a_0e_0 + t_0^{\frac12}) (ae_1+t^{\frac12}) 
%= a^2a_0 e_1 e_0 e_1 + t_0^{\frac12} a^2 e_1^2 + t^{\frac12} a_0 a e_1 e_0 + t_0^{\frac12}t^{\frac12} a e_1 \\
%&= a (t_0^{-\frac12}t^{\frac12}+t_0^{\frac12}t^{-\frac12}) e_1
%- t_0^{\frac12} a^2 \frac{(t^{\frac12}+t^{-\frac12})}{a} e_1 + t_0^{\frac12}t^{\frac12} a e_1 
%+ t^{\frac12} a_0 a e_1 e_0 \\
%&= a t_0^{-\frac12}t^{\frac12} e_1 + t^{\frac12} a_0 a e_1 e_0 
%= a e_1 t^{\frac12} (t_0^{-\frac12} + a_0 e_0) = a e_1 t^{\frac12} T_0^{-1},
%\end{align*}
%since $T_0^{-1} = T_0 - (t_0^{\frac12} - t_0^{-\frac12}) = a_0 e_0 + t^{\frac12} - t_0^{\frac12}+t_0^{-\frac12} = a_0 e_0 + t_0^{-\frac12}$.  Record the identities
\begin{equation}
\begin{array}{cc}
(ae_1) T_0 T_1 =  t^{\frac12} (ae_1) T_0^{-1},
&  T_1T_0 (ae_1) =  t^{\frac12} T_0^{-1}(ae_1), \\
(ae_{k-1}) T_k T_{k-1} =  t^{\frac12} (ae_{k-1}) T_k^{-1},
&  T_{k-1}T_k (ae_{k-1}) =  t^{\frac12} T_k^{-1}(ae_{k-1}),
\end{array}
\label{poleflips}
\end{equation}
\begin{equation*}
	\TikZ{[xscale=.3, yscale=.3]
		\draw[bend right=100] (1,3) to (2,3) (2,2) to (1,2);
		\Over[-.5,1][1,2]
		\Pole[.15][-.2, 3.2]
			\draw (.3,3+.2) arc (0:360:4pt and 3pt);
			\draw (.3,-.2) arc (0:-180:4pt and 3pt);
		\Over[1,0][2,1]
		\Over[-.5,1][2,0]
		\draw(2,1)--(2,2);
	}
	=t^{1/2} \ \TikZ{[xscale=.3, yscale=.3]
		\draw[bend right=100] (1,3) to (2,3) (2,2) to (1,2);
		\Over[-.5,1][1,0]
		\Pole[.15][-.2, 3.2]
			\draw (.3,3+.2) arc (0:360:4pt and 3pt);
			\draw (.3,-.2) arc (0:-180:4pt and 3pt);
		\Over[-.5,1][1,2]
		\draw(2,0)--(2,2);
	}
%%%%%%%%%%%%%%%
\qquad 
	\TikZ{[xscale=.3, yscale=-.3]
		\draw[bend right=100] (1,3) to (2,3) (2,2) to (1,2);
		\Under[-.5,1][2,0]
		\Pole[.15][-.2, 3.2]
			\draw (.3,-.2) arc (0:360:4pt and 3pt);
			\draw (.3,3+.2) arc (0:180:4pt and 3pt);
		\Over[1,0][2,1]
		\Over[-.5,1][1,2]
		\draw(2,1)--(2,2);
	}
	=t^{1/2} \ \TikZ{[xscale=.3, yscale=-.3]
		\draw[bend right=100] (1,3) to (2,3) (2,2) to (1,2);
		\Under[-.5,1][1,2]
		\Pole[.15][-.2, 3.2]
		\draw (.3,-.2) arc (0:360:4pt and 3pt);
		\draw (.3,3+.2) arc (0:180:4pt and 3pt);	
		\Over[-.5,1][1,0]
		\draw(2,0)--(2,2);
	}
\qquad
%%%%%%%%%%%%%
		\TikZ{[xscale=-.3, yscale=.3]
		\draw[bend right=100] (1,3) to (2,3) (2,2) to (1,2);
		\Under[-.5,1][2,0]
		\Pole[.15][-.2, 3.2]
			\draw (.3,3+.2) arc (0:360:4pt and 3pt);
			\draw (.3,-.2) arc (0:-180:4pt and 3pt);
		\Over[1,0][2,1]
		\Over[-.5,1][1,2]
		\draw(2,1)--(2,2);
	}
	=t^{1/2} \ \TikZ{[xscale=-.3, yscale=.3]
		\draw[bend right=100] (1,3) to (2,3) (2,2) to (1,2);
		\Under[-.5,1][1,2]
		\Pole[.15][-.2, 3.2]
			\draw (.3,3+.2) arc (0:360:4pt and 3pt);
			\draw (.3,-.2) arc (0:-180:4pt and 3pt);
		\Over[-.5,1][1,0]
		\draw(2,0)--(2,2);
	} 
%%%%%%%%%%%%%%%
\qquad 
	\TikZ{[xscale=-.3, yscale=-.3]
		\draw[bend right=100] (1,3) to (2,3) (2,2) to (1,2);
		\Over[-.5,1][1,2]
		\Over[1,0][2,1]
		\Pole[.15][-.2, 3.2]
			\draw (.3,-.2) arc (0:360:4pt and 3pt);
			\draw (.3,3+.2) arc (0:180:4pt and 3pt);
		\Over[-.5,1][2,0]
		\draw(2,1)--(2,2);
	}
	=t^{1/2} \ \TikZ{[xscale=-.3, yscale=-.3]
		\draw[bend right=100] (1,3) to (2,3) (2,2) to (1,2);
		\Over[-.5,1][1,0]
		\Pole[.15][-.2, 3.2]
			\draw (.3,-.2) arc (0:360:4pt and 3pt);
			\draw (.3,3+.2) arc (0:180:4pt and 3pt);
		\Over[-.5,1][1,2]
		\draw(2,0)--(2,2);
	} 
%\label{e1T0T1pictures}\tag{e1T0T1pictures}
\end{equation*}
%Note that
\begin{equation}
%\begin{align*}
(ae_2)T_1T_0T_1T_2
%&= T_1T_2e_1T_2^{-1}T_1^{-1}T_1T_0T_1T_2 
%= T_1 T_2 e_1 T_2^{-1} T_0 T_1 T_2 
%=T_1T_2e_1T_0T_2^{-1}T_1T_2 \\
%&= T_1 T_2 e_1 T_0T_1 T_2T_1^{-1}
%= t^{\frac12} T_1T_2 e_1 T_0^{-1} T_2 T_1^{-1}
%= t^{\frac12} T_1 T_2 e_1 T_2 T_0^{-1}T_1^{-1} \\
%&= t^{\frac12} T_1 T_2 e_1T_2T_1 T_1^{-1} T_0^{-1}T_1^{-1} 
%= a t T_1 T_2 e_1e_2 T_1^{-1} T_0^{-1}T_1^{-1} 
%= a^2t^{\frac32} e_2e_1e_2 T_1^{-1} T_0^{-1}T_1^{-1}  \\
%&
= t^{\frac32} (ae_2)T_1^{-1} T_0^{-1}T_1^{-1},
%\end{align*}
\label{oddleftpoleflip}
\end{equation}
$$	\TikZ{[xscale=.3, yscale=.3]
		\draw[bend right=100] (2,3) to (3,3) (3,2) to (2,2);
		\Over[-.5,1.25][2,2]
		\Pole[.15][-.2, 3.2]
			\draw (.3,3+.2) arc (0:360:4pt and 3pt);
			\draw (.3,-.2) arc (0:-180:4pt and 3pt);
		\draw[over](1,0)--(1,3);
		\Over[2,0][3,1]
		\Over[-.5,1.25][3,0]
		\draw(3,1)--(3,2);
	}
	=t^{3/2} \ \TikZ{[xscale=.3, yscale=.3]
		\draw[bend right=100] (2,3) to (3,3) (3,2) to (2,2);
		\Over[-.5,1][2,0]
		\Pole[.15][-.2, 3.2]
			\draw (.3,3+.2) arc (0:360:4pt and 3pt);
			\draw (.3,-.2) arc (0:-180:4pt and 3pt);
		\draw[over](1,0)--(1,3);
		\Over[-.5,1][2,2]
		\draw(3,0)--(3,2);
	}
$$
%Note also that
%since $T_0 = a_0e_0+t_0^{\frac12}$ then
%\begin{align*}
%e_1T_0e_1 &= a_0 e_1e_0e_1 + t_0^{\frac12} e_1^2
%= a_0 \frac{\(t_0t^{-1}\)}{a a_0} e_1 +t_0^{\frac12}\frac{-\(t\)}{a} e_1
%= \frac{1}{a} (t_0^{\frac12}t^{-\frac12}+t_0^{-\frac12}t^{\frac12}-t_0^{\frac12}t^{\frac12}
%-t_0^{\frac12}t^{-\frac12}) e_1 \\
%&= \frac{-t^{\frac12}}{a}(t_0^{\frac12}-t_0^{-\frac12}) e_1,
%\end{align*}
%then
\begin{eqnarray}
(ae_1)T_0(ae_1) = -t^{\frac12} (t_0^{\frac12}-t_0^{-\frac12}) (ae_1),
&\qquad 
(ae_{k-1})T_k (ae_{k-1}) = -t^{\frac12} (t_k^{\frac12}-t_k^{-\frac12}) (ae_{k-1}),
\label{polebubbles}\\
	\TikZ{[xscale=.3, yscale=.3]
		\draw[bend right=100] (1,3) to (2,3) (2,2) to (1,2);
		\draw[bend right=100] (1,1) to (2,1) (2,0) to (1,0);
		\Over[-.5,1.5][1,2]
		\Pole[.15][-.2, 3.2]
			\draw (.3,3+.2) arc (0:360:4pt and 3pt);
			\draw (.3,-.2) arc (0:-180:4pt and 3pt);
		\Over[-.5,1.5][1,1]
		\draw(2,1)--(2,2);
	}
	=-t^{\frac12} (t_0^{\frac12}-t_0^{-\frac12}) \ \TikZ{[xscale=.3, yscale=.3]
		\draw[bend right=100] (1,2) to (2,2) (2,0) to (1,0);
		\Pole[.15][-.2, 2.2]
			\draw (.3,2+.2) arc (0:360:4pt and 3pt);
			\draw (.3,-.2) arc (0:-180:4pt and 3pt);
	}
	&\qquad\qquad
	\TikZ{[xscale=-.3, yscale=-.3]
		\draw[bend right=100] (1,3) to (2,3) (2,2) to (1,2);
		\draw[bend right=100] (1,1) to (2,1) (2,0) to (1,0);
		\Over[-.5,1.5][1,2]
		\Pole[.15][-.2, 3.2]
			\draw (.3,3+.2) arc (0:360:4pt and 3pt);
			\draw (.3,-.2) arc (0:-180:4pt and 3pt);
		\Over[-.5,1.5][1,1]
		\draw(2,1)--(2,2);
	}
	=-t^{\frac12} (t_k^{\frac12}-t_k^{-\frac12}) \ 
	\TikZ{[xscale=-.3, yscale=.3]
		\draw[bend right=100] (1,2) to (2,2) (2,0) to (1,0);
		\Pole[.15][-.2, 2.2]
			\draw (.3,2+.2) arc (0:360:4pt and 3pt);
			\draw (.3,-.2) arc (0:-180:4pt and 3pt);
	}\nonumber
	\end{eqnarray}
and
%Using $T_1^{-1} = ae_1+t^{-\frac12}$,
%$$e_0T_1^{-1}e_0 = e_0 (ae_1+t^{-\frac12})e_0 = a e_0e_1e_0 + t^{-\frac12}e_0^2,$$
%Using this and  $e_1T_0^{-1}T_1^{-1} = t^{-\frac12}e_1T_0$ and $T_0e_0 = -t_0^{-\frac12}$ 
%and $e_0T_0^{-1} = -t_0^{\frac12}e_0$,
%and
%\begin{equation}
\begin{align*}
e_0 T_1^{-1} T_0^{-1} T_1^{-1}  e_0 
%&= e_0 (a e_1 + t^{-\frac12}) T_0^{-1} T_1^{-1} e_0 \\
%&= e_0 (a e_1 t^{-\frac12} T_0 ) e_0 + t^{-\frac12} e_0 T_0^{-1} T_1^{-1} e_0 \\
%&= t^{-\frac12} a ( - t_0^{-\frac12}) e_0e_1 e_0 + t^{-\frac12} (-t_0^{\frac12}) e_0 T_1^{-1} e_0 \\
%&= - t_0^{-\frac12} t^{-\frac12} a e_0 e_1 e_0 - t_0^{\frac12} t^{-\frac12} e_0 ( ae_1 + t^{-\frac12}) e_0 \\
%&= -t^{-\frac12} (t_0^{-\frac12}+t_0^{\frac12}) a e_0 e_1 e_0 - t^{-1} t_0^{\frac12} e_0^2 \\
= - t^{-\frac12}[\![t_0]\!] e_0 (ae_1) e_0  - t^{-1} t_0^{\frac12}e_0^2. % \\
%&= -t^{-\frac12}[\![t_0]\!] a e_0 e_1 e_0 + t^{-1} t_0^{\frac12} 
%\frac{ [\![ t_0]\!]}{a_0} e_0.
\end{align*}
%\label{leftresolve}
%\end{equation}
$$
\TikZ{[yscale=.3, xscale=.3]
	\PoleCaps[.15][0,3]
	\Over[2,.75][-.5,1.5]
	\Pole[.15][0,3]
	\draw[over] (1,1)--(1,2);
	\foreach \y in {3,1}{
		\draw [over, bend left=75] (1,\y) to (1-.4,\y-.3)  (1-.4,\y-.7)  to (1,\y-1) ;
		\draw[densely dotted]  (1-.4,\y)--(1-.4,\y-1) ; \filldraw [black] (1-.4,\y-.3) circle (2pt); \filldraw [black] (1-.4,\y-.7) circle (2pt);}
	\Over[2,2.25][-.5,1.5]
	\draw (2,2.25) -- (2,3) (2,0) -- (2,.75);
	\foreach \y in {0,3}{\foreach \x in {1,2} {\node[V] at (\x,\y){};}}
	}
	=- t^{-\frac12}[\![t_0]\!] \ 
	\TikZ{[yscale=.3, xscale=.3]
	\draw[densely dotted] (.5,0) to (.5,2.5);
	\foreach \x in {1,2}{\node[V] (b\x) at (\x, 0) {}; \node[V] (t\x) at (\x, 2.5){};}
	\foreach \y in {1,2, 3, 4}{\node[V]  (l\y) at (.5, \y*.5){}; }
	\draw[bend right=30] (l4) to (t1) (l3) to (t2) (b1) to (l1)  (b2) to (l2);
	}
- t^{-1} t_0^{\frac12}\TikZ{[yscale=.3, xscale=.3]
	\draw[densely dotted] (.25,0) to (.25,2.5);
	\foreach \x in {1,2}{\node[V] (b\x) at (\x, 0) {}; \node[V] (t\x) at (\x, 2.5){};}
	\foreach \y in {1,2, 3, 4}{\node[V]  (l\y) at (.25, \y*.5){}; }
		\draw[bend right=30] (l4) to (t1) (b1) to (l1);
		\draw  (l2) to ++(.5,0) to [bend right] ++(0,.5) to (l3);
		\draw (t2) to (b2);
	}
$$

\subsection{$TL_k$ as a diagram algebra}

 Using the pictorial notation, 
 the algebra $TL_k$ has a basis (see \cite[Theorem 3.4]{GMP12})
 of non-crossing diagrams with $k$ dots in the top row,
$k$ dots in the bottom row, edges connecting pairs of dots, an even number of left boundary to right boundary edges, and
$$
(-1)^{\#\{\mathrm{left\ boundary\ edges}\}}=1 
\qquad\hbox{and}\qquad
(-1)^{\#\{\mathrm{right\ boundary\ edges}\}}=1.
$$
For example, 
$$d_1 = 
\TikZ{[xscale=.5, yscale=.75]
\draw[densely dotted] (.5,0) rectangle (5.5,1.5);
\foreach \x in {1,...,5}{\node[V] (b\x) at (\x, 0) {}; \node[V] (t\x) at (\x, 1.5){};}
\foreach \y in {1,2}{\coordinate (l\y) at (.5, \y*.5); \coordinate (r\y) at (5.5, \y*.5); }
	\draw[bend right=40] 
		(t1) to  (t4) 
		(t2) to (t3)
		(b5) to (b4);
	\draw[bend left=15] 
		(b3) to (r1) node[V]{} 
		(b2) to (r2) node[V]{};
	\draw (b1)--(t5);
}
\qquad \text{ and } \qquad 
d_2= 
\TikZ{[xscale=.5, yscale=.75]
\draw[densely dotted] (.5,0) rectangle (5.5,1.5);
\foreach \x in {1,...,5}{\node[V] (b\x) at (\x, 0) {}; \node[V] (t\x) at (\x, 1.5){};}
\foreach \y in {1,2}{\coordinate (l\y) at (.5, \y*.5); \coordinate (r\y) at (5.5, \y*.5); }
	\draw[bend right=40] 
		(t4) to  (t5) 
		(b5) to (b4)
		(b3) to (b2);
	\draw[bend left=15] 
		(r1) node[V]{} to (t2)
		(r2) node[V]{} to (t3) 
		(l1) node[V]{} to (b1)
		 (t1) to (l2) node[V]{};
}
$$
are both basis elements of $TL_k$. Multiplication of basis elements can be computed pictorially by vertical concatenation, with self-connected loops and strands with both ends on the left or on the right replaced by constant coefficients according to the following local rules:
$$\TikZ{[scale=.3]
	 \draw (0,0) circle (20pt);
}  = -\(t\), \qquad 
\TikZ{[scale=.3]
	\draw[bend left=100] (0,5.5)node[V]{}  to (0,3.5) node[V]{};
	\draw[densely dotted] (1,0)--(0,0)--(0,1) (0,2)--(0,4) (0,5)--(0,6);
	\foreach\y in {1.5, 4.5}{\foreach \x in {-1,0,1}{
		\node[draw, fill =black, circle, inner sep=0pt, minimum size=.5pt] at (0, \y+.25*\x) {};}}
	\foreach \y in {.5, 2.5}{\node[V] at (0,\y){};}
	\draw[|-|] (-.5,2.9) to node[midway, left, align=center]{\tiny if even \#\\ \tiny connections} (-.5,.1);
} =   \frac{\(t_0 t^{-1}\)}{a_0}, \qquad 
\TikZ{[scale=.3]
	\draw[bend right=100] (0,5.5)node[V]{}  to (0,3.5) node[V]{};
	\draw[densely dotted] (-1,0)--(0,0)--(0,1) (0,2)--(0,4) (0,5)--(0,6);
	\foreach\y in {1.5, 4.5}{\foreach \x in {-1,0,1}{
		\node[draw, fill =black, circle, inner sep=0pt, minimum size=.5pt] at (0, \y+.25*\x) {};}}
	\foreach \y in {.5, 2.5}{\node[V] at (0,\y){};}
	\draw[|-|] (.5,2.9) to node[midway, right, align=center]{\tiny if even \#\\ \tiny connections} (.5,.1);
} =   \frac{\(t_kt^{-1}\)}{a_k}, 
 $$
 $$\TikZ{[scale=.3]
	\draw[bend left=100] (0,5.5)node[V]{}  to (0,3.5) node[V]{};
	\draw[densely dotted] (1,0)--(0,0)--(0,1) (0,2)--(0,4) (0,5)--(0,6);
	\foreach\y in {1.5, 4.5}{\foreach \x in {-1,0,1}{
		\node[draw, fill =black, circle, inner sep=0pt, minimum size=.5pt] at (0, \y+.25*\x) {};}}
	\foreach \y in {.5, 2.5}{\node[V] at (0,\y){};}
	\draw[|-|] (-.5,2.9) to node[midway, left, align=center]{\tiny if odd \#\\ \tiny connections} (-.5,.1);
} =   \frac{-\(t_0\)}{a_0}, 
	\qquad \text{ and } \qquad 
\TikZ{[scale=.3]
	\draw[bend right=100] (0,5.5)node[V]{}  to (0,3.5) node[V]{};
	\draw[densely dotted] (-1,0)--(0,0)--(0,1) (0,2)--(0,4) (0,5)--(0,6);
	\foreach\y in {1.5, 4.5}{\foreach \x in {-1,0,1}{
		\node[draw, fill =black, circle, inner sep=0pt, minimum size=.5pt] at (0, \y+.25*\x) {};}}
	\foreach \y in {.5, 2.5}{\node[V] at (0,\y){};}
	\draw[|-|] (.5,2.9) to node[midway, right, align=center]{\tiny if odd \#\\ \tiny connections} (.5,.1);
} =   \frac{-\(t_k\)}{a_k}.
$$
For example with $d_1$ and $d_2$ as above, 
$$d_1 d_2 = \TikZ{[xscale=.5, yscale=.75]
\draw[densely dotted] (.5,0) rectangle (5.5,1.5);
\foreach \x in {1,...,5}{\node[V] (b\x) at (\x, 0) {}; \node[V] (t\x) at (\x, 1.5){};}
\foreach \y in {1,2}{\coordinate (l\y) at (.5, \y*.5); \coordinate (r\y) at (5.5, \y*.5); }
	\draw[bend right=40] 
		(t1) to  (t4) 
		(t2) to (t3)
		(b5) to (b4);
	\draw[bend left=15] 
		(b3) to (r1) node[V]{} 
		(b2) to (r2) node[V]{};
	\draw (b1)--(t5);
\ShiftY{-1.5cm}
\draw[densely dotted] (.5,0) rectangle (5.5,1.5);
\foreach \x in {1,...,5}{\node[V] (b\x) at (\x, 0) {}; \node[V] (t\x) at (\x, 1.5){};}
\foreach \y in {1,2}{\coordinate (l\y) at (.5, \y*.5); \coordinate (r\y) at (5.5, \y*.5); }
	\draw[bend right=40] 
		(t4) to  (t5) 
		(b5) to (b4)
		(b3) to (b2);
	\draw[bend left=15] 
		(r1) node[V]{} to (t2)
		(r2) node[V]{} to (t3) 
		(l1) node[V]{} to (b1)
		 (t1) to (l2) node[V]{};
}=
\TikZ{[xscale=.5, yscale=.75]
	\draw[densely dotted] (.5,0) rectangle (5.5,3);
	\foreach \x in {1,...,5}{\node[V] (b\x) at (\x, 0) {}; \node[V] (t\x) at (\x, 3){};}
	\foreach \y in {1,2,3,4}{\coordinate (l\y) at (.5, \y*.6); \coordinate (r\y) at (5.5, \y*.6); }
	\draw[bend right=40] 
		(t1) to  (t4) 
		(t2) to (t3)
		(b5) to (b4)
		(b3) to (b2);
	\draw[bend left=15] 
		(l1) node[V]{} to (b1)
		 (t5) to (l3) node[V]{};
	\draw[thick, dashed]	(r4)node[V]{} .. controls (2,2) and (2,1).. (r1)node[V]{};
	\draw[very thick]	(r3)node[V]{}  .. controls (3,1.9) and (3,1.1)..  (r2)node[V]{};
	\draw (4.75,1.5) ellipse (.5cm and .15cm);
}
= (-\(t\) ) \Big(\frac{-\(t_k\)}{a_k}\Big) \Big(\frac{\(t_kt^{-1}\)}{a_k}\Big) 
\TikZ{[xscale=.5, yscale=.75]
\draw[densely dotted] (.5,0) rectangle (5.5,1.5);
\foreach \x in {1,...,5}{\node[V] (b\x) at (\x, 0) {}; \node[V] (t\x) at (\x, 1.5){};}
\foreach \y in {1,2}{\coordinate (l\y) at (.5, \y*.5); \coordinate (r\y) at (5.5, \y*.5); }
\draw[bend right=40] 
		(t1) to  (t4) 
		(t2) to (t3)
		(b5) to (b4)
		(b3) to (b2);
\draw[bend left=25] 
	(l1) node[V]{} to (b1)
	 (t5) to (l2) node[V]{};
}$$
(where the dashed strand is removed with a coefficient of $\frac{\(t_k t^{-1}\)}{a_k}$, and the thick strand is removed with a coefficient of $\frac{-\(t_k\)}{a_k}$).

\subsection{The through-strand filtration of $TL_k$}

A \emph{through-strand} is an edge that connects a top vertex to a bottom vertex. Define the ideals
$$TL_k^{(\le j)}=\hbox{$\CC$-span}\{ \hbox{diagrams with $\le j$ through-strands}\}.$$
Then the algebra $\zTL_k$ is filtered by ideals as 
\begin{equation}
TL_k=TL_k^{(\le k)}\supseteq TL_k^{(\le k-1)}\supseteq \cdots \supseteq TL_k^{(\le 1)}\supseteq TL_k^{(\le 0)} \supseteq 0.
\label{TLfilt}
\end{equation}
If
$$TL_k^{(j)} = \frac{TL_k^{(\le j)}}{TL_k^{(\le j-1)}},
\qquad\hbox{then}\quad
\dim(TL_k^{(j)})<\infty,\ \ \hbox{for $j\ge 1$,}
\qquad\hbox{and}\quad \dim(TL_k^{(\le 0)})=\infty,$$
as there can be an arbitrarily large number of edges which connect the left and right sides in diagrams with no through strands:
$$
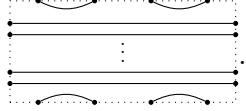
\TikZ{[xscale=.75]
\draw[dotted] (.5,0) rectangle (4.5,1.35);
\foreach \x in {1,...,4}{\node[V] (b\x) at (\x, 0) {}; \node[V] (t\x) at (\x, 1.35){};}
\foreach \x/\y in {1/.25,2/.4,4/.9,5/1.05}{\node[V] (l\x) at (.5, \y){}; \node[V] (r\x) at (4.5, \y){}; \draw (l\x)--(r\x);}
\draw[bend right=20] 
		(t1) to  (t2)	(t3) to  (t4)	(b4) to  (b3)	(b2) to  (b1) ;
\node at (2.5,.75){\tiny $\vdots$};
}.$$

\subsection{The elements $I_1$ and $I_2$}

As in \cite[\S 3.2]{GN}, define 
\begin{equation}\label{eq:defI1}
   I_1 = \begin{cases}
		(ae_1) (ae_3)\cdots (ae_{k-1}), &\text{if $k$ is even,} \\
		(ae_1) (ae_3)\cdots (ae_{k-2}) e_k, &\text{if $k$ is odd,}
	\end{cases}
   =
	\begin{cases}
	\TikZ{[scale=.5, yscale=.9]
		\draw[dotted] (.5,0) rectangle (7.5,1.2);
		\foreach \x in {1,2,3,4,6,7}{\node[V] (b\x) at (\x, 0) {}; \node[V] (t\x) at (\x, 1.2){};}
		\foreach \x/\y in {1/2, 3/4, 6/7}{\draw[bend right=60] (t\x) to (t\y);\draw[bend right=60] (b\y) to (b\x);}
		\node at (5,.6){$\cdots$};
	}
	&\text{if $k$ is even,}\\
	\TikZ{[scale=.5, yscale=.9]
		\draw[dotted] (.5,0) rectangle (8.5,1.2);
		\foreach \x in {1,2,3,4,6,7,8}{\node[V] (b\x) at (\x, 0) {}; \node[V] (t\x) at (\x, 1.2){};}
		\foreach \y in {1,2}{\coordinate (l\y) at (.5, \y*.5-.15); \coordinate (r\y) at (8.5, \y*.5-.15); }
		\foreach \x/\y in {1/2, 3/4, 6/7}{\draw[bend right=60] (t\x) to (t\y);\draw[bend right=60] (b\y) to (b\x);}
		\draw[bend right=30] (t8) to (r2) node[V]{}	(r1) node[V]{} to (b8);
		\node at (5,.6){$\cdots$};
	}
	&\text{if $k$ is odd,}
	\end{cases}
\end{equation}
and
\begin{equation}\label{eq:defI2}
    I_2 = \begin{cases}
		e_0 (ae_2) \cdots (ae_{k-2}) e_k, &\hbox{if $k$ is even,} \\
		e_0 (ae_2) \cdots (ae_{k-1}), &\hbox{if $k$ is odd.}
	\end{cases}
    =
	\begin{cases}\TikZ{[scale=.5, yscale=.9]
		\draw[dotted] (.5,0) rectangle (7.5,1.2);
		\foreach \x in {1,2,3,5,6,7}{\node[V] (b\x) at (\x, 0) {}; \node[V] (t\x) at (\x, 1.2){};}
		\foreach \x/\y in {2/3, 5/6}{\draw[bend right=60] (t\x) to (t\y);\draw[bend right=60] (b\y) to (b\x);}
		\foreach \y in {1,2}{\coordinate (l\y) at (.5, \y*.5-.15); \coordinate (r\y) at (7.5, \y*.5-.15); }
		\draw[bend right=30] (l2)node[V]{} to (t1) 	(b1)  to (l1) node[V]{} 	(t7) to (r2) node[V]{}	(r1) node[V]{} to (b7);
		\node at (4,.6){$\cdots$};
	}
	&\text{if $k$ is even,}\\
	\TikZ{[scale=.5, yscale=.9]
		\draw[dotted] (.5,0) rectangle (8.5,1.2);
		\foreach \x in {1,2,3,5,6,7,8}{\node[V] (b\x) at (\x, 0) {}; \node[V] (t\x) at (\x, 1.2){};}
		\foreach \y in {1,2}{\coordinate (l\y) at (.5, \y*.5-.15); \coordinate (r\y) at (8.5, \y*.5-.15); }
		\foreach \x/\y in {2/3, 5/6,7/8}{\draw[bend right=60] (t\x) to (t\y);\draw[bend right=60] (b\y) to (b\x);}
		\draw[bend right=30] (l2)node[V]{} to (t1) 	(b1)  to (l1) node[V]{} ;
		\node at (4,.6){$\cdots$};
	}
	&\text{if $k$ is odd.}
\end{cases}
\end{equation}
Up to a constant multiple the elements $I_1$ and $I_2$ are idempotents and
\begin{equation*}\label{eq:I121}
    I_1 I_2 I_1 = \begin{cases}
	\TikZ{[scale=.4, yscale=.9]
		\draw[dotted] (.5,0) rectangle (7.5,2);
		\foreach \x in {1,2,3,4,6,7}{\node[V] (b\x) at (\x, 0) {}; \node[V] (t\x) at (\x, 2){};}
		\foreach \y in {1,2,3,4}{\coordinate (l\y) at (.5, \y*.4){}; \coordinate(r\y) at (7.5, \y*.4){}; }
		\foreach \x/\y in {1/2, 3/4, 6/7}{\draw[bend right=60] (t\x) to (t\y);\draw[bend right=60] (b\y) to (b\x);}
		\draw (r3) node[V] {}--(l3) node[V] {} (r2) node[V] {} --(l2)node[V] {};
		\node at (5,.2){$\cdots$};	\node at (5,1.8){$\cdots$};
	}
	&\text{if $k$ is even,}\\
	\TikZ{[scale=.4, yscale=.9]
		\draw[dotted] (.5,0) rectangle (8.5,2);
		\foreach \x in {1,2,3,4,6,7,8}{\node[V] (b\x) at (\x, 0) {}; \node[V] (t\x) at (\x, 2){};}
		\foreach \y in {1,2,3,4}{\coordinate (l\y) at (.5, \y*.4){}; \coordinate(r\y) at (8.5, \y*.4){}; }
		\foreach \x/\y in {1/2, 3/4, 6/7}{\draw[bend right=60] (t\x) to (t\y);\draw[bend right=60] (b\y) to (b\x);}
		\draw[bend right=30] (t8) to (r4) node[V]{}	(r1) node[V]{} to (b8);
		\draw (r3) node[V] {}--(l3) node[V] {} (r2) node[V] {} --(l2)node[V] {};
		\node at (5,.2){$\cdots$};	\node at (5,1.8){$\cdots$};
	}&\text{if $k$ is odd,}
	\end{cases}
\qquad
%\end{equation}
%and 
%\begin{equation}\label{def:I212}
    I_2 I_1 I_2 = \begin{cases}
	\TikZ{[scale=.4, yscale=.9]
		\draw[dotted] (.5,0) rectangle (7.5,2);
		\foreach \x in {1,2,3,5,6,7}{\node[V] (b\x) at (\x, 0) {}; \node[V] (t\x) at (\x, 2){};}
		\foreach \x/\y in {2/3, 5/6}{\draw[bend right=60] (t\x) to (t\y);\draw[bend right=60] (b\y) to (b\x);}
		\foreach \y in {1,2,3,4}{\coordinate (l\y) at (.5, \y*.4); \coordinate (r\y) at (7.5, \y*.4); }
		\draw[bend right=30] (l4)node[V]{} to (t1) 	(b1)  to (l1) node[V]{} 	(t7) to (r4) node[V]{}	(r1) node[V]{} to (b7);
		\draw (r3) node[V] {}--(l3) node[V] {} (r2) node[V] {} --(l2)node[V] {};
		\node at (4,.2){$\cdots$};	\node at (4,1.8){$\cdots$};
	}
	&\text{if $k$ is even,}\\
	\TikZ{[scale=.4, yscale=.9]
		\draw[dotted] (.5,0) rectangle (8.5,2);
		\foreach \x in {1,2,3,5,6,7,8}{\node[V] (b\x) at (\x, 0) {}; \node[V] (t\x) at (\x, 2){};}
		\foreach \y in {1,2,3,4}{\coordinate (l\y) at (.5, \y*.4); \coordinate (r\y) at (8.5, \y*.4); }
		\foreach \x/\y in {2/3, 5/6,7/8}{\draw[bend right=60] (t\x) to (t\y);\draw[bend right=60] (b\y) to (b\x);}
		\draw[bend right=30] (l4)node[V]{} to (t1) 	(b1)  to (l1) node[V]{} ;
		\draw (r3) node[V] {}--(l3) node[V] {} (r2) node[V] {} --(l2)node[V] {};
		\node at (4,.2){$\cdots$};	\node at (4,1.8){$\cdots$};
	}
	&\text{if $k$ is odd.}
	\end{cases}
\end{equation*}
Proposition \ref{IIIexpansion} gives another striking formula for the elements $I_1I_2I_1$ and $I_2I_1I_2$.

\subsection{The element $Z I_1$ in $TL_k$}

Conceptually, the diagram
$${\def\K{5}
F =
\TikZ{[scale=.3, yscale=.9]
	\node at (3.5,1.5) {\tiny$\cdots$};
	\PoleCaps[.15,\K+.85][0,3]
	\Pole[.15][1.5,3]\Pole[\K+.85][1.5,3][\K]
	\foreach \x in {1,2,5}{
			\draw[over] (\x,3) to (\x,0);}
	\draw[over, rounded corners] 
		(2.75,1) to (1,1)  to [bend left=10]  (-.7,1.5) to [bend left=10] (1,2) to (2.75,2) 
		(\K-.75,2) to (\K,2) to [bend left=10] (\K+1.7,1.5) to [bend left=10] (\K,1) to (\K-.75,1);
	\Pole[.15][0,1.5]\Pole[\K+.85][0,1.5][\K]
	\foreach \x in {1,2,5}{
			\draw[over] (\x,1.5) to (\x,0);}
	\foreach \y in {0,3}{\foreach \x in {1, 2, 5} {\node[V] at  (\x,\y){};}}
	}
}
\quad\hbox{would be a central element of $H_k$}
$$
if it represented a true element of the algebra $H_k$.
Though the diagram $F$ does not naturally represent an element of $H_k$, the diagrams
$$
{\def\K{8}
D^{\mathrm{even}} =
I_1 \big(T_0^{-1}(ae_2)(ae_4)\cdots (ae_{k-2})T_k\big)I_1 =
 \TikZ{[scale=.3]
	\node at (5.5,1.5) {\tiny$\cdots$};
	\PoleCaps[.15,\K+.85][0,3]
	\Over[1,1][-.5,1.5]\Over[\K,1][\K+1.5,1.5]
	\Pole[.15][0,3]\Pole[\K+.85][0,3][\K]
	\Over[1,2][-.5,1.5]\Over[\K,2][\K+1.5,1.5]
	\foreach \x in {1,3,7}{
			\draw[bend right=100] (\x,3) to (\x+1,3) (\x+1,0) to (\x,0);}
	\draw (1,2) to[bend left] (1.5,2.2)--(5,2.2) (6,2.2)--(7.5,2.2) to [bend left] (8,2);
	\draw (1,1) to[bend right] (1.5,.8)--(5,.8) (6,.8)--(7.5,.8) to [bend right] (8,1);
	\foreach \y in {0,3}{\foreach \x in {1,...,4,7,\K} {\node[V] at  (\x,\y){};}}
	}\ , \quad \text{and}
}
$$
$$
{\def\K{9}
D^{\mathrm{odd}} = I_2\big(T_1^{-1}T_0^{-1}T_1^{-1}(ae_3)(ae_5)\cdots (ae_{k-2})T_k\big)I_2
= 
	\TikZ{[yscale=.3, xscale=-.3]
	\node at (5.5,1.5) {\tiny$\cdots$};
	\PoleCaps[.15,\K+.85][0,3]
	\Over[1,1][-.5,1.5]\Over[\K-1,1][\K+1.5,1.5]
	\Pole[.15][0,3]\Pole[\K+.85][0,3][\K]
	\Over[1,2][-.5,1.5]
	\draw[over] (9,1)--(9,2);
	\Over[\K-1,2][\K+1.5,1.5]
	\foreach \x in {1,3,7}{
			\draw[bend right=100] (\x,3) to (\x+1,3) (\x+1,0) to (\x,0);}
	\foreach \y in {2,0}{\Ek[\y][9]}
	\draw (1,2) to[bend left] (1.5,2.2)--(5,2.2) (6,2.2)--(7.5,2.2) to [bend left] (8,2);
	\draw (1,1) to[bend right] (1.5,.8)--(5,.8) (6,.8)--(7.5,.8) to [bend right] (8,1);
	\foreach \y in {0,3}{\foreach \x in {1,...,4,7,8,9} {\node[V] at  (\x,\y){};}}
	}
}
$$
do appear in the algebra $TL_k$ and play an important role in the proof of the following theorem. See also \cite[Thm.\ 4.1]{GN}, using Remark \ref{rk:conversion-to-GN} below as a guide.

\begin{thm}  \label{IIIexpansion}
Let
$Z = W_1+W_1^{-1}+\cdots + W_k+W_k^{-1}$ which, as noted in \eqref{Zdefn}
is a central element of $H_k$.  As elements of $TL_k$,
\begin{align*}
&\hbox{if $k$ is even, then}\quad
D^{\mathrm{even}} 
= a_0a_k I_1 I_2 I_1 + [\![ t_0t_kt^{-1}]\!] I_1
\quad\hbox{and}\quad
ZI_1 = [k]D^{\mathrm{even}},\quad \hbox{and} \\
&\hbox{if $k$ is odd, then}\quad
D^{\mathrm{odd}} = t^{-\frac12} \Big( \frac{-\(t_0\)}{a_0} \Big) \left(a_0a_k I_2I_1I_2 -\(t_0 t_k^{-1}\) I_2\right)
\quad\hbox{and}\quad
t^{-\frac12}\Big(\frac{-\(t_0\)}{a_0}\Big) ZI_2 = [k]D^{\mathrm{odd}}.
\end{align*}
\end{thm}
\begin{proof}
\textbf{Case: $k$ even.}
Let 
\begin{align*}
L^{\mathrm{even}} 
&= I_1 \big((ae_2)(ae_4)\cdots (ae_{k-2})e_k\big) I_1
= {\def\K{8}
	\TikZ{[scale=.3]
		\foreach \y in {0,3}{\foreach \x in {1,...,4,7,\K} {\node[V] at (\x,\y){};}}
		\foreach \y in {0,1}{ \node[V] (\K\y) at (\K+.5,\y+1){};}
		\foreach \x in {0,\K}{	\draw[densely dotted] (\x+.5,0)--(\x+.5,3);}
		\foreach \x in {1,3,7}{	\draw[bend right=100] (\x,3) to (\x+1,3) (\x+1,0) to (\x,0);}
		\draw (\K0) to ++(-\K+1,0) to [bend left] ++(0,1) to (\K1);
		\node at (5.5,2.75) {\tiny$\cdots$};
		\node at (5.5,.25) {\tiny$\cdots$};
	} 
}
= \Big( \frac{\(t_kt^{-1}\)}{a_k}\Big) I_1, \\
M^{\mathrm{even}} 
&= I_1 \big(e_0 (ae_2)(ae_4)\cdots (ae_{k-2})\big) I_1 
= {\def\K{8}
	\TikZ{[scale=.3]
		\foreach \y in {0,3}{\foreach \x in {1,...,4,7,\K} {\node[V] at (\x,\y){};}}
		\foreach \y in {0,1}{ \node[V] (0\y) at (0+.5,\y+1){};}
		\foreach \x in {0,\K}{	\draw[densely dotted] (\x+.5,0)--(\x+.5,3);}
		\foreach \x in {1,3,7}{	\draw[bend right=100] (\x,3) to (\x+1,3) (\x+1,0) to (\x,0);}
		\draw (00) to ++(\K-1,0) to [bend right] ++(0,1) to (01);
		\node at (5.5,2.75) {\tiny$\cdots$};
		\node at (5.5,.25) {\tiny$\cdots$};
	}
} = \Big( \frac{\(t_0t^{-1}\)}{a_0} \Big) I_1,
\quad\hbox{and} \\ 
P^{\mathrm{even}} 
&= I_1 \big( (ae_2) (ae_4) \cdots (ae_{k-2}) \big) I_1 
= {\def\K{8}
		\TikZ{[scale=.3]
		\foreach \y in {0,3}{\foreach \x in {1,...,4,7,\K} {\node[V] at (\x,\y){};}}
		\foreach \x in {0,\K}{	\draw[densely dotted] (\x+.5,0)--(\x+.5,3);}
		\foreach \x in {1,3,7}{	\draw[bend right=100] (\x,3) to (\x+1,3) (\x+1,0) to (\x,0);}
		\draw (1.5,1) to ++(\K-2,0) to [bend right] ++(0,1) to ++(-\K+2,0) to [bend right] (1.5,1);
		\node at (5.5,2.75) {\tiny$\cdots$};
		\node at (5.5,.25) {\tiny$\cdots$};
	} %\Bigg)
}
= -\(t\) I_1.
\end{align*}
Using $T_0^{-1}= a_0e_0+t_0^{-\frac12}$ for the left pole and $T_k = a_ke_k+t_k^{\frac12}$ for the right pole,
%{\def\K{8}
\begin{align*}
D^{\mathrm{even}} 
&= a_0a_k I_1I_2I_1 %ACROSS
%	\TikZ{[scale=.3]
%		\foreach \y in {0,3}{\foreach \x in {1,...,4,7,\K} {\node[V] at (\x,\y){};}}
%		\foreach \x in {0,\K}{\foreach \y in {0,1}{ \node[V] (\x\y) at (\x+.5,\y+1){};}
%			\draw[densely dotted] (\x+.5,0)--(\x+.5,3);}
%		\foreach \x in {1,3,7}{	\draw[bend right=100] (\x,3) to (\x+1,3) (\x+1,0) to (\x,0);}
%		\draw (00)--(\K0) (01)--(\K1);
%		\node at (5.5,2.75) {\tiny$\cdots$};
%		\node at (5.5,.25) {\tiny$\cdots$};
%	}  
+ a_0 t_k^{\frac12} M^{\mathrm{even}} %LEFT
+ a_k t_0^{-\frac12} L^{\mathrm{even}} %RIGHT
+ t_0^{\frac12}t_k^{\frac12} P^{\mathrm{even}}%BLOB 
\\
&= a_0a_k I_1I_2I_1
+ (t_k^{\frac12}\(t_0t^{-1}\) +  t_0^{-\frac12}\(t_kt^{-1}\) - t_0^{-\frac12}t_k^{\frac12} [\![t]\!] ) I_1 \\
&= a_0a_k I_1 I_2 I_1 + [\![ t_0t_kt^{-1}]\!] I_1,
\end{align*}
% }
%since $(a a_0 t_k^{\frac12} + a a_k t_0^{-\frac12} - t_0^{-\frac12}t_k^{\frac12} [\![t]\!] )
%= [\![t_0t^{-1}]\!] t_k^{\frac12} + [\![t_kt^{-1}]\!] t_0^{-\frac12} 
%- t_0^{-\frac12}t_k^{\frac12}(t^{\frac12}+t^{-\frac12}) = [\![ t_0t_kt^{-1}]\!]$.
which completes the proof of the first statement.

Using $(ae_1)T_1^{-1} = (-t^{\frac12})(ae_1)$ and $(ae_1)T_1T_0(ae_1) = t^{\frac12}(ae_1)T_0^{-1}(ae_1)$ gives
$$
R^{\mathrm{even}} = I_1 \big(T_1^{-1} (ae_2) (ae_4) \cdots (ae_{k-2}) T_kT_1T_0 \big)I_2
=
{\def\K{8}
\TikZ{[scale=.3]
	\PoleCaps[.15,\K+.85][0,4]
	\draw (5,3) .. controls (5,2.8) .. (6,2.7) to (8.2,2.7);
	\draw (6,2)--(8.2,2) .. controls (9.5,2) ..  (9.5,2.35)
			(6,2)--(1,2) .. controls (-.5,2) .. (-.5,1.5);
	\Pole[.15][0,4]\Pole[\K+.85][0,4][\K]
	\foreach \x in {6}{\draw[over] (\x,3)--(\x,1);}
	\draw[over] (8.2,2.7) .. controls (9.5,2.7) ..  (9.5,2.35);	
        \draw[bend right=100] (5,1) to (6,1) (6,3) to (5,3);
	\foreach \x in {1,3,5,7}{
			\draw[bend right=100] (\x+1,0) to (\x,0);
			\draw[bend right=100] (\x,4) to (\x+1,4);
			}
	\draw[over] (-.5,1.5) .. controls (-.5,1.3) .. (.5,1.3) to (4.7,1.3) .. controls (5,1.3) .. (5,1);
	\foreach \y in {0,4}{\foreach \x in {1,...,\K} {\node[V] at  (\x,\y){};}}
	}
}
= (-t^\frac12)t^{\frac12} D^{\mathrm{even}},
$$
and using $T_{k-1}(ae_{k-1}) = (-t^{-\frac12})(ae_{k-1})$ 
and $T_{k-1}^{-1}T_k^{-1}(ae_{k-1}) = t^{-\frac12}T_k(ae_{k-1})$ gives
$$
S^{\mathrm{even}} = I_1 \big(T_0(ae_2)(ae_4)\cdots (ae_{k-2}) T_{k-1}^{-1}T_k^{-1}T_{k-1} \big) I_1
=
{\def\K{8}
\TikZ{[xscale=.3, yscale=-.3]
	\PoleCaps[.15,\K+.85][0,4]
	\draw (5,3) .. controls (5,2.8) .. (6,2.7) to (8.2,2.7);
	\draw (6,2)--(8.2,2) .. controls (9.5,2) ..  (9.5,2.35)
			(6,2)--(1,2) .. controls (-.5,2) .. (-.5,1.5);
	\Pole[.15][0,4]\Pole[\K+.85][0,4][\K]
	\foreach \x in {6}{\draw[over] (\x,3)--(\x,1);}
	\draw[over] (8.2,2.7) .. controls (9.5,2.7) ..  (9.5,2.35);	
        \draw[bend right=100] (5,1) to (6,1) (6,3) to (5,3);
	\foreach \x in {1,3,5,7}{
			\draw[bend right=100] (\x+1,0) to (\x,0);
			\draw[bend right=100] (\x,4) to (\x+1,4);
			}
	\draw[over] (-.5,1.5) .. controls (-.5,1.3) .. (.5,1.3) to (4.7,1.3) .. controls (5,1.3) .. (5,1);
	\foreach \y in {0,4}{\foreach \x in {1,...,\K} {\node[V] at (\x,\y){};}}
	}
}
=(-t^{-\frac12})t^{-\frac12}D^{\mathrm{even}}.
$$
Pictorially,
{\def\K{8}
$$
I_1W_{1+2i}I_1 = 
\TikZ{[scale=.3]
%	\Label[0,4][5][\tiny$1\!+\!2i$]
	\PoleCaps[.15,\K+.85][0,4]
	\draw (5,3) .. controls (5,2.8) .. (6,2.7) to (8.2,2.7);
	\draw (6,2)--(8.2,2) .. controls (9.5,2) ..  (9.5,2.35)
			(6,2)--(1,2) .. controls (-.5,2) .. (-.5,1.5);
	\Pole[.15][0,4]\Pole[\K+.85][0,4][\K]
	\foreach \x in {1, 2,3,4,6,7,8}{\draw[over] (\x,3)--(\x,1);}
	\draw[over] (8.2,2.7) .. controls (9.5,2.7) ..  (9.5,2.35);	
	\foreach \x in {1,3,5,7}{
		\foreach \y in {4,1}{
			\draw[bend right=100] (\x,\y) to (\x+1,\y) (\x+1,\y-1) to (\x,\y-1);}}
	\draw[over] (-.5,1.5) .. controls (-.5,1.3) .. (.5,1.3) to (4.7,1.3) .. controls (5,1.3) .. (5,1);
	\foreach \y in {0,4}{\foreach \x in {1,...,\K} {\node[V] at  (\x,\y){};}}
	}\ ,
\qquad
I_1W_{2+2i}I_1 =
\TikZ{[scale=.3]
%	\Label[0,4][5][\tiny$1\!+\!2i$]
	\PoleCaps[.15,\K+.85][0,4]
	\draw (6,3) .. controls (6,2.8) .. (7,2.7) to (8.2,2.7);
	\draw (6,2)--(8.2,2) .. controls (9.5,2) ..  (9.5,2.35)
			(6,2)--(1,2) .. controls (-.5,2) .. (-.5,1.5);
	\Pole[.15][0,4]\Pole[\K+.85][0,4][\K]
	\foreach \x in {1, 2,3,4,5,7,8}{\draw[over] (\x,3)--(\x,1);}
	\draw[over] (8.2,2.7) .. controls (9.5,2.7) ..  (9.5,2.35);	
	\foreach \x in {1,3,5,7}{
		\foreach \y in {4,1}{
			\draw[bend right=100] (\x,\y) to (\x+1,\y) (\x+1,\y-1) to (\x,\y-1);}}
	\draw[over] (-.5,1.5) .. controls (-.5,1.3) .. (.5,1.3) to (5.7,1.3) .. controls (6,1.3) .. (6,1);
	\foreach \y in {0,4}{\foreach \x in {1,...,\K} {\node[V] at  (\x,\y){};}}
	}\ ,
$$
	}
{\def\K{8}
$$
I_1W_{1+2i}^{-1}I_1
= \TikZ{[xscale=.3, yscale=-.3]
%	\Label[0,4][5][\tiny$1\!+\!2i$]
	\PoleCaps[.15,\K+.85][0,4]
	\draw (5,3) .. controls (5,2.8) .. (6,2.7) to (8.2,2.7);
	\draw (6,2)--(8.2,2) .. controls (9.5,2) ..  (9.5,2.35)
			(6,2)--(1,2) .. controls (-.5,2) .. (-.5,1.5);
	\Pole[.15][0,4]\Pole[\K+.85][0,4][\K]
	\foreach \x in {1, 2,3,4,6,7,8}{\draw[over] (\x,3)--(\x,1);}
	\draw[over] (8.2,2.7) .. controls (9.5,2.7) ..  (9.5,2.35);	
	\foreach \x in {1,3,5,7}{
		\foreach \y in {4,1}{
			\draw[bend right=100] (\x,\y) to (\x+1,\y) (\x+1,\y-1) to (\x,\y-1);}}
	\draw[over] (-.5,1.5) .. controls (-.5,1.3) .. (.5,1.3) to (4.7,1.3) .. controls (5,1.3) .. (5,1);
	\foreach \y in {0,4}{\foreach \x in {1,...,\K} {\node[V] at  (\x,\y){};}}
	}\ , \quad \text{ and } \quad 
I_1W_{2+2i}^{-1}I_1 = 
\TikZ{[xscale=.3, yscale=-.3]
%	\Label[0,4][5][\tiny$1\!+\!2i$]
	\PoleCaps[.15,\K+.85][0,4]
	\draw (6,3) .. controls (6,2.8) .. (7,2.7) to (8.2,2.7);
	\draw (6,2)--(8.2,2) .. controls (9.5,2) ..  (9.5,2.35)
			(6,2)--(1,2) .. controls (-.5,2) .. (-.5,1.5);
	\Pole[.15][0,4]\Pole[\K+.85][0,4][\K]
	\foreach \x in {1, 2,3,4,5,7,8}{\draw[over] (\x,3)--(\x,1);}
	\draw[over] (8.2,2.7) .. controls (9.5,2.7) ..  (9.5,2.35);	
	\foreach \x in {1,3,5,7}{
		\foreach \y in {4,1}{
			\draw[bend right=100] (\x,\y) to (\x+1,\y) (\x+1,\y-1) to (\x,\y-1);}}
	\draw[over] (-.5,1.5) .. controls (-.5,1.3) .. (.5,1.3) to (5.7,1.3) .. controls (6,1.3) .. (6,1);
	\foreach \y in {0,4}{\foreach \x in {1,...,\K} {\node[V] at  (\x,\y){};}}
	}\ .
$$
	}
Working left to right removing loops,
\begin{align*}
I_1W_{1+2i}I_1 &= \big( t^{\frac12} t^{\frac12} ( -[\![t]\!] ) \big)^i  
(t^{-\frac12}t^{\frac12}( -[\![t]\!] ) )^{\frac{k}{2}-1-i} R^{\mathrm{even}}
= (-[\![t]\!])^{\frac{k}{2}-1} t^{i+\frac12} (-t^{\frac12}) D^{\mathrm{even}}, \\
%I_1W_{2+2i}I_1 &= t^{-1}I_1W_{1+2i}I_1
%= (-a[\![t]\!])^{\frac{k}{2}-1} t^i (-1) D^{\mathrm{even}}, \\
I_1W^{-1}_{1+2i}I_1 
&= \big(t^{-\frac12} t^{-\frac12} ( -[\![t]\!] ) \big)^i 
\big( t^{-\frac12} t^{\frac12} ( -[\![t]\!] ) \big)^{\frac{k}{2}-1-i} S
= (-[\![t]\!])^{\frac{k}{2}-1} t^{-(i+\frac12)} (-t^{-\frac12}) D^{\mathrm{even}}, 
%I_1W^{-1}_{2+2i}I_1
%&= t I_1 W^{-1}_{1+2i} I_1
%= (-a[\![t]\!])^{\frac{k}{2}-1}  t^{-i} (-t^2) D^{\mathrm{even}},
\end{align*}
for $i\in \{0, \ldots, \frac{k}{2}-1\}$.  
Since $I_1W_{1+2i}I_1$ and $I_1 W_{2+2i}I_1$ only differ by two twists
(similarly $I_1W^{-1}_{1+2i}I_1$ and $I_1 W^{-1}_{2+2i}I_1$ only differ by two twists)
the relations $T_i^{\pm1}(ae_i) = (ae_i)T_i^{\pm1} = (-t^{\mp\frac12})(ae_i)$ give
\begin{align*}
I_1W_{2+2i}I_1 &= (-t^{-\frac12})(-t^{-\frac12})t^{-1}I_1W_{1+2i}I_1
= (-[\![t]\!])^{\frac{k}{2}-1} t^{i+\frac12} (-t^{-\frac12}) D^{\mathrm{even}}
\quad\hbox{and} \\
I_1W^{-1}_{2+2i}I_1
&= (-t^{\frac12})(-t^{-\frac12}) I_1 W^{-1}_{1+2i} I_1
= (-[\![t]\!])^{\frac{k}{2}-1}  t^{-(i+\frac12)} (-t^{\frac12}) D^{\mathrm{even}},
\quad\hbox{for $i\in \{0, \ldots, \frac{k}{2}-1\}$.}
\end{align*}
Thus
\begin{align*}
( -[\![t]\!] )^{\frac{k}{2}} Z I_1 
&= ZI_1^{2} = I_1ZI_1 
= \sum_{i=0}^{\frac{k}{2}-1} 
I_1(W_{1+2i}+W_{2+2i}+W^{-1}_{1+2i}+W^{-1}_{2+2i})I_1 \\
&= -(- [\![ t ]\!] )^{\frac{k}{2}-1}  D^{\mathrm{even}} 
\sum_{i=0}^{\frac{k}{2}-1} \big(t^{i+\frac12}(t^{\frac12}+t^{-\frac12}) + t^{-(i+\frac12)}(t^{\frac12}+t^{-\frac12})\big) \\
%&= -(t^{-\frac12}+t^{\frac12})(-a [\![ t ]\!] )^{\frac{k}{2}-1}  D^{\mathrm{even}} 
%\sum_{i=0}^{\frac{k}{2}-1} 
%\big( t^{i+\frac12} +  t^{-i-\frac12} \big) \\
&= (- [\![ t ]\!] )^{\frac{k}{2}}  D^{\mathrm{even}} 
\Big( \frac{ t^{\frac{k}{2}}-t^{-\frac{k}{2}} }{t^{\frac12}-t^{-\frac12}} \Big)
= (- [\![ t ]\!] )^{\frac{k}{2}}  [k] D^{\mathrm{even}}.
\end{align*}

\medskip\noindent
\textbf{Case: $k$ odd.}  Let 
%$E^+ = I_2\big( (ae_3)(ae_5)\cdots (ae_{k-2})T_k \big) I_2$ and
%$E^- = I_2\big( (ae_3) (ae_5) \cdots (ae_{k-2}) T_k^{-1})I_2$  and
\begin{align*}
L^{\mathrm{odd}} 
&= I_2\big( (ae_3) (ae_5) \cdots (ae_{k-2}) e_k \big) I_2 = 
	{\def\K{9}
	\TikZ{[scale=.3]
		\foreach \x in {1,2,...,5,8,\K} {
				\node[V] (t\x) at (\x,3){};
				\node[V] (b\x) at (\x,0){};}
			\node[V] (r1) at (\K+.75,1){};
			\node[V] (r2) at (\K+.75,2){};
			\node[V] (bl) at (.25,1-.5){};
			\node[V] (tl) at (.25,2.5){};
			\node[V] (l1) at (.25,1){};
			\node[V] (l2) at (.25,2){};
		\draw (l1) to  [bend right=100] (l2);
		\foreach \x in {-.25,\K+.25}{	\draw[densely dotted] (\x+.5,0)--(\x+.5,3);}
		\foreach \x in {2,4,8}{	\draw[bend right=100] (\x,3) to (\x+1,3) (\x+1,0) to (\x,0);}
		\draw (r1) to ++(-\K+1,0) to [bend left] ++(0,1) to (r2);
		\draw[bend right=30] (tl)node[V]{} to (t1) 	(b1)  to (bl) node[V]{} ;
		\node at (6.5,2.75) {\tiny$\cdots$};
		\node at (6.5,.25) {\tiny$\cdots$};
	} 
}
 = \Big(\frac{-\(t_0\)}{a_0}\Big)
\Big( \frac{\(t_kt^{-1}\)}{a_k}\Big) I_2,  \\
M^{\mathrm{odd}} 
&= I_2\big((ae_1)(ae_3)\cdots (ae_{k-2})\big)I_2 = 
	{\def\K{9}
	\TikZ{[scale=.3]
		\foreach \x in {1,2,...,5,8,\K} {
				\node[V] (t\x) at (\x,3){};
				\node[V] (b\x) at (\x,0){};}
			\node[V] (l1) at (.25,1){};
			\node[V] (l2) at (.25,2){};
			\node[V] (bl) at (.25,1-.5){};
			\node[V] (tl) at (.25,2.5){};
		\foreach \x in {-.25,\K+.25}{	\draw[densely dotted] (\x+.5,0)--(\x+.5,3);}
		\foreach \x in {2,4,8}{	\draw[bend right=100] (\x,3) to (\x+1,3) (\x+1,0) to (\x,0);}
		\draw (l1) to ++(\K-.25,0) to [bend right] ++(0,1) to (l2);
		\draw[bend right=30] (tl)node[V]{} to (t1) 	(b1)  to (bl) node[V]{} ;
		\node at (6.5,2.75) {\tiny$\cdots$};
		\node at (6.5,.25) {\tiny$\cdots$};
	} 
}
 	= \Big( \frac{-\(t_0\)}{a_0} \Big) I_2,
\quad\hbox{and} \\ 
P^{\mathrm{odd}} 
&= I_2\big( (ae_3)(ae_5) \cdots (ae_{k-2}) \big)I_2 = 
{\def\K{9}
	\TikZ{[scale=.3]
		\foreach \x in {1,2,...,5,8,\K} {
				\node[V] (t\x) at (\x,3){};
				\node[V] (b\x) at (\x,0){};}
			\node[V] (bl) at (.25,1-.5){};
			\node[V] (tl) at (.25,2.5){};
			\node[V] (l1) at (.25,1){};
			\node[V] (l2) at (.25,2){};
		\draw (l1) to  [bend right=100] (l2);
		\foreach \x in {-.25,\K+.25}{	\draw[densely dotted] (\x+.5,0)--(\x+.5,3);}
		\foreach \x in {2,4,8}{	\draw[bend right=100] (\x,3) to (\x+1,3) (\x+1,0) to (\x,0);}
		\draw (3,1) to (9,1) to [bend right] (9,2) to (2,2) to [bend right] (2,1) to (3,1);
		\draw[bend right=30] (tl)node[V]{} to (t1) 	(b1)  to (bl) node[V]{} ;
		\node at (6.5,2.75) {\tiny$\cdots$};
		\node at (6.5,.25) {\tiny$\cdots$};
	} 
}
 = \Big(\frac{-\(t_0\)}{a_0}\Big) ( -\(t\) ) I_2.
\end{align*}
Using
$e_0 T_1^{-1} T_0^{-1} T_1^{-1}  e_0 
= - t^{-\frac12}[\![t_0]\!] e_0 (ae_1) e_0  - t^{-1} t_0^{\frac12}e_0^2$ and
$T_k = a_ke_k+t_k^{\frac12}$ gives
\begin{align*}
D^{\mathrm{odd}}
&= - t^{-\frac12}[\![t_0]\!] a_k I_2I_1I_2 
- t^{-1}t_0^{\frac12} a_k L^{\mathrm{odd}} 
- t^{-\frac12}  [\![t_0]\!] t_k^{\frac12} M^{\mathrm{odd}}
- t^{-1} t_0^{\frac12} t_k^{\frac12} P^{\mathrm{odd}} \\
&= t^{-\frac12} \Big( \frac{-\(t_0\)}{a_0} \Big) \left(
 a_0a_k I_2I_1I_2 + \big(
	-t^{-\frac12} t_0^{\frac12} \(t_kt^{-1}\) 
	-t_k^{\frac12}\(t_0\)
	+ t^{-\frac12}t_0^{\frac12}t_k^{\frac12} \(t\) \big) I_2
\right)\\
&= t^{-\frac12} \Big( \frac{-\(t_0\)}{a_0} \Big) \left(a_0a_k I_2I_1I_2 -\(t_0 t_k^{-1}\) I_2\right),
\end{align*}
%\begin{align*}
%-t^{-\frac12} t_0^{\frac12} &\(t_kt^{-1}\) 
%-t_k^{\frac12}\(t_0\)
%+ t^{-\frac12}t_0^{\frac12}t_k^{\frac12} \(t\) \\
%&=-t^{-\frac12} t_0^{\frac12}t_k^{\frac12}t^{-\frac12}-t^{-\frac12} t_0^{\frac12}t_k^{-\frac12}t^{\frac12}
%-t_k^{\frac12}t_0^{\frac12}-t_k^{\frac12}t_0^{-\frac12}
%+t^{-\frac12}t_0^{\frac12}t_k^{\frac12}t^{\frac12}
%+t^{-\frac12}t_0^{\frac12}t_k^{\frac12}t^{-\frac12} \\
%&=-t^{-1} t_0^{\frac12}t_k^{\frac12} - t_0^{\frac12}t_k^{-\frac12}
%- t_k^{\frac12}t_0^{\frac12} - t_k^{\frac12}t_0^{-\frac12}
%+ t_0^{\frac12}t_k^{\frac12}
%+ t^{-1}t_0^{\frac12}t_k^{\frac12} \\
%&=-t^{-1} t_0^{\frac12}t_k^{\frac12} - t_0^{\frac12}t_k^{-\frac12}
%- t_k^{\frac12}t_0^{-\frac12}
%+ t^{-1}t_0^{\frac12}t_k^{\frac12} 
%= - t_0^{\frac12}t_k^{-\frac12}
%- t_k^{\frac12}t_0^{-\frac12}
%=-\(t_0t_k^{-1}\)
%\\
%\end{align*}
%
which completes the proof of the first statement.

Using $(ae_2)T_2^{-1} = -t^{\frac12}(ae_2)$ 
and $T_2T_1T_0T_1(ae_2) = t^{\frac32} T_1^{-1}T_0^{-1}T_1^{-1}(ae_2)$,
$$
{\def\K{9}
R^{\mathrm{odd}} 
= I_2(T_2^{-1} (ae_3) (ae_5) \cdots (ae_{k-2}) T_k T_2T_1T_0T_1)I_2 
= 
\TikZ{[scale=.3]
	\PoleCaps[.15,\K+.85][0,4]
	\foreach \y in {0,4}{\foreach \x in {1,...,4,5,8,9} {\node[V] at (\x,\y){};}}
	\node at (6.5,2.35) {\tiny$\cdots$};
	\foreach \y in {4,1}{
		\draw [bend left=75] (1,\y) to (.6,\y-.3)node[V]{}   (.6,\y-.7)node[V]{}  to (1,\y-1) ;
		\draw[densely dotted]  (.6,\y)--(.6,\y-1) ; 
		}
	\foreach \x in {2,4,8}{\draw[bend right=100] (\x,4) to (\x+1,4);}
	\foreach \x in {2,4,8}{\draw[bend right=100] (\x+1,0) to (\x,0);}
	\draw[bend right=100](2,1) to (2+1,1) (2+1,4-1) to (2,4-1);
	\draw (2,3) .. controls (2,2.8) .. (3,2.7) to (6,2.7) (7,2.7)--(9.2,2.7);
	\draw (7,2)--(9.2,2) .. controls (10.5,2) ..  (10.5,2.35)
			(6,2)--(1,2) .. controls (-.5,2) .. (-.5,1.5);
	\Pole[.15][0,4]\Pole[\K+.85][0,4][\K]
	\foreach \x in {1,3}{\draw[over] (\x,3)--(\x,1);}%vertical lines
	\draw[over] (9.2,2.7) .. controls (10.5,2.7) ..  (10.5,2.35);
	\draw[over](-.5,1.5)..controls(-.5,1.25)..(1.75,1.25)..controls(2,1.25)..(2,1);
	}
	= (-t^{\frac12})t^{\frac32}D^{\mathrm{odd}}.
}
$$
Using $T_{k-1}(ae_{k-1}) = -t^{-\frac12}(ae_{k-1})$ and $T_{k-1}^{-1}T_k^{-1}(ae_{k-1}) 
= t^{-\frac12}  T_k (ae_{k-1})$ gives
{\def\K{9}
$$S^{\mathrm{odd}} 
= I_2(T_1^{-1}T_0^{-1}T_1^{-1} (ae_3)(ae_5)\cdots (ae_{k-2}) T_{k-1}^{-1}T_k^{-1}T_{k-1})I_2 
=\TikZ{[scale=.3]
	\PoleCaps[.15,\K+.85][0,4]
	\foreach \y in {0,4}{\foreach \x in {1,...,4,5,8,9} {\node[V] at (\x,\y){};}}
	\node at (6.5,2.35) {\tiny$\cdots$};
	\foreach \y in {4,1}{
		\draw [bend left=75] (1,\y) to (.6,\y-.3)node[V]{}   (.6,\y-.7)node[V]{}  to (1,\y-1) ;
		\draw[densely dotted]  (.6,\y)--(.6,\y-1) ; 
		}
	\foreach \x in {2,4,8}{\draw[bend right=100] (\x,4) to (\x+1,4);}
	\foreach \x in {2,4,8}{\draw[bend right=100] (\x+1,0) to (\x,0);}
	\draw[bend right=100](8,1) to (8+1,1) (8+1,4-1) to (8,4-1);
	\draw (-.5,2.25).. controls(-.5,2) .. (0,2) to (6,2);
	\draw (7,2)--(9.2,2) .. controls (10.5,2) ..  (10.5,1.5);
	\Pole[.15][0,4]\Pole[\K+.85][0,4][\K]
	\draw[over] (10.5,1.5).. controls (10.5,1.25) ..  (9.2,1.25) to (8.5,1.25) 
			.. controls (8,1.25) .. (8,1) ;
	\foreach \x in {1,9}{\draw[over] (\x,3)--(\x,1);}%vertical lines
	\draw[over] (8,3) .. controls (8,2.8) .. (7,2.7) 
		(6,2.7) to (0,2.7) .. controls(-.5,2.7) .. (-.5,2.25);
	}
= (-t^{-\frac12})t^{-\frac12}D^{\mathrm{odd}}.
$$}
Pictorially,
{\def\K{9}
$$
I_2W_{1+2i}I_2 = \TikZ{[scale=.3]
	\PoleCaps[.15,\K+.85][0,4]
	\foreach \y in {0,4}{\foreach \x in {1,...,9} {\node[V] at (\x,\y){};}}
	\foreach \y in {4,1}{
		\draw [bend left=75] (1,\y) to (.6,\y-.3)node[V]{}   (.6,\y-.7)node[V]{}  to (1,\y-1) ;
		\draw[densely dotted]  (.6,\y)--(.6,\y-1) ; 
		\foreach \x in {2,4,6,8}{\draw[bend right=100] (\x,\y) to (\x+1,\y) (\x+1,\y-1) to (\x,\y-1);}}
	\draw (6,3) .. controls (6,2.8) .. (7,2.7) to (9.2,2.7);
	\draw (7,2)--(9.2,2) .. controls (10.5,2) ..  (10.5,2.35)
			(7,2)--(1,2) .. controls (-.5,2) .. (-.5,1.5);
	\Pole[.15][0,4]\Pole[\K+.85][0,4][\K]
	\foreach \x in {1,2,3,4,5,7,8,9}{\draw[over] (\x,3)--(\x,1);}%vertical lines
	\draw[over] (9.2,2.7) .. controls (10.5,2.7) ..  (10.5,2.35);
	\draw[over](-.5,1.5)..controls(-.5,1.25)..(5.75,1.25)..controls(6,1.25)..(6,1);
	}\ ,
\qquad
I_2W_{2+2i}I_2 = \TikZ{[scale=.3]
	\PoleCaps[.15,\K+.85][0,4]
	\foreach \y in {0,4}{\foreach \x in {1,...,9} {\node[V] at (\x,\y){};}}
	\foreach \y in {4,1}{
		\draw [bend left=75] (1,\y) to (.6,\y-.3)node[V]{}   (.6,\y-.7)node[V]{}  to (1,\y-1) ;
		\draw[densely dotted]  (.6,\y)--(.6,\y-1) ; 
		\foreach \x in {2,4,6,8}{\draw[bend right=100] (\x,\y) to (\x+1,\y) (\x+1,\y-1) to (\x,\y-1);}}
	\draw (7,3) .. controls (7,2.8) .. (8,2.7) to (9.2,2.7);
	\draw (7,2)--(9.2,2) .. controls (10.5,2) ..  (10.5,2.35)
			(7,2)--(1,2) .. controls (-.5,2) .. (-.5,1.5);
	\Pole[.15][0,4]\Pole[\K+.85][0,4][\K]
	\foreach \x in {1,2,3,4,5,6,8,9}{\draw[over] (\x,3)--(\x,1);}%vertical lines
	\draw[over] (9.2,2.7) .. controls (10.5,2.7) ..  (10.5,2.35);
	\draw[over](-.5,1.5)..controls(-.5,1.25)..(6.75,1.25)..controls(7,1.25)..(7,1);
	}\ ,
$$
	}
{\def\K{9}
$$
I_2W_{1+2i}^{-1}I_2 = 
\TikZ{[xscale=.3,yscale=-.3]
	\PoleCaps[.15,\K+.85][0,4]
	\foreach \y in {0,4}{\foreach \x in {1,...,9} {\node[V] at (\x,\y){};}}
	\foreach \y in {4,1}{
		\draw [bend left=75] (1,\y) to (.6,\y-.3)node[V]{}   (.6,\y-.7)node[V]{}  to (1,\y-1) ;
		\draw[densely dotted]  (.6,\y)--(.6,\y-1) ; 
		\foreach \x in {2,4,6,8}{\draw[bend right=100] (\x,\y) to (\x+1,\y) (\x+1,\y-1) to (\x,\y-1);}}
	\draw (6,3) .. controls (6,2.8) .. (7,2.7) to (9.2,2.7);
	\draw (7,2)--(9.2,2) .. controls (10.5,2) ..  (10.5,2.35)
			(7,2)--(1,2) .. controls (-.5,2) .. (-.5,1.5);
	\Pole[.15][0,4]\Pole[\K+.85][0,4][\K]
	\foreach \x in {1,2,3,4,5,7,8,9}{\draw[over] (\x,3)--(\x,1);}%vertical lines
	\draw[over] (9.2,2.7) .. controls (10.5,2.7) ..  (10.5,2.35);
	\draw[over](-.5,1.5)..controls(-.5,1.25)..(5.75,1.25)..controls(6,1.25)..(6,1);
	}\ , \quad \text{ and } \quad 
I_2W_{2+2i}^{-1}I_2 = 
\TikZ{[xscale=.3,yscale=-.3]
	\PoleCaps[.15,\K+.85][0,4]
	\foreach \y in {0,4}{\foreach \x in {1,...,9} {\node[V] at (\x,\y){};}}
	\foreach \y in {4,1}{
		\draw [bend left=75] (1,\y) to (.6,\y-.3)node[V]{}   (.6,\y-.7)node[V]{}  to (1,\y-1) ;
		\draw[densely dotted]  (.6,\y)--(.6,\y-1) ; 
		\foreach \x in {2,4,6,8}{\draw[bend right=100] (\x,\y) to (\x+1,\y) (\x+1,\y-1) to (\x,\y-1);}}
	\draw (7,3) .. controls (7,2.8) .. (8,2.7) to (9.2,2.7);
	\draw (7,2)--(9.2,2) .. controls (10.5,2) ..  (10.5,2.35)
			(7,2)--(1,2) .. controls (-.5,2) .. (-.5,1.5);
	\Pole[.15][0,4]\Pole[\K+.85][0,4][\K]
	\foreach \x in {1,2,3,4,5,6,8,9}{\draw[over] (\x,3)--(\x,1);}%vertical lines
	\draw[over] (9.2,2.7) .. controls (10.5,2.7) ..  (10.5,2.35);
	\draw[over](-.5,1.5)..controls(-.5,1.25)..(6.75,1.25)..controls(7,1.25)..(7,1);
	}\ .
$$
	}
Working left to right removing loops,
\begin{align*}
I_2W_{2+2i}I_2 &= \big( t^{\frac12} t^{\frac12} (-[\![t]\!] ) \big)^i  
\big( t^{-\frac12} t^{\frac12} ( -[\![t]\!] ) \big)^{\frac{k-1}{2}-1-i} R^{\mathrm{odd}}
= (- [\![t]\!])^{\frac{k-3}{2}} t^{i+1}(-t) D^{\mathrm{odd}}, \\
I_2W^{-1}_{2+2i}I_2 
&= \big( t^{-\frac12} t^{-\frac12} ( -[\![t]\!] ) \big)^i 
\big( t^{-\frac12} t^{\frac12} ( -[\![t]\!] ) \big)^{\frac{k-1}{2}-1-i} S^{\mathrm{odd}}
= ( - [\![t]\!] )^{\frac{k-3}{2}} t^{-(i+1)} (-1) D^{\mathrm{odd}}, 
\end{align*}
for $i\in \{0, \ldots, \frac{k-1}{2}-1\}$.  
Since $I_2W_{2+2i}I_2$ and $I_2 W_{3+2i}I_2$ only differ by two twists
(similarly $I_2W^{-1}_{2+2i}I_2$ and $I_2 W^{-1}_{3+2i}I_2$ only differ by two twists)
the relations $T_i^{\pm1}e_i = e_iT_i^{\pm1} = (-t^{\mp\frac12})e_i$ give
\begin{align*}
I_2W_{3+2i}I_2 &= (-t^{-\frac12})(-t^{-\frac12}) I_2 W_{2+2i} I_2 
%= (-a [\![t]\!])^{\frac{k-3}2}t^i (-t)D^{\mathrm{odd}} 
= (- [\![t]\!])^{\frac{k-3}2}t^{i+1} (-1)D^{\mathrm{odd}}, \quad \text{and}
\\
I_2W_{3+2i}^{-1} I_2 &= (-t^{\frac12})(-t^{\frac12}) I_2 W_{2+2i}^{-1} I_2 
%= (-a [\![t]\!])^{\frac{k-3}2}t^{-i} (-1)D^{\mathrm{odd}}
= (- [\![t]\!])^{\frac{k-3}2}t^{-(i+1)} (-t)D^{\mathrm{odd}},
\qquad\hbox{for $i\in \{0, \ldots, \frac{k-1}{2}-1\}$.}
\end{align*}
Next,
\begin{align*}
I_2 W_1 I_2 &= 
{\def\K{9}	
\TikZ{[scale=.3]
	\PoleCaps[.15,\K+.85][0,4]
	\foreach \y in {0,4}{\foreach \x in {1,...,4,5,8,9} {\node[V] at (\x,\y){};}}
	\node at (6.5,2.35) {\tiny$\cdots$};
	\foreach \y in {4,1}{
		\draw [bend left=75] (1,\y) to (.6,\y-.3)node[V]{}   (.6,\y-.7)node[V]{}  to (1,\y-1) ;
		\draw[densely dotted]  (.6,\y)--(.6,\y-1) ; 
		\foreach \x in {2,4,8}{\draw[bend right=100] (\x,\y) to (\x+1,\y) (\x+1,\y-1) to (\x,\y-1);}}
	\draw (1,3) .. controls (1,2.8) .. (2,2.7) to (6,2.7) (7,2.7)--(9.2,2.7);
	\draw (7,2)--(9.2,2) .. controls (10.5,2) ..  (10.5,2.35)
			(6,2)--(1,2) .. controls (-.5,2) .. (-.5,1.5);
	\Pole[.15][0,4]\Pole[\K+.85][0,4][\K]
	\foreach \x in {2,3,4,5,8,9}{\draw[over] (\x,3)--(\x,1);}
	\draw[over] (9.2,2.7) .. controls (10.5,2.7) ..  (10.5,2.35);
	\Over[1,1][-.5,1.5]
	}
}
= (-t_0^{-\frac12})( - [\![t]\!] )^{\frac{k-1}{2}} I_2((ae_1)(ae_3)\cdots (ae_{k-2})T_k)I_2,
\quad\hbox{and} \\
 I_2 W_1^{-1} I_2 &= 
 	{\def\K{9}	
\TikZ{[xscale=.3, yscale=-.3]
	\PoleCaps[.15,\K+.85][0,4]
	\foreach \y in {0,4}{\foreach \x in {1,...,4,5,8,9} {\node[V] at (\x,\y){};}}
	\node at (6.5,2.35) {\tiny$\cdots$};
	\foreach \y in {4,1}{
		\draw [bend left=75] (1,\y) to (.6,\y-.3)node[V]{}   (.6,\y-.7)node[V]{}  to (1,\y-1) ;
		\draw[densely dotted]  (.6,\y)--(.6,\y-1) ; 
		\foreach \x in {2,4,8}{\draw[bend right=100] (\x,\y) to (\x+1,\y) (\x+1,\y-1) to (\x,\y-1);}}
	\draw (1,3) .. controls (1,2.8) .. (2,2.7) to (6,2.7) (7,2.7)--(9.2,2.7);
	\draw (7,2)--(9.2,2) .. controls (10.5,2) ..  (10.5,2.35)
			(6,2)--(1,2) .. controls (-.5,2) .. (-.5,1.5);
	\Pole[.15][0,4]\Pole[\K+.85][0,4][\K]
	\foreach \x in {2,3,4,5,8,9}{\draw[over] (\x,3)--(\x,1);}
	\draw[over] (9.2,2.7) .. controls (10.5,2.7) ..  (10.5,2.35);
	\Over[1,1][-.5,1.5]
	}
} 
= (-t_0^{\frac12})(-[\![t]\!])^{\frac{k-1}{2}} I_2((ae_1)(ae_3)\cdots (ae_{k-2})T_k^{-1})I_2.
\end{align*}
Using $-t_0^{-\frac12}T_k-t_0^{\frac12}T_k^{-1}
=-t_0^{-\frac12}(a_ke_k+t_k^{\frac12})-t_0^{\frac12}(a_ke_k+t_k^{-\frac12})
=-\(t_0\)a_ke_k - \(t_ot_k^{-1}\)$,
\begin{align*}
I_2(W_1+W_1^{-1})I_2
&= (-\(t\))^{\frac{k-1}{2}}\Big( -\(t_0\) a_k I_2I_1I_2 - \(t_0t_k^{-1}\)M^{\mathrm{odd}}\Big) \\
&= (-\(t\))^{\frac{k-1}{2}}\Big( -\(t_0\) a_k I_2I_1I_2 - \(t_0t_k^{-1}\) \Big(\frac{-\(t_0\)}{a_0}\Big) I_2\Big) \\
&= (-\(t\))^{\frac{k-1}{2}}\Big(\frac{-\(t_0\)}{a_0}\Big) \Big( a_0a_k I_2I_1I_2 - \(t_0t_k^{-1}\) I_2\Big) %&= t^{\frac12}(-\(t\))^{\frac{k-1}{2}} D^{\mathrm{odd}}
= -(t+1)(-\(t\))^{\frac{k-3}{2}} D^{\mathrm{odd}}.
\end{align*}
Thus
\begin{align*}
\Big(\frac{-[\![t_0]\!]}{a_0}\Big) &( - [\![t]\!] )^{\frac{k-1}{2}}Z I_2 
= Z I_2^2 = I_2 Z I_2 \\
&= I_2(W_1+W_1^{-1})I_2 + \sum_{i=0}^{\frac{k-1}{2}-1} I_2(W_{2+2i}+W_{3+2i}+W_{2+2i}^{-1}+W_{3+2i}^{-1})I_2 \\
&= 
-(t+1)(-\(t\))^{\frac{k-3}{2}} t^{\frac12}D^{\mathrm{odd}} 
+ (-\(t\))^{\frac{k-3}{2}} \Big(\sum_{i=0}^{\frac{k-3}{2}} (t^{i+1}-t^{-(i+1)})(-t-1)D^{\mathrm{odd}}
\Big) \\
&= - (-\(t\))^{\frac{k-3}{2}}(t+1) D^{\mathrm{odd}}\Big( 1+ \sum_{i=0}^{\frac{k-3}{2}} (t^{i+1}-t^{-(i+1)})
\Big) 
= (-\(t\))^{\frac{k-1}{2}} t^{\frac12}D^{\mathrm{odd}} [k].
\end{align*} 
%So
%\begin{align*}
%\Big(\frac{-[\![t_0]\!]}{a_0}\Big) &Z I_2
%=
%t^{\frac12} [k] D^{\mathrm{odd}} 
%=
%t^{\frac12} [k] t^{-\frac12} \Big( \frac{-\(t_0\)}{a_0} \Big) \left(
% a_0a_k I_2I_1I_2 - \(t_0 t_k^{-1}\)I_2
%\right).
%\end{align*}
\end{proof}

\begin{cor}\label{IIIovercenter} 
Let $Z= W_1+W_1^{-1}+\cdots+W_k+W_k^{-1}$ and let $I_1$ and $I_2$
be as defined in \eqref{eq:defI1} and \eqref{eq:defI2}.  If $k$ is even, then
$$a_0a_k I_1I_2I_1 = \Big(\frac{1}{[k]} Z-\(t_0t_kt^{-1}\) \Big)I_1\qquad\hbox{and}\qquad
a_0a_k I_2I_1I_2 = \Big(\frac{1}{[k]} Z-\(t_0t_kt^{-1}\) \Big)I_2.$$
If $k$ is odd, then
$$a_0a_k I_1I_2I_1 
= \Big(\frac{1}{[k]} Z+ \(t_0t_k^{-1}\) \Big)I_1
\quad\hbox{and}\quad
a_0a_k I_2I_1I_2 
= \Big(\frac{1}{[k]} Z+ \(t_0t_k^{-1}\) \Big)I_2.
$$
\end{cor}
\begin{proof}
As observed in the proof of Theorem \ref{IIIexpansion}, the products $I_1Z I_1$ and $I_2Z_2$
reduce to computation of the diagram with a single string going around all the poles ($D^{\mathrm{even}}$ or $D^{\mathrm{odd}}$).  These diagrammatics give that
there are constants $C, C_1, C_2$ and $D, D_1, D_2$ such that
$$I_1^2 = CI_1, \quad I_1I_2I_1 = (C_1 Z  +C_2) I_1,
\quad I_2^2=DI_2, \quad I_1I_2I_1 = (D_1 Z  + D_2)I_2.$$
Then, computing $(I_1I_2I_1)^2$ in two different ways, we have
\begin{align*}
I_1I_2I_1I_1I_2I_1 &= CI_1I_2I_1I_2I_1 = C(D_1Z+D_2)I_1I_2I_1, \quad \text{and}\\
I_1I_2I_1I_1I_2I_1 &= (C_1Z+C_2)I_1 I_1I_2I_1 = C(C_1Z+C_2)I_1I_2I_1,
\end{align*}
which indicates that $C_1Z+C_2 = D_1Z+D_2$.

Theorem \ref{IIIexpansion} gives that, if $k$ is even, then
$$a_0a_k I_1 I_2 I_1 
= D^{\mathrm{even}} - [\![ t_0t_kt^{-1}]\!] I_1
= \frac{1}{[k]} ZI_1 - [\![ t_0t_kt^{-1}]\!] I_1,$$
and if $k$ is odd, then
$$a_0a_k I_2I_1I_2 = t^{\frac12}\Big(\frac{a_0}{-\(t_0\)}\Big) D^{\mathrm{odd}}
+\(t_0t_k^{-1}\)I_2
= \frac{1}{[k]}ZI_2 + \(t_0t_k^{-1}\)I_2.
$$
\end{proof}

\begin{remark}\textbf{Comparison to de Gier-Nichols.}\label{rk:conversion-to-GN}
Let us explain how to relate the constants in Corollary \ref{IIIovercenter} and Proposition \ref{centralchar} to the
values which appear in \cite{GN}.  Let
$$
\begin{array}{ccccc}
t_0^{\frac12} = -iq^{\omega_1}, &\quad
&t^{\frac12} = q^{-1}, &\quad
&t_k^{\frac12} = -iq^{\omega_2}, \\
T_0 = -i\zg_0, &\quad
&T_i = -\zg_i, &\quad
&T_k = -i\zg_k, \\
e_0 = \ze_0, &\quad
&e_i = \ze_i, &\quad
&e_k = \ze_k.
\end{array}
$$
Then
$$(\zg_0-q^{\omega_1})(\zg_0-q^{-\omega_1})=0,
\quad
(\zg_i+q^{-1})(\zg_i-q)=0,
\quad
(\zg_k-q^{\omega_1})(\zg_-q^{-\omega_1})=0,
$$
as in \cite[Definitions 2.4, 2.6, and 2.8]{GN}, and
$$\zg_0 = q^{\omega_1}-(q^{1+\omega_1}-q^{-(1+\omega_1)})\ze_0,
\quad
\zg_i = \ze_i-q^{-1},
\quad
\zg_k = q^{\omega_2}-(q^{1+\omega_2}-q^{-(1+\omega_2)})e_k,
$$
as in \cite[(5)]{GN}.  Following \cite[Definitions 2.8 and (9)]{GN},
\begin{align*}
\zJ_0^{(C)} &= \zg_1^{-1}\cdots \zg_{k-1}^{-1}\zg_k\zg_{k-1}\cdots \zg_2\zg_1\zg_0
=(-1)^{k-1}(-i)(-i)(-1)^{k-1}T_1^{-1}\cdots T_{k-1}^{-1}T_kT_{k-1}\cdots T_1T_0 = -W_1, \\
\zJ_i^{(C)} &=\zg_i\zJ_{i-1}^{(C)}\zg_i = (-1)^2T_i(-W_i)T_i = -W_{i+1}
\quad\hbox{for $i\in \{1, \ldots, k-1\}$, and} \\
\zZ_k &= \sum_{i=0}^{k-1} (J_i^{(C)}+(J_i^{(C)})^{-1}) = -(W_1+w_1^{-1}+\cdots+W_k+W_k^{-1})
=-Z.
\end{align*}
Use the notation
$[x] = \frac{t^{\frac{x}{2}} - t^{-\frac{x}{2}} }{t^{\frac12}-t^{-\frac12}} = \frac{q^x-q^{-x} }{ q - q^{-1} }$
and let
$a_0, a$ and $a_k$  take the favorite values from Remark \ref{bdequivi} so that
$$a = -1, \quad a_0 = -\(t_0t^{-1}\), \text{ and }  a_k = -\(t_kt^{-1}\),
\quad\hbox{and set}\quad
\theta = c+\frac{k-1}{2}\quad\hbox{and}\quad
z=\(t^\theta\)[k],
$$
as in Proposition \ref{centralchar}.  Following \cite[Theorem 4.1]{GN} and remembering
that $\zZ_k = -Z$, let
$$\Theta = \theta+\frac{1}{\log q}i\pi
\quad\hbox{so that}\quad
\begin{array}{rl}
-[k]\(t^\theta\) &= -[k](t^{\frac{\theta}{2}}+t^{-\frac{\theta}{2}})
=[k](-q^{-\theta}-q^{\theta}) \\
&=[k](q^{-(\theta+\frac{1}{\log q}i\pi)}+q^{\theta+\frac{1}{\log q}i\pi})
=[k] (q^{-\Theta}+q^{\Theta}) = [k]\frac{[2\Theta]}{[\Theta]}.
\end{array}
$$
Note that
\begin{align*}
a_0a_k &= \(t_0t^{-1}\)\(t_kt^{-1}\) 
= (t_0^{\frac12}t^{-\frac12}+t_0^{-\frac12}t^{\frac12})
(t_k^{\frac12}t^{-\frac12}+t_k^{-\frac12}t^{\frac12}) \\
&=(-iq^{-\omega_1-1}+iq^{\omega_1+1})(-iq^{-\omega_2-1}+iq^{\omega_2+1})
=-[\omega_1+1][\omega_2+1](q-q^{-1})^2.
\end{align*}
Then the constant $b$ that appears in \cite[Definition 3.6 and Theorem 4.1]{GN} 
to make $I_1I_2I_1 = bI_1$ and $I_2I_1I_2 = bI_2$ as operators on a simple $TL_k$-module
is  computed from Corollary \ref{IIIovercenter} as follows:
\begin{align*}
b&=\frac{\frac{1}{[k]} z - \(t_0t_kt^{-1}\)}{a_0a_k} 
=\frac{\frac{1}{[k]}[k][\![t^\theta]\!] - \(t_0t_kt^{-1}\)}{ [\![t_0t^{-1}]\!] [\![t_kt^{-1}]\!]}
= \frac{[\![t^\theta]\!] -[\![t_0t_kt^{-1}]\!]}{ [\![t_0t^{-1}]\!] [\![t_kt^{-1}]\!]}\\
&= - \frac{ (q^\Theta+q^{-\Theta})
+ (-iq^{\omega_1})(-iq^{\omega_2})q+(iq^{-\omega_1})(iq^{-\omega_2})q^{-1} }
{-[\omega_1+1][\omega_2+1](q-q^{-1})^2 } \\
&= \frac{ q^\Theta+q^{-\Theta}
- q^{\omega_1+\omega_2+1} - q^{-(\omega_1+\omega_2+1)} }
{[\omega_1+1][\omega_2+1](q-q^{-1})^2 } \\
&=   \frac{( (q^{\omega_1 + \omega_2 + 1 + \Theta})^{\frac12} 
- (q^{\omega_1 + \omega_2 + 1 + \Theta})^{-\frac12})
(  (q^{\omega_1 + \omega_2 + 1 - \Theta})^{\frac12} - (q^{\omega_1 + \omega_2 + 1 - \Theta})^{-\frac12}))
}
{[\omega_1+1][\omega_2+1](q-q^{-1})^2 } \\
&= \frac{[\half(\omega_1 + \omega_2 + 1 + \Theta)][\half(\omega_1 + \omega_2 + 1 - \Theta)]}{[\omega_1 + 1][\omega_2+1]}
\quad\hbox{when $k$ is even, and}
\end{align*}
\begin{align*}
b
&=\frac{\frac{1}{[k]}z + \(t_0t_k^{-1}\) }{a_0a_k} 
=\frac{\frac{1}{[k]}[k][\![t^\theta]\!] + \(t_0t_k^{-1}\) }{[\![t_0 t^{-1}]\!] [\![t_kt^{-1}]\!]} 
= \frac{[\![t^\theta]\!] + [\![t_0 t_k^{-1}]\!]}{[\![t_0 t^{-1}]\!] [\![t_kt^{-1}]\!]} \\
&= \frac{ -(q^\Theta+q^{-\Theta})
+ (-iq^{\omega_1})(iq^{-\omega_2})+(iq^{-\omega_1})(-iq^{\omega_2}) }
{-[\omega_1+1][\omega_2+1](q-q^{-1})^2 } \\
&= \frac{ -q^\Theta - q^{-\Theta}
+ q^{\omega_1-\omega_2} + q^{-(\omega_1-\omega_2)} }
{- [\omega_1+1][\omega_2+1](q-q^{-1})^2 } \\
&= - \frac{( (q^{\omega_1-\omega_2-\Theta})^{\frac12} - (q^{\omega_1-\omega_2-\Theta})^{-\frac12})
 		( (q^{\omega_1-\omega_2+\Theta})^{\frac12} - (q^{\omega_1-\omega_2+\Theta})^{-\frac12})
 }
{[\omega_1+1][\omega_2+1](q-q^{-1})^2 } \\
&= - \frac{[\half(\omega_1 - \omega_2  - \Theta)][\half(\omega_1 - \omega_2 + \Theta)]}
{[\omega_1 + 1][\omega_2+1]}
	\qquad\hbox{when $k$ is odd.}
\end{align*}
\end{remark}

\section{Calibrated representations of $H^{\mathrm{ext}}_k$ and $TL^{\mathrm{ext}}_k$}

In this section we classify and construct all irreducible calibrated representations of the 
extended two boundary Temperley-Lieb algebras $TL^{\mathrm{ext}}_k$.  
This is done by using the classification
of irreducible calibrated $H_k^{\mathrm{ext}}$-modules from \cite{DR}, which we begin by reviewing in Sections \ref{sec:Calibrated H reps} and \ref{configurationsofboxes}.  Using the
formulas for the elements $p_i^{(1^3)}$,
$p_0^{(\emptyset, 1^2)}$, $p_0^{(1^2, \emptyset)}$, $p_{0^\vee}^{(\emptyset, 1^2)}$, and 
$p_{0^\vee}^{(1^2,\emptyset)}$ that one quotients $H_k^{\mathrm{ext}}$ by to obtain $TL^{\mathrm{ext}}_k$,
we determine exactly which irreducible calibrated representations of $H_k^{\mathrm{ext}}$
factor through the quotient, thus providing a full classification of irreducible calibrated
representations of $TL^{\mathrm{ext}}_k$.

\subsection{Calibrated representations of $H^{\mathrm{ext}}_k$}
\label{sec:Calibrated H reps}

A \emph{calibrated $H_k^{\mathrm{ext}}$-module} is an $H_k^{\mathrm{ext}}$-module $M$ such 
that $W_0, W_1, \ldots, W_k$ are simultaneously diagonalizable as operators on $M$.  
Let $%c_1, \ldots, c_k, 
r_1, r_2\in \CC$ such that 
\begin{equation}
%\gamma_i = -t^{c_i}, \quad\quad
-t^{r_1} = -t_k^{\frac12}t_0^{-\frac12} \quad\hbox{ and }\quad -t^{r_2}=t_k^{\frac12}t_0^{\frac12}.
\label{eq:r1andr2}
\end{equation}
For $\cc =(c_1, \ldots, c_k)\in \CC^k$ let $c_{-i}=-c_i$ and define
\begin{align}
Z(\cc) 
&= \{ \varepsilon_i \ |\ c_i = 0\}
\sqcup \{\varepsilon_j-\varepsilon_i\ |\ \hbox{$0<i<j$ and $c_j-c_i=0$} \}, \nonumber \\
&\hspace{1.6in}
\sqcup \{ \varepsilon_j+\varepsilon_i\ |\ \hbox{$0<i<j$ and $c_j+c_i=0$} \},
\label{Z(c)origdefn} \\
P(\cc) 
&= \{ \varepsilon_i \ |\ c_i \in \{ \pm r_1, \pm r_2\} \}
\sqcup \{\varepsilon_j-\varepsilon_i\ |\ 0<i<j\ \hbox{and}\ c_j-c_i = \pm1\} \nonumber\\
&\hspace{1.6in}
\sqcup \{\varepsilon_j+\varepsilon_i\ |\ 0<i<j\ \hbox{and}\ c_j+c_i = \pm1\},
\label{P(c)origdefn}
\end{align}
where $\{ \vep_1, \dots, \vep_n\}$ is an orthonormal basis for the weights corresponding to $\fgl_n$ (see \cite[\S 3]{DR}). 
A \emph{local region} is a pair $(\cc, J)$ with $\cc\in \CC^k$ and $J\subseteq P(\cc)$.
The set of \emph{standard tableaux} of shape $(\cc,J)$ is 
\begin{equation}
\cF^{(\cc,J)} = \{ w\in \cW_0\ |\ R(w)\cap Z(\cc) = \emptyset,\ R(w)\cap P(\cc) = J\}
\label{stdtaborigdefn}
\end{equation}
(see the following section for a visualization of this set as fillings of box arrangements).
A \emph{skew local region} is a local region $(\cc, J)$, $\cc= (c_1,\ldots, c_k)$, such that \hfil\break
\indent if $w \in \cF^{(\cc,J)}$ then $w\cc = ((w\cc)_1, \ldots, (w\cc)_n)$ satisfies
\begin{equation}
\begin{array}{c}
(w\cc)_1\ne 0, \quad (w\cc)_2\ne 0, \quad (w\cc)_1\ne -(w\cc)_2, \\
\\
(w\cc)_i\ne (w\cc)_{i+1}\ \hbox{for $i=1,\ldots, k-1$,}\quad\hbox{and}\quad
(w\cc)_i\ne (w\cc)_{i+2}\ \hbox{for $i=1,\ldots, k-2$.}
\end{array}
\label{skewlocalregiondefn}
\end{equation}
The following theorem constructs and classifies the calibrated irreducible representations of $H_k^{\mathrm{ext}}$.

\begin{thm}\label{thm:calibconst} \cite[Theorem 3.3]{DR}
Assume $t^{\frac12}$, $t_0^{\frac12}$, and  $t_k^{\frac12}$ are invertible, 
$t^{\frac12}$ is not a root of unity, and
$$
t_0^{\frac12}t_k^{\frac12},-t_0^{-\frac12}t_k^{\frac12}\not\in \{1, -1, t^{\pm\frac12}, -t^{\pm\frac12}, t^{\pm1}, -t^{\pm1}\}
\quad \hbox{and}\quad
t_0^{\frac12}t_k^{\frac12}\ne (-t_0^{-\frac12}t_k^{\frac12})^{\pm1}.
$$
Let $r_1, r_2$ be as in \eqref{eq:r1andr2}.
\begin{enumerate}[(a)]
\item Let  $(\cc,J)$ be a skew local region and let $z\in \CC^\times$. Define 
\begin{equation}
H_k^{(z,\cc,J)}  = \mathrm{span}_\CC \{ v_w \ |\  w \in \cF^{(\cc,J)}\},
\label{HgammaJ}
\end{equation}
so that the symbols $v_w$ are a labeled basis of the vector space $H_k^{(z,\cc,J)}$.  Let
$$\gamma_i = -t^{c_i}
\quad\hbox{for $i=1,2,\ldots, k$,} \qquad \text{and} \qquad
\gamma_0 = z\gamma_{w^{-1}(1)}^{-1}\cdots \gamma_{w^{-1}(k)}^{-1}.
$$
Then the following formulas make $H_k^{(z,\cc,J)}$ into an irreducible $H^\ext_k$-module:
\begin{align}
PW_1&\cdots W_k v_w = zv_w, \qquad P v_w = \gamma_0 v_w, \qquad
W_i v_w = \gamma_{w^{-1}(i)} v_w, \label{snormActionPW} \\
T_i v_w &= [T_i]_{ww} v_w + \sqrt{-([T_i]_{ww}-t^{\frac12})([T_i]_{ww}+t^{-\frac12})}\  v_{s_iw},
\quad\hbox{for $i=1, \dots,k-1$}, \label{snormActionT}\\
T_0v_w &= [T_0]_{ww}v_w+
			 \sqrt{-([T_0]_{ww}-t_0^{\frac12})([T_0]_{ww}+t_0^{-\frac12})}\  v_{s_0w},
\label{snormActionX}
\end{align}
where $v_{s_iw} = 0$ if $s_iw\not\in \cF^{(\cc,J)}$, and
\begin{equation}
[T_i]_{ww} = \frac{t^{\frac12}-t^{-\frac12}}{1-\gamma_{w^{-1}(i)}\gamma^{-1}_{w^{-1}(i+1)}} 
\quad \text{ and } \quad 
[T_0]_{ww} = \frac{ (t_0^{\frac12}-t_0^{-\frac12}) + (t_k^{\frac12} -t_k^{-\frac12})\gamma^{-1}_{w^{-1}(1)}} 
{1-\gamma^{-2}_{w^{-1}(1)}}.
\end{equation}

\item  
The map
$$
\begin{matrix}
\CC^\times\times\{\hbox{skew local regions $(\cc, J)$}\}
&\longleftrightarrow 
&\{\hbox{irreducible calibrated $H_k^\ext$-modules}\} \\
(z,\cc, J) &\longmapsto &H_k^{(z,\cc, J)}
\end{matrix}
$$
\end{enumerate}
is a bijection.
\end{thm}

\subsection{Configurations of boxes}\label{configurationsofboxes}

Let $(\cc, J)$ be a local region with $\cc=(c_1,\ldots, c_k)$,
\begin{equation}
\cc\in \ZZ^k \quad\hbox{or}\quad \cc\in (\ZZ+\hbox{$\frac12$})^k,
\qquad \hbox{and}\qquad 0\le c_1\le \cdots\le c_k.
\label{boxconfigcases}
\end{equation}
Start with an infinite arrangement of NW to SE diagonals, numbered consecutively from $\ZZ$ or $\ZZ + \half$, 
increasing southwest to northeast (see Example \ref{ex:smallconfiguration}).
The \emph{configuration} $\kappa$ of boxes corresponding to the local region $(\cc,J)$ 
has $2k$ boxes 
(labeled $\mathrm{box}_{-k}, \ldots, \mathrm{box}_{-1}, \mathrm{box}_1, \ldots, \mathrm{box}_k$) 
with the following conditions.
\begin{enumerate}[\quad($\kappa$1)]
\item Location: $\mathrm{box}_i$ is on  diagonal $c_i$, where $c_{-i} = -c_{i}$ for $i\in \{-k, \ldots, -1\}$. 
\item Same diagonals: $\mathrm{box}_i$ is NW of $\mathrm{box}_j$ if
$i<j$ and $\mathrm{box}_i$ and $\mathrm{box}_j$ are on the same diagonal. 
\item Adjacent diagonals: \\
		{\quad} If $\varepsilon_j-\varepsilon_i\in J$, then $\mathrm{box}_j$ is NW (strictly north and weakly west) of
			 $\mathrm{box}_i$:
			\quad $	\TikZ{[xscale=.4, yscale=-.4]
						\Part{1,1}
						\node at (.5,.5) {\tiny$j$}; \node at (.5,1.5) {\tiny$i$}; 
			}$\\
		If $\varepsilon_j-\varepsilon_i\in P(\cc)-J$, then $\mathrm{box}_j$ is SE (weakly south and strictly east) of 
			$\mathrm{box}_i$:
			\quad $	\TikZ{[xscale=.4, yscale=-.4]
						\Part{2}
						\node at (.5,.5) {\tiny$i$}; \node at (1.5,.5) {\tiny$j$}; 
					}$
\item Markings: There is a marking on each of the diagonals $r_1$, $-r_1$, $r_2$ and $-r_2$. \\
		If $\varepsilon_i\in J$, $\mathrm{box}_i$ is NW of the marking on diagonal $c_i$:  
				\quad $	\TikZ{[xscale=.4, yscale=-.4]
						\Part{1}
						\filldraw [red] (1,1) circle (4pt);
						\node at (.5,.5) {\tiny$i$};
					}$\\
		If $\varepsilon_i\in P(\cc) - J$, then $\mathrm{box}_i$ is SE of the marking in diagonal $c_i$ :
				\quad $	\TikZ{[xscale=.4, yscale=-.4]
					\Part{1}
					\filldraw [red] (0,0) circle (4pt);
					\node at (.5,.5) {\tiny$i$};
			}$
\end{enumerate}

\noindent 
%A \emph{configuration of boxes} is a collection of boxes (points in $\ZZ\times \ZZ$ obtained from a pair $(\cc,J)$ 
%by ($\kappa1$-$\kappa4$).  

\noindent
A \emph{standard filling} of the boxes of $\kappa$ is a bijective function $S\colon \kappa \to \{-k, \ldots, -1, 1, \ldots k\}$ 
such that
\begin{enumerate}[\quad(S1)]
\item \label{S1} Symmetry: $S(\mathrm{box}_{-i}) = -S(\mathrm{box}_i)$.
\item \label{S2} Same diagonals: \\
If $0<i<j$ and $\mathrm{box}_i$ and $\mathrm{box}_j$ are on the same diagonal then $S(\mathrm{box}_i) <S(\mathrm{box}_j)$. %\\
%If $\mathrm{box}_i$ and $\mathrm{box}_j$ are on the same diagonal and $\mathrm{box}_j$ is SE of $\mathrm{box}_i$, then $S(\mathrm{box}_j) >S(\mathrm{box}_i)$. 
\item \label{S3} Adjacent diagonals: \\
If $0<i<j$, $\mathrm{box}_i$ and $\mathrm{box}_j$ are on adjacent diagonals, and $\mathrm{box}_j$ is NW of $\mathrm{box}_i$, then $S(\mathrm{box}_j) <S(\mathrm{box}_i)$. \\
If $0<i<j$, $\mathrm{box}_i$ and $\mathrm{box}_j$ are on adjacent diagonals, and $\mathrm{box}_j$ is SE of $\mathrm{box}_i$, then $S(\mathrm{box}_j) >S(\mathrm{box}_i)$. 
\item \label{S4} Markings: \\
	If $\mathrm{box}_i$ is on a marked diagonal and is SE of the marking, then $S(\mathrm{box}_i) >0$.\\
	If $\mathrm{box}_i$ is on a marked diagonal and is NW of the marking, then $S(\mathrm{box}_i) <0$.
\end{enumerate}
The \emph{identity filling} of a configuration $\kappa$ is the filling $F$ of the boxes of $\kappa$ given by
$F(\mathrm{box}_i) = i$,
for $i=-k,\ldots, -1, 1, \ldots, k$.
The identity filling of $\kappa$ is usually not a standard filling of $\kappa$
(see Example \ref{ex:smallconfiguration}).

\begin{example}\label{ex:smallconfiguration}
Let $k=4$, $r_1 = 1$, and $r_2 = 3$. Consider $\cc = (2,2,3)$. Then 
$$Z(\cc) = \{\vep_{1} -\vep_{2} \}  \qquad \text{ and } \qquad 
	P(\cc) = \left\{  \begin{matrix}
		\vep_{3},\ \vep_{3} - \vep_{1},\ \vep_{3} - \vep_{2}
		\end{matrix}
		\right\}.$$ 
The box configurations corresponding to $J = \{ \vep_{3}-\vep_{2} \}$ and 
$J = \{ \vep_{3}, \vep_{3} - \vep_{1}, \vep_{3} - \vep_{2} \}$
(filled with their identity fillings) are
\medskip

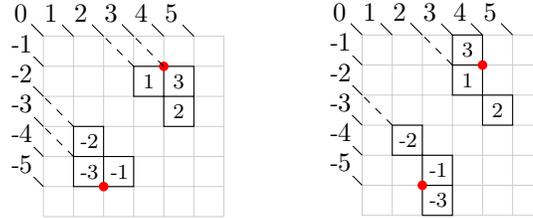
\centerline{\begin{tabular}{ccc}
\begin{tikzpicture}[xscale=.4, yscale=-.4]
	\draw [step=1cm, very thin, black!20!white] (-3,-3) grid (3,3);
	\draw (-1,2)--(-1,0)--(-2,0)--(-2,2)--(0,2)--(0,1)--(-2,1);
	\draw (1,-2)--(1,0)--(2,0)--(2,-2)--(0,-2)--(0,-1)--(2,-1);
	\filldraw [red] (1,-2) circle (4pt);	\filldraw [red] (-1,2) circle (4pt);
	\draw[dashed]
		(0,-2)--(-1,-3)	(1,-2)--(0,-3)
		(-2,0)--(-3,-1)	(-2,1)--(-3,0);
	\foreach \x in {0, 1, ..., 5} { \draw (\x-3, -3) to +(-.75,-.75) node[fill=white, inner sep=1.5pt]{\small $\x$};}
	\foreach \x in { 1, ..., 5} { \draw (-3, \x-3) to +(-.75,-.75) node[fill=white, inner sep=1.5pt]{\small -$\x$};}
	\node at (.5,-1.5) {\scriptsize $1$}; 	\node at (-.5,1.5) {\scriptsize -$1$}; 
	\node at (1.5,-.5) {\scriptsize $2$}; 	\node at (-1.5,.5) {\scriptsize -$2$}; 
	\node at (1.5,-1.5) {\scriptsize $3$}; 	\node at (-1.5,1.5) {\scriptsize -$3$}; 
\end{tikzpicture}&&
\begin{tikzpicture}[xscale=.4, yscale=-.4]
	\draw [step=1cm, very thin, black!20!white] (-3,-3) grid (3,3);
	\draw (1,-3)--(1,0)--(2,0)--(2,-1)--(0,-1)--(0,-3)--(1,-3)	(0,-2)--(1,-2);
	\draw (-1,3)--(-1,0)--(-2,0)--(-2,1)--(0,1)--(0,3)--(-1,3)	(0,2)--(-1,2);
	\filldraw [red] (1,-2) circle (4pt);	\filldraw [red] (-1,2) circle (4pt);
	\draw[dashed]	(0,-2)--(-1,-3)	(-2,0)--(-3,-1);
	\foreach \x in {0, 1, ..., 5} { \draw (\x-3, -3) to +(-.75,-.75) node[fill=white, inner sep=1.5pt]{\small $\x$};}
	\foreach \x in { 1, ..., 5} { \draw (-3, \x-3) to +(-.75,-.75) node[fill=white, inner sep=1.5pt]{\small -$\x$};}
	\node at (.5,-1.5) {\scriptsize $1$}; 	\node at (-.5,1.5) {\scriptsize -$1$}; 
	\node at (1.5,-.5) {\scriptsize $2$}; 	\node at (-1.5,.5) {\scriptsize -$2$}; 
	\node at (.5,-2.5) {\scriptsize $3$}; 	\node at (-.5,2.5) {\scriptsize -$3$}; 
\end{tikzpicture} \\ 
$\displaystyle J = \left\{ \vep_{3}-\vep_{2} \right\}$ && $\displaystyle J = \left\{ \vep_{3}, \vep_{3} - \vep_{1}, \vep_{3} - \vep_{2}\right\}$
\end{tabular}}

\noindent
For both configurations, the identity filling is not a standard filling. Examples of standard fillings of the configuration corresponding to $J = \{\vep_2 - \vep_3\}$ include
$$\TikZ{[xscale=.3, yscale=-.3]
	\BOX{0,-2}{$1$} \BOX{1,-2}{$2$} \BOX{1,-1}{$3$}
	\BOX{-1,1}{-$1$} \BOX{-2,1}{-$2$} \BOX{-2,0}{-$3$}
	\filldraw [red] (1,-2) circle (4pt);	\filldraw [red] (-1,2) circle (4pt);
}, \qquad 
\TikZ{[xscale=.3, yscale=-.3]
	\BOX{0,-2}{-$1$} \BOX{1,-2}{$2$} \BOX{1,-1}{$3$}
	\BOX{-1,1}{$1$} \BOX{-2,1}{-$2$} \BOX{-2,0}{-$3$}
	\filldraw [red] (1,-2) circle (4pt);	\filldraw [red] (-1,2) circle (4pt);
}, \qquad \text{ and } \qquad 
\TikZ{[xscale=.3, yscale=-.3]
	\BOX{0,-2}{-$2$} \BOX{1,-2}{$1$} \BOX{1,-1}{$3$}
	\BOX{-1,1}{$2$} \BOX{-2,1}{-$1$} \BOX{-2,0}{-$3$}
	\filldraw [red] (1,-2) circle (4pt);	\filldraw [red] (-1,2) circle (4pt);
}, \qquad \text{ but not } \qquad 
\TikZ{[xscale=.3, yscale=-.3]
	\BOX{0,-2}{-$3$} \BOX{1,-2}{-$2$} \BOX{1,-1}{$1$}
	\BOX{-1,1}{$3$} \BOX{-2,1}{$2$} \BOX{-2,0}{-$1$}
	\filldraw [red] (1,-2) circle (4pt);	\filldraw [red] (-1,2) circle (4pt);
}.$$
\end{example}

\begin{prop} \label{prop:stdtabbijection} \cite[Proposition 3.1]{DR}
Let $\kappa$ be a configuration of boxes corresponding to a local region $(\cc, J)$ with $\cc\in \ZZ^k$ or $\cc\in (\ZZ+\frac12)^k$.
For $w\in \cW_0$ let $S_w$ be the filling of 
the boxes of $\kappa$ given by 
$$S_w(\mathrm{box}_i) = w(i),
\quad\hbox{for $i=-k, \ldots, -1, 1, \ldots, k$.}
$$
The map 
$$\begin{matrix}
\cF^{(\cc,J)} &\longrightarrow &\{ \hbox{standard fillings $S$ of the boxes of $\kappa$}\} \\
w &\longmapsto &S_w
\end{matrix}
\qquad\hbox{is a bijection}.
$$
\end{prop}

\subsection{Calibrated representations of $TL^{\mathrm{ext}}_k$}

The following theorem determines which calibrated irreducible representations of 
$H_k^{\mathrm{ext}}$ are $TL_k^{\mathrm{ext}}$-modules.  In Theorem 
\ref{TLconfigurations} the answer is stated in terms of the configuration of boxes
$\kappa$.  By ($\kappa$1)--($\kappa$4) of Section \ref{configurationsofboxes}
the local region $(\cc,J)$ is determined by $\kappa$.  See Theorem \ref{thm:modules-in-tensor-space}
for the explicit conversion from $\kappa$ to $(\cc, J)$ for the irreducible calibrated
$TL_k$-modules.

\begin{thm} \label{TLconfigurations}Assume that if $r_1, r_2 \in \ZZ$ or $r_1, r_2 \in \ZZ+ \half$, then $r_2> r_1+1$. 
Let $\kappa$ be the configuration of boxes corresponding to a skew 
local region $(\cc, J)$ with $\cc\in \ZZ^k$ or $\cc\in (\ZZ+\frac12)^k$.
The irreducible calibrated $H_k^{\mathrm{ext}}$-module $H_k^{(z,\cc, J)}$ is a $TL_k^{\mathrm{ext}}$-module if and only if 
$\kappa$ is a $180^\circ$ rotationally symmetric skew shape with two rows of $k$ boxes each
(with or without markings),
\begin{equation}\label{eq:TL-shapes}
\TikZ{[xscale=.4,yscale=-.4]
	\draw (0,1) rectangle (5,2);\draw (7,0) rectangle (12,1);}
\qquad \text{or} \qquad
\TikZ{[xscale=.4,yscale=-.4]
	\draw (0,0) rectangle (5,1) (-2,1) rectangle (3,2);}\ .
\end{equation}
%Namely, whenever $\kappa$ has two rows, and if $\kappa$ has two boxes of content $d$ and one box of content $d-1$, then the box of content $d-1$ is between the two boxes of contend $d$ (is NW of one and SE of the other). 
%	\begin{scope}[thick,densely dotted]
%		\draw (1,0) to +(-1,-1) node[above left, inner sep=0pt]{\footnotesize$r_2 - \ell$};
%		\draw (-2,1) to +(-1,-1) node[above left, inner sep=0pt]{\footnotesize$\ell+1-r_2 - k$};
%		\draw (11,1) to +(1,1) node[below right, inner sep=0pt]{\footnotesize$r_2 - 1+k-\ell$};
%		\draw (8,2) to +(1,1) node[below right, inner sep=0pt]{\footnotesize$\ell-r_2$};
%	\end{scope}
%		\foreach \x in {(\B,1),(\A,1)}{\node[bV, blue] at \x {};}%inner beads
%		\foreach \x in {(0,2), (\A+\B,0)}{\node[bV, red] at \x {};}%outer beads
\end{thm}
\begin{proof} Let
$P = \{ p_{0}^{(\emptyset, 1^2)}, p_{0}^{(1^2, \emptyset)}, p_{0^\vee}^{(\emptyset, 1^2)},
p_{0^\vee}^{(1^2, \emptyset)},
p_1^{(1^3)}, p_2^{(1^3)}, \ldots, p_{k-2}^{(1^3)} \}
$
so that $TL_k$ is the quotient of $H_k$ by the ideal generated by the set $P$.
For $w\in \cF^{(\cc,J)}$ let $S_w$ be the standard tableau of shape $\kappa$ corresponding to $w$ 
as given in Proposition \ref{prop:stdtabbijection}.
For $j\in \{-k, \ldots, -1, 1, \ldots, k\}$,
$$\hbox{$(w\cc)_j$ is the diagonal number of $S_w(j)$,}
$$
where $S_w(j)$ is the box containing $j$ in $S_w$.

\smallskip

\noindent\emph{Step 1: Rewriting of the conditions for $pv_w=0$.}
By the construction of $H_k^{(z, \cc,J)}$ in Theorem \ref{thm:calibconst},  the module
$H_k^{(z, \cc,J)}$ has basis $\{v_w\ |\ w\in \cF^{(\cc,J)} \}$
and, if $w\in \cF^{(\cc,J)}$ then
\begin{align*}
\tau_i v_w = 0 \quad &\hbox{if and only if}\quad (w\cc)_{i+1}=(w\cc)_i\pm 1, \\
f_{\varepsilon_i-r_2}v_w = 0
\quad
&\hbox{if and only if}\quad
(w\cc)_i=r_2,  \quad \text{and}\\
f_{\vep_i-\vep_j+1}v_w = 0\quad
&\hbox{if and only if}\quad
(w\cc)_i = (w\cc)_j-1.
\end{align*}
Let $i\in \{1, \ldots, k-2\}$. Using the expansion of $p_i^{(1^3)}$ in terms of the $\tau_i$ from Proposition \ref{idempexpansion}, 
\begin{align*}
p_i^{(1^3)} v_w 
&=\tau_i\tau_{i+1}\tau_i v_w
-  t^{-\frac12} \tau_{i+1}\tau_i \frac{f_{\vep_{i+1}-\vep_{i+2}+1}}{f_{\vep_{i+1}-\vep_{i+2}}} v_w
- t^{-\frac12} \tau_i\tau_{i+1} \frac{f_{\vep_{i+1}-\vep_i+1}}{f_{\vep_{i+1}-\vep_i}} v_w
\\
&\qquad
+ t^{-1} \tau_i 
\frac{f_{\vep_{i+1}-\vep_{i+2}+1} f_{\vep_{i+2}-\vep_i+1} } 
{f_{\vep_{i+1}-\vep_{i+2}} f_{\vep_{i+2}-\vep_i} } v_w
+ t^{-1} \tau_{i+1}  
\frac{f_{\vep_{i+2}-\vep_i+1} f_{\vep_{i+1}-\vep_i+1} } 
{f_{\vep_{i+2}-\vep_i} f_{\vep_{i+1}-\vep_i}}  v_w
\\
&\qquad
- t^{-\frac32} 
\frac{f_{\vep_{i+1}-\vep_{i+2}+1} f_{\vep_{i+2}-\vep_i+1} f_{\vep_{i+1}-\vep_i+1} } 
{f_{\vep_{i+1}-\vep_{i+2}} f_{\vep_{i+2}-\vep_i} f_{\vep_{i+1}-\vep_i+1} } v_w,
\end{align*}
we consider the condition $p_i^{(1^3)} v_w =0$ term-by-term. First, 
 $\tau_i\tau_{i+1}\tau_i v_w = 0$ exactly when $(wc)_{i+1} = (wc)_i\pm 1$ 
or $(s_iwc)_{i+2}=(s_iwc)_{i+1}\pm 1$ or $(s_{i+1}s_iw)_{i+1} = (s_{i+1}s_iw)_i=\pm1$, i.e.\ when
$$
(w\cc)_{i+1}=(w\cc)_i\pm 1 \quad\hbox{or}\quad
(w\cc)_{i+2} = (w\cc)_i\pm 1 \quad\hbox{or}\quad
(w\cc)_{i+2}=(w\cc)_{i+1}\pm 1.
$$
Next,  $- t^{-\frac32} 
\frac{f_{\vep_{i+1}-\vep_{i+2}+1} f_{\vep_{i+2}-\vep_i+1} f_{\vep_{i+1}-\vep_i+1} } 
{f_{\vep_{i+1}-\vep_{i+2}} f_{\vep_{i+2}-\vep_i} f_{\vep_{i+1}-\vep_i+1} } v_w=0$ 
exactly when
$$
(w\cc)_{i+1}=(w\cc)_i+1 \quad\hbox{or}\quad
(w\cc)_{i+2} = (w\cc)_i+ 1 \quad\hbox{or}\quad
(w\cc)_{i+1}=(w\cc)_{i+1}+ 1.
$$
Thus $- t^{-\frac32} 
\frac{f_{\vep_{i+1}-\vep_{i+2}+1} f_{\vep_{i+2}-\vep_i+1} f_{\vep_{i+1}-\vep_i+1} } 
{f_{\vep_{i+1}-\vep_{i+2}} f_{\vep_{i+2}-\vep_i} f_{\vep_{i+1}-\vep_i+1} } v_w=0$ 
already implies $\tau_i\tau_{i+1}\tau_i v_w=0$, and similarly for the other terms in the expansion
of $p_i^{(1^3)}v_w=0$.
Thus $p_i^{(1^3)}v_w=0$ if and only if
\begin{equation}
(w\cc)_i = (w\cc)_{i+1}-1\ \ \hbox{or}\ \ 
(w\cc)_i = (w\cc)_{i+2}-1\ \ \hbox{or}\ \ 
(w\cc)_{i+1} = (w\cc)_{i+2}-1.
\label{icondition}
\end{equation}
Similarly, $p_0^{(\emptyset, 1^2)}v_w=0$  if and only if
\begin{equation}
(w\cc)_1 \in\{ r_1, r_2\}\ \hbox{or}\ (w\cc)_2\in \{r_1, r_2\} \ \hbox{or}\ 
(w\cc)_2=(w\cc)_1+1 \ \hbox{or}\ (w\cc)_2=(w\cc)_{-1}+1;
\label{first0condition}
\end{equation}
$p_0^{(1^2,\emptyset)}v_w=0$ if and only if
\begin{equation}
(w\cc)_1 \in\{ -r_1, -r_2\}\ \hbox{or}\ (w\cc)_2\in \{-r_1, -r_2\}\ \hbox{or}\ 
(w\cc)_2=(w\cc)_1+1\ \hbox{or}\ (w\cc)_2=(w\cc)_{-1}+1;
\end{equation}
$p_{0^\vee}^{(\emptyset, 1^2)}v_w=0$  if and only if
\begin{equation}
(w\cc)_1 \in\{ -r_1, r_2\}\ \hbox{or}\ (w\cc)_2\in \{-r_1, r_2\}\ \hbox{or}\  
(w\cc)_2=(w\cc)_1+1\ \hbox{or}\ (w\cc)_2=(w\cc)_{-1}+1;
\end{equation}
and $p_{0^\vee}^{(1^2,\emptyset)}v_w=0$ if and only if
\begin{equation}
(w\cc)_1 \in\{ r_1, -r_2\}\ \hbox{or}\ (w\cc)_2\in \{r_1, -r_2\}\ \hbox{or}\  
(w\cc)_2=(w\cc)_1+1\ \hbox{or}\ (w\cc)_2=(w\cc)_{-1}+1.
\label{last0condition}
\end{equation}

\smallskip

\noindent\emph{Step 2: If $\kappa$ is as in \eqref{eq:TL-shapes} and 
$w\in \cF^{(\cc, J)}$ and $p\in P$ then $pv_w=0$.}
Assume $\kappa$ has the form given in \eqref{eq:TL-shapes}
and let $w\in \cF^{(\cc,J)}$.  
Since $\kappa$ has only two rows the positions of $(-2, -1, 1, 2)$ in $S_w$ take one of the following forms:
$$\begin{matrix}
\TikZ{[xscale=.5,yscale=-.5]
	\draw[densely dotted] (2,0) to +(-.5, -.5) 
			node[above left, inner sep=2pt] {\scriptsize$(w\cc)_{-1}$};
	\draw[densely dotted] (-1,2) to +(.5,.5)
			node[below right, inner sep=2pt] {\scriptsize$(w\cc)_{1}$};
	\BOX{2,0}{\scriptsize-$1$}\BOX{3,0}{\scriptsize$2$}
	\BOX{-3,1}{\scriptsize-$2$}\BOX{-2,1}{\scriptsize$1$}
	}
&\qquad
&\TikZ{[xscale=.5,yscale=-.5]
	\draw[densely dotted] (2,0) to +(-.5, -.5) 
			node[above left, inner sep=2pt] {\scriptsize$(w\cc)_1$};
	\draw[densely dotted] (-1,2) to +(.5,.5)
			node[below right, inner sep=2pt] {\scriptsize$(w\cc)_{-1}$};
	\BOX{2,0}{\scriptsize$1$}\BOX{3,0}{\scriptsize$2$}
	\BOX{-3,1}{\scriptsize-$2$}\BOX{-2,1}{\scriptsize-$1$}
	}
&\qquad
&\TikZ{[xscale=.5,yscale=-.5]
	\draw[densely dotted] (0,0) to +(-.5, -.5) 
			node[above left, inner sep=2pt] {\scriptsize-$\frac12$};
	\draw[densely dotted] (1,0) to +(-.5, -.5) 
			node[above left, inner sep=2pt] {\scriptsize$\frac12$};
	\BOX{0,0}{\scriptsize-$2$}\BOX{1,0}{\scriptsize-$1$}
	\BOX{1,1}{\scriptsize$1$}\BOX{2,1}{\scriptsize$2$}
	}
&\qquad
&\TikZ{[xscale=.5,yscale=-.5]
	\draw[densely dotted] (0,0) to +(-.5, -.5) 
			node[above left, inner sep=2pt] {\scriptsize$0$};
	\draw[densely dotted] (1,0) to +(-.5, -.5) 
			node[above left, inner sep=2pt] {\scriptsize$1$};
	\draw[densely dotted] (0,1) to +(-.5, -.5) 
			node[above left, inner sep=2pt] {\scriptsize-$1$};
	\BOX{0,0}{\scriptsize-$1$}\BOX{1,0}{\scriptsize$2$}
	\BOX{0,1}{\scriptsize-$2$}\BOX{1,1}{\scriptsize$1$}
	}
\\
(wc)_1 < -\frac12,
&&(wc)_1 > \frac12,
&&(wc)_1 = -\frac12,
&&(wc)_1 = 0.
\end{matrix}
$$
In each of these cases, the conditions in \eqref{first0condition}--\eqref{last0condition}
give that $p_{0}^{(\emptyset, 1^2)}v_w=0$, $p_{0}^{(1^2, \emptyset)}v_w=0$, $p_{0^\vee}^{(\emptyset, 1^2)}v_w=0$ and $p_{0^\vee}^{(1^2, \emptyset)}v_w=0$.
Next, let $i\in\{1, \ldots, k-2\}$.  Since $\kappa$ has only two rows, then either $i$ or $i+1$ are in the same row
$$\TikZ{[xscale=.6,yscale=-.6]
	\draw[densely dotted] (0,0) to +(-.4, -.4) 
			node[above left, inner sep=2pt] {\scriptsize$(w\cc)_i$};
	\draw[densely dotted] (2,1) to +(.4,.4)
			node[below right, inner sep=2pt] {\scriptsize$(w\cc)_{i\!+\!1}$};
	\BOX{0,0}{\scriptsize$i$}\BOX{1,0}{\scriptsize$i\!+\!1$}
	}
	\qquad\TikZ{[xscale=.6,yscale=-.6]
	\draw[densely dotted] (0,0) to +(-.4, -.4) 
			node[above left, inner sep=2pt] {\scriptsize$(w\cc)_i$};
	\draw[densely dotted] (2,1) to +(.4,.4)
			node[below right, inner sep=2pt] {\scriptsize$(w\cc)_{i+2}$};
	\BOX{0,0}{\scriptsize$i$}\BOX{1,0}{\scriptsize$i\!+\!2$}
	}
$$
or $i$ and $i+2$ are in the same row.  Thus, by \eqref{icondition}, $p_iv_w=0$.
This completes the proof that if $\kappa$ is of the form \eqref{eq:TL-shapes} then 
$H_k^{(z, \cc,J)}$ is a $TL_k^{\mathrm{ext}}$-module.

\smallskip

\noindent\emph{Step 3: If $\kappa$ is not as in \eqref{eq:TL-shapes} 
then there exists $w\in \cF^{(\cc,J)}$ and 
$p\in P$ such that $pv_w\ne0$.}  Let $2k$ be the number of boxes in $\kappa$.  The proof
is by induction on $k$.

First, if $k=2$, then the condition \eqref{icondition} does not apply.
If $\cc = (r_1, r_2)$ then there are $8$ possibilities for $w\cc$: $(r_1, r_2)$, $(-r_1, r_2)$,
$(r_1, - r_2)$, $(-r_1, -r_2)$, $(r_2, r_1)$, $(-r_2, r_1)$, $(r_2, -r_1)$ and $(-r_2, -r_1)$.  None 
of these satisfy all of the conditions \eqref{first0condition}--\eqref{last0condition}. 
If $\cc = (c_1, c_1+1)$, then $s_1\cc=(c_1+1, c_1)$ does not satisfy \eqref{first0condition}
and $s_0s_1s_0s_1\cc = (-c, -c-1)$ does not satisfy \eqref{last0condition}.  
Thus that only the shaded local regions in Figure \ref{fig:rank2regulars} can have $pv_w=0$ for all
$p\in P$ and all $w\in \cF^{(\cc,J)}$.  For all of these, $\kappa$ is as in \eqref{eq:TL-shapes}.

Next, assume $k>2$ and proceed inductively.  If $H_k^{(z, \cc,J)}$ is a calibrated $TL_k^{\mathrm{ext}}$-module
then $\Res^{TL_k^{\mathrm{ext}}}_{TL_{k-1}^{\mathrm{ext}}}(H_k^{(z, \cc,J)})$ is calibrated $TL_{k-1}^{\mathrm{ext}}$-module.  This means
that if $S_w$ is a standard tableau of shape $\kappa$ and $S_w'$ is $S_w$ except with
the boxes $S_w(k)$ and $S_w(-k)$ removed and $\kappa'$ is the shape of $S_w'$, then
$\kappa'$ must be as in \eqref{eq:TL-shapes} and have only two rows.  The box $S_w(k)$ is
in a SE corner of $\kappa$ and the box $S_w(-k)$ is in a NW corner of $\kappa$.
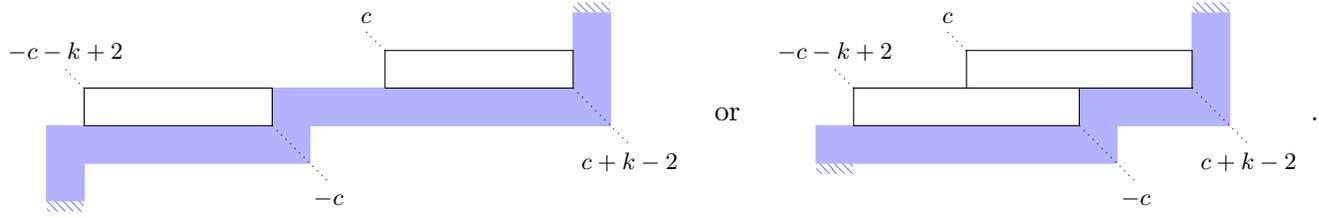
$$\TikZ{[xscale=.5,yscale=-.5]
	\filldraw[white, very thin, pattern=north west lines, pattern color=blue!50] 
		(5,1-0.05) rectangle (6,-1.29) (-9,4-.05) rectangle (-8,4.3);
	\filldraw[blue!30] (-9,2) to (-3,2) to (-3,1) to (5,1) to (5,-1) to (6,-1) 
				to (6,1) to (6,2) to (-2,2) to (-2,3) to (-8,3) to (-8,4) to (-9,4) to (-9,2);
	\draw[fill = white] (0,0) rectangle (5,1);\draw[fill = white] (-8,1) rectangle (-3,2);
	\draw[dotted] (0,0) to +(-.5,-.5) 
			node[above, inner sep=3pt] {\footnotesize$c$};
	\draw[dotted] (5,1) to +(1.5,1.5) 
			node[below, inner sep=3pt] {\footnotesize$c+k-2$};
	\draw[dotted] (-3,2) to +(1.5,1.5) 
			node[below, inner sep=3pt] {\footnotesize$-c$};
	\draw[ dotted] (-8,1) to +(-.5,-.5) 
			node[above, inner sep=3pt] {\footnotesize$-c-k+2$};
	}
\quad \text{or} \quad 
\TikZ{[xscale=.5,yscale=-.5]
	\filldraw[white, very thin, pattern=north west lines, pattern color=blue!50] 
		(6,1-0.05) rectangle (7,-1.29) (-4,3-0.05) rectangle (-3,3.3);
	\filldraw[blue!30] (-4,3) to (-4,2) to (3,2) to (3,1) to (6,1) to (6,-1)
				to (7,-1) to (7,1) to (7,2) to (4,2) to (4,3) to (-4,3);
	\draw (0,0) rectangle (6,1) (-3,1) rectangle (3,2);
	\draw[dotted] (0,0) to +(-.5,-.5) 
			node[above, inner sep=3pt, fill=white] {\footnotesize$c$};
	\draw[dotted] (6,1) to +(1.5,1.5) 
			node[below, inner sep=3pt] {\footnotesize$c+k-2$};
	\draw[dotted] (3,2) to +(1.5,1.5) 
			node[below, inner sep=3pt] {\footnotesize$-c$};
	\draw[dotted] (-3,1) to +(-.5,-.5) 
			node[above, inner sep=3pt] {\footnotesize$-c-k+2$};
	}\ .
$$

Given that $\kappa'$ has only two rows and $\kappa$ is obtained from $\kappa'$ by adding 
boxes that could contain $k$ and $-k$ in a standard tableau, the following are possibilities that we discard for $\kappa$:
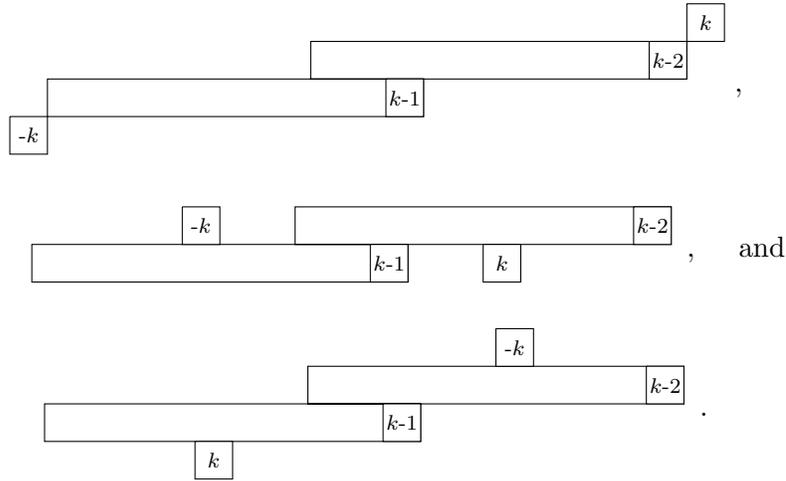
$$
\TikZ{[xscale=.5,yscale=-.5]
	\draw (0,0) rectangle (10,1) (-7,1) rectangle (3,2);
	\BOX{10,-1}{\scriptsize$k$}\BOX{-8,2}{\scriptsize-$k$}
	\BOX{9,0}{\scriptsize$k$-$2$}\BOX{2,1}{\scriptsize$k$-$1$}
	}\ ,
$$
$$\phantom{\text{ or}}\quad 
\TikZ{[xscale=.5,yscale=-.5]
	\draw (0,0) rectangle (10,1) (-7,1) rectangle (3,2);
	\BOX{5,1}{\scriptsize$k$}\BOX{-3,0}{\scriptsize-$k$}
	\BOX{9,0}{\scriptsize$k$-$2$}\BOX{2,1}{\scriptsize$k$-$1$}
	}\ , \quad \text{ and} 
$$
$$
\TikZ{[xscale=.5,yscale=-.5]
	\draw (0,0) rectangle (10,1) (-7,1) rectangle (3,2);
	\BOX{-3,2}{\scriptsize$k$}\BOX{5,-1}{\scriptsize-$k$}
	\BOX{9,0}{\scriptsize$k$-$2$}\BOX{2,1}{\scriptsize$k$-$1$}
	} \ .
$$
Namely, in each case there is a standard tableaux that has $k-2$, $k-1$ and $k$ in positions that 
do not satisfy the conditions in \eqref{icondition}.  Thus, in these cases,
there exists an $S_w$ of shape $\kappa$ for which $p^{(1^3)}_{k-2} v_w \ne 0$.  Further, in the case
$$
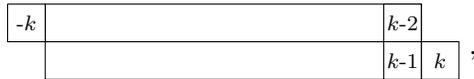
\TikZ{[xscale=.5,yscale=-.5]
	\draw (0,0) rectangle (10,1) (0,1) rectangle (10,2);
	\BOX{10,1}{\scriptsize$k$}\BOX{-1,0}{\scriptsize-$k$}
	\BOX{9,0}{\scriptsize$k$-$2$}\BOX{9,1}{\scriptsize$k$-$1$}
	}\ ,
$$
the shape $\kappa$ does not satisfy the $(w\cc)_{k-2}\ne (w\cc)_{k}$
from \eqref{skewlocalregiondefn} and the module $H_k^{(z, \cc,J)}$ is not calibrated.
In summary, unless $\kappa$ is of the form given in \eqref{eq:TL-shapes}
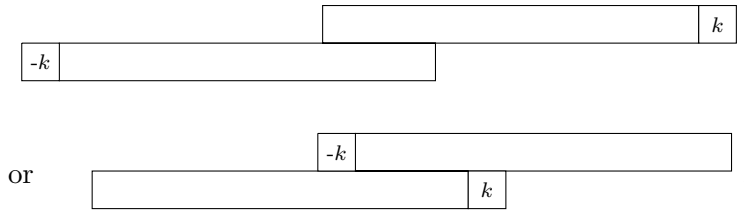
$$
\TikZ{[xscale=.5,yscale=-.5]
	\draw (0,0) rectangle (10,1) (-7,1) rectangle (3,2);
	\BOX{10,0}{\scriptsize$k$}\BOX{-8,1}{\scriptsize-$k$}
	} 
$$
$$\text{or} \qquad
\TikZ{[xscale=.5,yscale=-.5]
	\draw (0,0) rectangle (10,1) (-7,1) rectangle (3,2);
	\BOX{3,1}{\scriptsize$k$}\BOX{-1,0}{\scriptsize-$k$}
	}\ ,
$$
then either $H_k^{(z, \cc,J)}$ is not calibrated or 
there exists an $S_w$ of shape $\kappa$ for which $p^{(1^3)}_{k-2} v_w \ne 0$.
\end{proof}

\pagebreak~

\begin{figure}[h]\caption{
Calibrated representations of $H_2$ have regular central character. For each $(\cc, J)$ the corresponding configuration of boxes $\kappa$ is displayed
in the local region of chambers corresponding to the elements of $\cF^{(\cc,J)}$; 
only the boxes on positive diagonals are shown, since they determine 
$\kappa$ when $\cc$ is regular.
The local regions marked in blue are those that factor through the Temperley-Lieb quotient. 
}\label{fig:rank2regulars}
{\def\RTwo{2.6}
\def\ROne{1.2}
\def\ONE{.5}
\def\TOP{3.75}
\centerline{\begin{tikzpicture}[scale=3.25, every node/.style={inner sep=3pt}]
	\filldraw[gray!15] (0,\TOP)--(0,0)--(\TOP,\TOP)--(0,\TOP);
	%%%%%%%%%%%%%
	\node[WhiteCircle, inner sep=17.5pt](r1r2) at (\ROne,\RTwo){};
	\node[B, black] at (r1r2){};
		\node[above right] at (r1r2) {
			\begin{tikzpicture}[BoxArr]
			\BOX{0,0}{}\BOX{2,0}{}\bDOT{0,0}\bDOT{2,0}
			\end{tikzpicture}
			};		
		\node[above left] at (r1r2) {
			\begin{tikzpicture}[BoxArr]
				\BOX{0,0}{}\BOX{2,0}{}\bDOT{1,1}\bDOT{2,0}
			\end{tikzpicture}
			};
		\node[below right] at (r1r2) {
			\begin{tikzpicture}[BoxArr]
			\BOX{0,0}{}\BOX{2,0}{}\bDOT{0,0}\bDOT{3,1}
			\end{tikzpicture}
			};
		\node[below left] at (r1r2) {
			\begin{tikzpicture}[BoxArr]
				\BOX{0,0}{}\BOX{2,0}{}\bDOT{1,1}\bDOT{3,1}
			\end{tikzpicture}
			};
	\node[WhiteCircle, inner sep=17.5pt](r1-1r1) at (\ROne-\ONE,\ROne){};	
	\filldraw[blue!30] (\ROne-\ONE,\ROne) to +(-.189, -.189) arc (225:45:.267cm) to (\ROne-\ONE,\ROne);
	\node[B, blue!75!black] at (r1-1r1){};
		\node[above] at (r1-1r1) {
			\begin{tikzpicture}[BoxArr]
			\BOX{0,0}{}\BOX{1,0}{}\bDOT{1,0}
			\node at (3,1){};
			\end{tikzpicture}
			};	
		\node[above right, inner sep=0pt] at (r1-1r1) {
			\begin{tikzpicture}[BoxArr]
			\BOX{0,0}{}\BOX{0,1}{}\bDOT{0,0}
			\node at (-2.5,2){};
			\end{tikzpicture}
			};	
		\node[below right] at (r1-1r1) {
			\begin{tikzpicture}[BoxArr]
			\BOX{0,0}{}\BOX{0,1}{}\bDOT{1,1}
			\end{tikzpicture}
			};			
		\node[below left, inner sep=0pt] at (r1-1r1) {
			\begin{tikzpicture}[BoxArr]
			\BOX{0,0}{}\BOX{1,0}{}\bDOT{2,1}
			\node at (3.75,-.2){};
			\end{tikzpicture}
			};	
	\node[WhiteCircle, inner sep=17.5pt](r2-1r2) at (\RTwo-\ONE,\RTwo){};	
	\filldraw[blue!30] (\RTwo-\ONE,\RTwo) to +(-.189, -.189) arc (225:45:.267cm) to (\RTwo-\ONE,\RTwo);
	\node[B, blue!75!black] at (r2-1r2){};
		\node[above] at (r2-1r2) {
			\begin{tikzpicture}[BoxArr]
			\BOX{0,0}{}\BOX{1,0}{}\bDOT{1,0}
			\node at (3,1){};
			\end{tikzpicture}
			};	
		\node[above right, inner sep=0pt] at (r2-1r2) {
			\begin{tikzpicture}[BoxArr]
			\BOX{0,0}{}\BOX{0,1}{}\bDOT{0,0}
			\node at (-2.5,2){};
			\end{tikzpicture}
			};	
		\node[below right] at (r2-1r2) {
			\begin{tikzpicture}[BoxArr]
			\BOX{0,0}{}\BOX{0,1}{}\bDOT{1,1}
			\end{tikzpicture}
			};			
		\node[below left, inner sep=0pt] at (r2-1r2) {
			\begin{tikzpicture}[BoxArr]
			\BOX{0,0}{}\BOX{1,0}{}\bDOT{2,1}
			\node at (3.75,-.2){};
			\end{tikzpicture}
			};
	\node[WhiteCircle, inner sep=17.5pt](r1r1+1)  at (\ROne,\ROne+\ONE)  {};		
	\filldraw[blue!30] (\ROne,\ROne+\ONE)to +(-.189, -.189) arc (225:45:.267cm) to (\ROne,\ROne+\ONE);
	\node[B, blue!75!black] at (r1r1+1) {};
		\node[left] at (r1r1+1) {
			\begin{tikzpicture}[BoxArr]
			\BOX{0,0}{}\BOX{1,0}{}\bDOT{1,1}
			\node at (2.5,1){};
			\end{tikzpicture}
			};	
		\node[below left, inner sep=0pt] at (r1r1+1) {
			\begin{tikzpicture}[BoxArr]
			\BOX{0,0}{}\BOX{0,1}{}\bDOT{1,2}
			\node at (1.3,-1.75){};
			\end{tikzpicture}
			};	
		\node[below right] at (r1r1+1) {
			\begin{tikzpicture}[BoxArr]
			\BOX{0,0}{}\BOX{0,1}{}\bDOT{0,1}
			\end{tikzpicture}
			};			
		\node[above right, inner sep=0pt] at (r1r1+1) {
			\begin{tikzpicture}[BoxArr]
			\BOX{0,0}{}\BOX{1,0}{}\bDOT{0,0}
			\node at (-.25,3.5){};
			\end{tikzpicture}
			};	
	\node[WhiteCircle, inner sep=17.5pt](r2r2+1) at (\RTwo,\RTwo+\ONE) {};		
	\filldraw[blue!30] (\RTwo,\RTwo+\ONE) to +(-.189, -.189) arc (225:45:.267cm) to (\RTwo,\RTwo+\ONE);
	\node[B] at  (r2r2+1){};	
		\node[left] at (r2r2+1) {
			\begin{tikzpicture}[BoxArr]
			\BOX{0,0}{}\BOX{1,0}{}\bDOT{1,1}
			\node at (2.5,1){};
			\end{tikzpicture}
			};	
		\node[below left, inner sep=0pt] at (r2r2+1) {
			\begin{tikzpicture}[BoxArr]
			\BOX{0,0}{}\BOX{0,1}{}\bDOT{1,2}
			\node at (1.3,-1.75){};
			\end{tikzpicture}
			};	
		\node[below right] at (r2r2+1) {
			\begin{tikzpicture}[BoxArr]
			\BOX{0,0}{}\BOX{0,1}{}\bDOT{0,1}
			\end{tikzpicture}
			};			
		\node[above right, inner sep=0pt] at (r2r2+1) {
			\begin{tikzpicture}[BoxArr]
			\BOX{0,0}{}\BOX{1,0}{}\bDOT{0,0}
			\node at (-.25,3.5){};
			\end{tikzpicture}
			};	
	\node[WhiteCircle, inner sep=14pt](c1c1+1) at (.5*\ROne+.5*\RTwo - .5*\ONE,.5*\ROne+.5*\RTwo + .5*\ONE) {};	
	\filldraw[blue!30] (.5*\ROne+.5*\RTwo - .5*\ONE,.5*\ROne+.5*\RTwo + .5*\ONE) to +(-.151, -.151) arc (225:45:.214cm) to (.5*\ROne+.5*\RTwo - .5*\ONE,.5*\ROne+.5*\RTwo + .5*\ONE) ;
	\node[B, blue!75!black] at (c1c1+1){};
		\node[below right] at (c1c1+1) {
			\begin{tikzpicture}[BoxArr]
			\BOX{0,0}{}\BOX{0,1}{}
			\end{tikzpicture}
			};
		\node[above left] at (c1c1+1) {
			\begin{tikzpicture}[BoxArr]
			\BOX{0,0}{}\BOX{1,0}{}
			\end{tikzpicture}
			};
	\node[WhiteCircle, inner sep=11pt](c1r2) at (.5,\RTwo) {};		
	\node[B] at (c1r2){};
		\node[below] at (c1r2) {
		\begin{tikzpicture}[BoxArr]
			\BOX{0,0}{}\BOX{2,0}{}\bDOT{3,1}
		\end{tikzpicture}
		};
		\node[above] at (c1r2) {
		\begin{tikzpicture}[BoxArr]
			\BOX{0,0}{}\BOX{2,0}{}\bDOT{2,0}
		\end{tikzpicture}
		};
	\node[WhiteCircle, inner sep=11pt](c1r1) at (.2,\ROne){};		
	\node[B] at (c1r1){};
		\node[below] at (c1r1) {
			\begin{tikzpicture}[BoxArr]
				\BOX{0,0}{}\BOX{2,0}{}\bDOT{3,1}
			\end{tikzpicture}
			};
		\node[above] at (c1r1) {
			\begin{tikzpicture}[BoxArr]
			\BOX{0,0}{}\BOX{2,0}{}\bDOT{2,0}
			\end{tikzpicture}
			};
	\node[WhiteCircle, inner sep=11pt](r1c2) at (\ROne,.5*\RTwo+.5*\TOP){};		
	\node[B] at (r1c2) {};
		\node[left] at (r1c2) {
			\begin{tikzpicture}[BoxArr]
				\BOX{0,0}{}\BOX{0,-2}{}\bDOT{1,1}
			\end{tikzpicture}
			};
		\node[right] at (r1c2) {
			\begin{tikzpicture}[BoxArr]
				\BOX{0,0}{}\BOX{0,-2}{}\bDOT{0,0}
			\end{tikzpicture}
			};
	\node[WhiteCircle, inner sep=11pt](r2c2) at (\RTwo,\TOP-.15){};		
	\node[B] at (r2c2){};
		\node[left] at (r2c2) {
			\begin{tikzpicture}[BoxArr]
				\BOX{0,0}{}\BOX{0,-2}{}\bDOT{1,1}
			\end{tikzpicture}
			};
		\node[right] at (r2c2) {
			\begin{tikzpicture}[BoxArr]
				\BOX{0,0}{}\BOX{0,-2}{}\bDOT{0,0}
			\end{tikzpicture}
			};
	\node[WhiteCircle, inner sep=11pt](c1c2) at (.5*\ROne + .5*\RTwo,.5*\TOP+.5*\RTwo){};		
	\node[B] at (c1c2){};
			\node[below, outer sep=2pt] at (c1c2) {
			\begin{tikzpicture}[BoxArr]
				\BOX{0,0}{}\BOX{2,0}{}
			\end{tikzpicture}
			};
	\node[WhiteCircle, inner sep=14pt](c1-c1+1) at (.25*\ONE,.75*\ONE) {};	
	\filldraw[blue!30]  (.25*\ONE,.75*\ONE) to +(.151, -.151) arc (-45:135:.214cm) to  (.25*\ONE,.75*\ONE);
	\node[B, blue!75!black] at (c1-c1+1){};
		\node[below left] at (c1-c1+1) {
			\begin{tikzpicture}[BoxArr]
			\BOX{0,0}{}\BOX{0,1}{}
			\end{tikzpicture}
			};
		\node[above right] at (c1-c1+1) {
			\begin{tikzpicture}[BoxArr]
			\BOX{0,0}{}\BOX{1,0}{}
			\end{tikzpicture}
			};
	%%%%%%%%%%%	
	\draw[thick] (0,-.2)--(0,\TOP) node[above] {$c_1=0$};
	\draw[thick] (-.2,-.2)--(\TOP,\TOP) node[above right] {$c_1=c_2$};
	\draw[thick] (-.2,0)--(2,0) node[right] {$c_2=0$};
	\draw[dotted, thick, blue!75!black] (\TOP-\ONE,\TOP)--(-.2,-.2+\ONE) node[below left]{\small$c_2=c_1+1$};
	\draw[dotted, thick, blue!75!black] (-.2,\ONE+.2)--(\ONE+.2,-.2) node[below right]{\small$c_2=-c_1+1$};
	\draw[dotted, thick] (\ROne+.5,\ROne)--(-.2,\ROne) node[left] {\small$c_2=r_1$};
	\draw[dotted, thick] (\RTwo+.5,\RTwo)--(-.2,\RTwo) node[left] {\small$c_2=r_2$};
	\draw[dotted, thick] (\ROne,\ROne-.5)--(\ROne,\TOP) node[above] {\small$c_1=r_1$};
	\draw[dotted, thick] (\RTwo,\RTwo-.5)--(\RTwo,\TOP) node[above] {\small$c_1=r_2$};
	%%%%%%%%%%
	%%%%%%%%%%
\end{tikzpicture}}}
\end{figure}
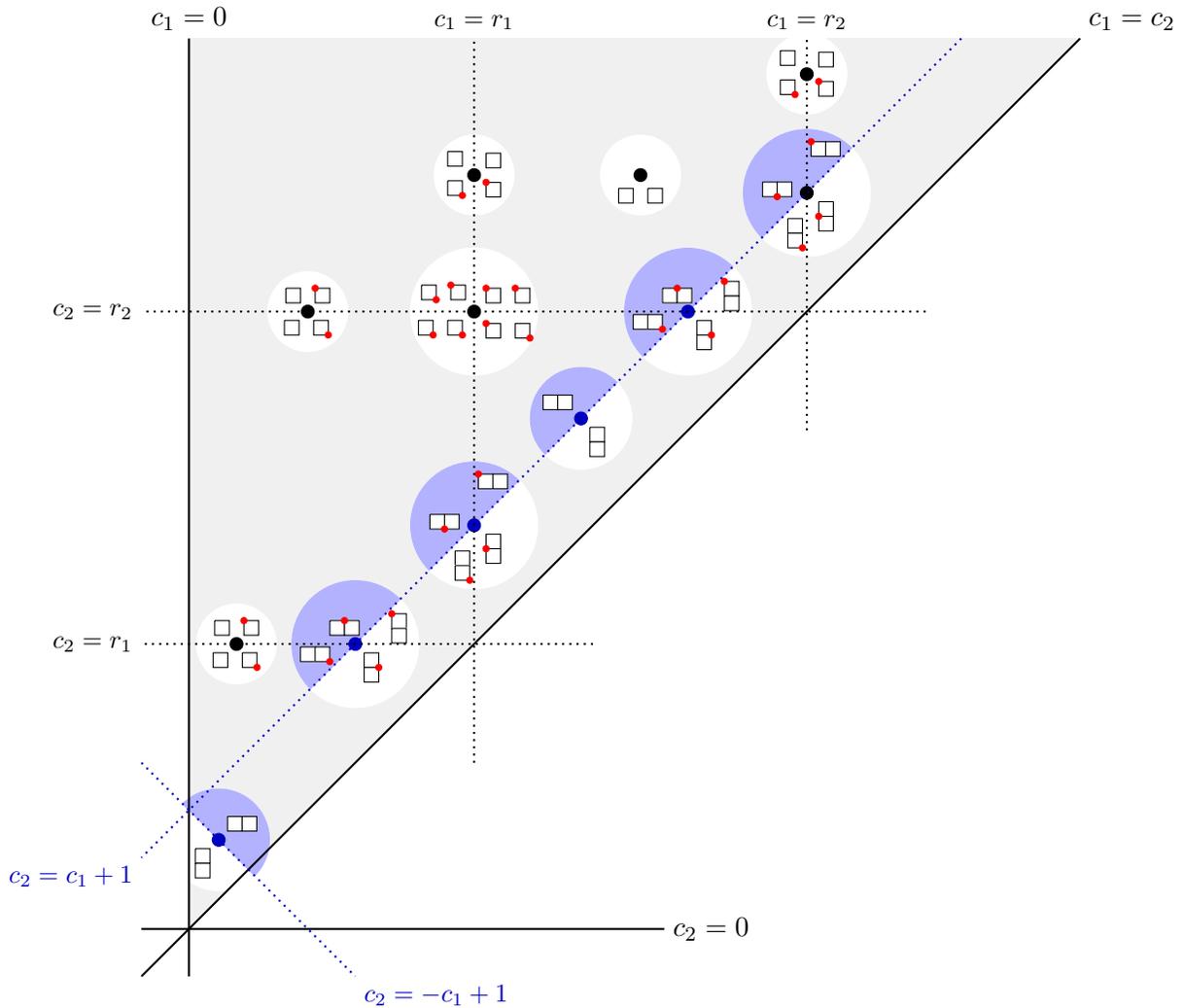

\pagebreak

The following proposition determines the action of the central element $Z$ on
each of the irreducible calibrated $TL_k^{\mathrm{ext}}$-modules.

\begin{prop} \label{centralchar}
Let $Z = W_1+W_1^{-1}+\cdots+W_k+W_k^{-1}$ be the central element of $TL_k^{\mathrm{ext}}$
studied in Theorem \ref{IIIexpansion}.  Assume that $\cc = (c, c+1, \ldots, c+k-1)$
and $H_k^{(z,\cc,J)}$ is an irreducible calibrated $TL_k^{\mathrm{ext}}$ as in Theorem 
\ref{TLconfigurations}.
If $v\in H_k^{(z,\cc,J)}$, then
$$Zv=\(t^\theta\)[k]v,
\qquad \hbox{where}\quad \theta = c+\hbox{$\frac{k-1}{2}$},
\ \ \(t^\theta\) = t^{\frac{\theta}{2}}+t^{-\frac{\theta}{2}}
\ \ \hbox{and}\ \  [k] = \frac{t^{\frac{k}{2}} - t^{-\frac{k}{2}}}{t^{\frac12}-t^{-\frac12}}.$$
\end{prop}
\begin{proof}
Let $v\in H_k^{(z,\cc,J)}$ be such that $W_i v=q^{c+i-1}$ for $i\in \{1, \ldots, k\}$.
Then $Zv_w = zv_w$ where
\begin{align*}
z&= t^{-(c+k-1)}+\cdots + t^{-(c+1)}+t^{-c}+t^c+t^{c+1}+\cdots +t^{c+k-1} \\
&= (t^{c+\frac{k-1}{2}}+t^{-(c+\frac{k-1}{2})}) (t^{-\frac{k-1}{2}}+\cdots+t^{\frac{k-1}{2}}) 
= (t^{\frac{\theta}{2}} + t^{-\frac{\theta}{2}}) 
\frac{ t^{\frac{k}{2}} - t^{-\frac{k}{2} } }{ t^{\frac12}-t^{-\frac12} }
= [\![t^\theta]\!] [k].
\end{align*}
Since $Z$ is a central element of $H_k^{\mathrm{ext}}$ and $H_k^{(z,\cc,J)}$ is a
simple $H_k^{\mathrm{ext}}$-module, Schur's lemma implies that if $v\in H_k^{(z,\cc,J)}$
then $Zv=zv$.
\end{proof}

\section{Schur-Weyl duality between $TL^{\mathrm{ext}}_k$ and $U_q\fgl_2$}
\def\({(\!(}
\def\){)\!)}

In this section we show that the Schur-Weyl duality studied in \cite{DR} provides
calibrated irreducible representations of the two boundary Temperley-Lieb algebra.
We classify these representations using the combinatorial classification of
irreducible calibrated $TL^{\mathrm{ext}}_k$ modules
obtained in Theorem \ref{TLconfigurations}. We follow the combinatorics of \cite[\S 4]{Da} and \cite[\S 5]{DR}. Note that similar constructions hold for replacing $\fgl_2$ with $\fsl_2$---see, for example, \cite[\S4]{Da}

%\subsection{Schur-Weyl duality between $TL^{\mathrm{ext}}_k$ and $U_q\fgl_2$}
The irreducible finite dimensional representations $L(\lambda)$ of $U_q(\fgl_2)$ 
are indexed by $\lambda = (\lambda_1, \lambda_2)\in \ZZ^2$ with $\lambda_1\ge \lambda_2$. By the Clebsch-Gordan formula or the Littlewood-Richardson rule (see \cite[(5.16)]{Mac}) 
$$L(a,0)\otimes L(b,0)
	= L(a+b,0)\oplus L(a+b-1,1) \oplus \cdots \oplus L(a+1, b-1)\oplus L(a,b),$$
and
$$L(\lambda_1,\lambda_2) \otimes L(1,0) 
	= \begin{cases} 
		L(\lambda_1+1, \lambda_2)  \oplus L(\lambda_1, \lambda_2 + 1), 
		&\text{if } \lambda_1 > \lambda_2,\\
		 L(\lambda_1+1, \lambda_2),  & \text{if } \lambda_1 = \lambda_2.
 	\end{cases}$$

Let $a,b\in \ZZ_{\ge 0}$ with $a\ge b$ and fix the simple $U_q\fgl_2$-modules
\begin{equation}
M=L(a,0), \qquad N=L(b,0)\qquad V = L(1,0).
\label{MNVchoices}
\end{equation}
We identify $(\lambda_1, \lambda_2)\in \ZZ^2$ with a left-justified arrangement of boxes with $\lambda_i$ boxes in the $i$th row. As in \cite[(5.28)]{DR} the \emph{shifted content} of a box in row $i$ and column $j$ as 
\begin{equation}\label{eq:shifted}
\tilde{c}(\textrm{box}) = j-i - \half(a+b-2)
\end{equation}
i.e.\ the shifted content is its diagonal number, where the box in the upper left corner has shifted content $- \half(a+b-2)$.  

For $j\in \ZZ_{\ge -1}$ let $\cP^{(j)}$ be an index set for the irreducible $U_q\fgl_2$-modules that appear in $M\otimes N\otimes V^{\otimes j}$. In particular,
\begin{enumerate}[\quad]
\item $\cP^{(-1)} = (a,0)$,
\quad $\cP^{(0)} = \{ (a+b-j,j)\ |\ j=0,1,\ldots, b\}$ \quad and
\item $\cP^{(j)} = \{ (a+b+j-\ell,\ell) ~|~  0\leq \ell\leq \half(j+a+b) \}$, for $j\ge 1$.
\end{enumerate}
Following \cite[\S 5.4]{DR}, the associated Bratteli diagram has 
\begin{enumerate}[\quad]
\item vertices on level $j$ labeled by the partitions in $\cP^{(j)}$, 
\item an edge $(a,0) \longrightarrow \mu$ for each $\mu \in\cP^{(0)}$, and
\item for each $j\ge 0$, $\mu \in\cP^{(j)}$ and $\lambda \in\cP^{(j+1)}$, there is 
$$\hbox{an edge}\quad \mu\to \lambda\quad 
\text{if $\lambda$ is obtained from $\mu$
by adding a box.}
$$
\end{enumerate}
The case when $a = 6$ and $b = 3$ is illustrated in Figure \ref{Fig:2bdryTLBrat}.

Assume $q \in \CC^\times$ and $a>b+2$ so that the generality condition 
$(a+1)-(b+1)\not\in \{0, \pm1, \pm 2\}$ of \cite[Theorem 5.5]{DR} is satisfied.  
Define
\begin{equation}
r_1 = \hbox{$\frac12$}(a-b) \qquad\hbox{and}\qquad
r_2 = \hbox{$\frac12$}(a+b+2),
\label{SWrvalues}
\end{equation}
and let $H^{\mathrm{ext}}_k$ be the extended two boundary Hecke algebra with parameters $t_0^{\frac12}$,
$t_k^{\frac12}$, and $t^{\frac12}$ given by
\begin{equation}
t^{\frac12} = q, \quad t_0 = -t^{r_2 - r_1} =  -q^{(b+1)},
\quad\hbox{and}\quad 
t_k = -t^{r_2+r_1} =  -q^{2(a+1)},
\label{SWparams}
\end{equation}
so that  $-t_k^{\frac12}t_0^{-\frac12} = -t^{r_1}$ and $t_k^{\frac12}t_0^{\frac12} = -t^{r_2}$ as in \cite[(3.5), (5.21)]{DR}. 
By \cite[Theorem 5.4 and (5.21)]{DR}
there are 
$$\hbox{commuting actions of 
$U_q\fgl_2$ and $H^{\mathrm{ext}}_k$ on $M\otimes N\otimes V^{\otimes k}$,}$$
where the $H^{\mathrm{ext}}_k$ action is given via R-matrices for the quantum group $U_q\fgl_2$.

\begin{thm}\label{thm:modules-in-tensor-space} Let $a,b\in \ZZ_{\ge 0}$ with $a>b+2$.
Let $q\in \CC^\times$ not a root of unity and let $H^{\mathrm{ext}}_k$ be the two 
boundary Hecke algebra
with parameters $t_0^{\frac12}$, $t_k^{\frac12}$ and $t^{\frac12}$ as in \eqref{SWparams}.
Let $U_q\fgl_2$ be the Drinfeld-Jimbo quantum group corresponding to $\fgl_2$ and let
$M$, $N$ and $V$ be the simple $U_q\fgl_2$-modules given in \eqref{MNVchoices}.
Then the $H^{\mathrm{ext}}_k$ action factors through $TL^{\mathrm{ext}}_k$ 
and, as $(U_q\fgl_2, TL^{\mathrm{ext}}_k)$-bimodules,
$$M\otimes N\otimes V^{\otimes k} \cong \bigoplus_{\lambda\in \cP^{(k)} } L(\lambda)\otimes B_k^\lambda
\qquad\hbox{with}\quad
B_k^{(a+b+k-\ell,\ell)}\cong H^{(z,\cc,J)},$$
where $z =(-1)^kq^{(a+b-\ell)(a+b-\ell-1)+\ell(\ell-3)-a(a-1)-b(b-1)-k(a+b-2)}$ and
$(\cc,J)$ is the local region corresponding to the configuration $\kappa$ of $2k$ boxes
\begin{equation}
{\def\A{6} \def\B{3}
\TikZ{[xscale=.4,yscale=-.4]
	\draw (1,0) rectangle (11,1) (-2,1) rectangle (8,2);
	\begin{scope}[thick,densely dotted]
		\draw (1,0) to +(-1,-1) node[above left, inner sep=0pt]{\footnotesize$r_2 - \ell$};
		\draw (-2,1) to +(-1,-1) node[above left, inner sep=0pt]{\footnotesize$\ell+1-r_2 - k$};
		\draw (11,1) to +(1,1) node[below right, inner sep=0pt]{\footnotesize$r_2 - 1+k-\ell$};
		\draw (8,2) to +(1,1) node[below right, inner sep=0pt]{\footnotesize$\ell-r_2$};
	\end{scope}
		\foreach \x in {(\B,1),(\A,1)}{\node[bV, blue] at \x {};}%inner beads
		\foreach \x in {(0,2), (\A+\B,0)}{\node[bV, red] at \x {};}%outer beads
}}
\label{kappapicture}
\end{equation}
that has $k$ boxes in each row, the shifted content of the leftmost box in the first row is $r_2-\ell$, the shifted content of the leftmost box in the second row is $\ell+1-r_2-k$.  Between the rows there are blue markers in diagonals with shifted content $\pm r_1$ and there are 
red markers in diagonals with shifted content $\pm r_2$,  as pictured.
Explicitly, $\cc=(c_1, c_2, \ldots, c_k)$ is the sequence of 
$$\text{absolute values of} \quad  
c,\ c+1,\ \cdots,\ c+k-1, \quad \text{ where } c = \hbox{$\frac12$}(a+b)-\ell+1,
$$
arranged in increasing order;
and $J$ is the union of 
$$J_1 = \begin{cases} \emptyset,\ & \hbox{if $a\ge b\ge \ell$,}\\
			 \{ \varepsilon_{\ell-b}\},\ & \hbox{if $a \ge  \ell > b$,}\\
			  \{ \varepsilon_{a-b}\}, & \hbox{if $\ell > a > b$},
			  \end{cases}$$
			  and			  
$$J_2= \begin{cases} \emptyset,\ & \hbox{if $\half(a+b+2) > \ell$,}\\
			  \{\vep_2-\vep_1, \vep_4-\vep_3, \ldots, 
				\vep_{2\ell-a-b}-\vep_{2\ell-a-b-1}\}, & \hbox{if $\ell\ge \half(a+b+2)$ and $a+b$ even,}\\
			  \{
			  	\vep_3-\vep_2, \vep_5-\vep_4, \ldots, 
				\vep_{2\ell-a-b}-\vep_{2\ell-a-b-1}\}, & \hbox{if $\ell\ge \half(a+b+2)$ and $a+b$ odd.}\end{cases}
$$
\end{thm}

\begin{proof}  Fix $\lambda = (a+b+k-\ell,\ell)\in \cP^{(k)}$.  
The sum of the contents of the boxes in
$\lambda$ is 
\begin{align*}
\sum_{\mathrm{box}\in \lambda} c(\mathrm{box}) 
&= (0+1+\ldots + (a+b+k-\ell-1))+(-1+0+\cdots+\ell-2) \\
&= \hbox{$\frac12$}(a+b+k-\ell-1)(a+b+k-\ell)+\hbox{$\frac12$}\ell(\ell-3).
\end{align*}
By \cite[Theorem 5.5 and (5.35)]{DR},
$\cB_k^\lambda\cong H_k^{(z,\cc,J)}$ where 
$$z=(-1)^k q^{2c_0},\quad\hbox{where}\quad
c_0= -\hbox{$\frac12$}(k(a+b-2)+a(a-1)+b(b-1)) + \sum_{\mathrm{box}\in \lambda} c(\mathrm{box}),
$$
and $\cc$ and $J$ and the corresponding configuration
$\kappa$ of $2k$ boxes are determined as follows.

\smallskip
%As in \cite[(5.28)]{DR}, the shifted content of a box is 
%$\tilde c(\mathrm{box}) = c(\mathrm{box}) -\hbox{$\frac12$}(a+b-2)$.
Place markers at the NW corner of the boxes at positions $(1, a+b+1)$, $(2,a+1)$, $(2, b+1)$, and  $(3,1)$ so that these markers are in the diagonals with shifted contents $\pm r_1$ and $\pm r_2$.
{\def\A{6} \def\B{3} \def\H{.65} \def\W{13} \def\L{8}
$$\lambda = (a+b+k-\ell, \ell) = \quad
\TikZ{[xscale=.4,yscale=-.4]
		\draw (0,0) rectangle (\W,1);
		\draw (0,1) to (0,2) to (\L,2) to (\L,1);
		\draw (\B,1) to (\B,2) (\A,1) to (\A,0) (\A+\B,0) to (\A+\B,1);
		\foreach \x in {(\B,1),(\A,1)}{\node[bV, blue] at \x {};}
		\foreach \x in {(0,2), (\A+\B,0)}{\node[bV, red] at \x {};}
		\begin{scope}[<->, shorten >=2pt, shorten <=2pt, every node/.style={fill=white,inner sep=2pt}]
		\draw (0,2+\H) to node{\footnotesize$b$} (\B,2+\H);
		\draw(\B,2+\H) to node{\footnotesize$\ell-b$} (\L,2+\H);
		\draw(0,-\H) to node{\footnotesize$a$} (\A,-\H);
		\draw(\A,-\H) to node{\footnotesize$b$} (\A+\B,-\H);
		\draw (\A+\B,-\H) to node{\footnotesize$k-\ell$} (\W,-\H);
		\end{scope}
		\foreach \x in {0,\A,\A+\B,\W}{\draw (\x,-\H-.3) to (\x,-\H+.3);}
		\foreach \x in {0,\B,\L}{\draw (\x,2+\H-.3) to (\x,2+\H+.3);}
		\draw[thick,densely dotted] (\L,2) to +(1,1) node[below right, inner sep=0pt]{\footnotesize$\ell-r_2$};
		\draw[thick,densely dotted] (\W,1) to +(1,1) node[below right, inner sep=0pt]{\footnotesize$r_2 - 1 + k -\ell$};
		}$$}
Following \cite[(5.27)]{DR}, let
$$S^{(0)}_{\mathrm{max}} = \begin{cases}
(a+b-\ell, \ell), &\hbox{if $a\ge b \ge \ell$,} \\
(a,b), &\hbox{if $a\ge \ell\ge b$}
\end{cases}
$$
(since $a\ge b$ we are in the left case of \cite[(5.15)]{DR} with $c=d=1$ so that $\mu^c = \min(\ell,b)$
and $S^{(0)}_{\mathrm{max}} = \mathring{\mu} = (a+b-\mu^c,\mu^c)$):  
{\def\A{6} \def\B{3} \def\H{.65} \def\W{13} \def\L{2}
$$
\begin{matrix}
\TikZ{[xscale=.4,yscale=-.4]%k=12
		\filldraw[black!20] (0,0) rectangle (\A+1,1); \draw (\A+1,0) to (\A+1,1);
		\draw[fill = black!20] (0,1) to (0,2) to (\L,2) to (\L,1);
		\draw (0,1) to (0,0) to (\W,0) to (\W,1) to (\L,1); \draw (\A+1,0) to (\A+1,1);
		\node[fill = black!10, inner sep = 1.5pt, rounded corners, below right, outer sep=1pt] at (0,0) {\tiny$\displaystyle S^{(0)}_{\mathrm{max}}$};
		\begin{scope}[<->, shorten >=2pt, shorten <=2pt, every node/.style={fill=white,inner sep=2pt}]
		\draw (0,2+\H) to node{\footnotesize$\ell$} (\L,2+\H);
		\draw(0,-\H) to node{\footnotesize$a+b-\ell$} (\A+1,-\H);
		\draw(\A+1,-\H) to node{\footnotesize$k$} (\W,-\H);
		\end{scope}
		\foreach \x in {0,\A+1,\W}{\draw (\x,-\H-.3) to (\x,-\H+.3);}
		\foreach \x in {0,\L}{\draw (\x,2+\H-.3) to (\x,2+\H+.3);}
		\foreach \x in {(\B,1),(\A,1)}{\node[bV, blue] at \x {};}
		\foreach \x in {(0,2), (\A+\B,0)}{\node[bV, red] at \x {};}
		}\\\text{if } a \ge b \ge \ell
	\end{matrix}\qquad\qquad 
	\def\L{8}
	\begin{matrix}
\TikZ{[xscale=.4,yscale=-.4]%k=12
		\filldraw[black!20] (0,0) rectangle (\A,1); \draw (\A,0) to (\A,1);
		\draw[fill = black!20] (0,1) to (0,2) to (\B,2) to (\B,1); \draw (\B,2) to (\L,2) to (\L,1);
		\draw (0,1) to (0,0) to (\W,0) to (\W,1) to (\B,1); \draw (\A,0) to (\A,1);
		\node[fill = black!10, inner sep = 1.5pt, rounded corners, below right, outer sep=1pt] at (0,0) {\tiny$\displaystyle S^{(0)}_{\mathrm{max}}$};
		\begin{scope}[<->, shorten >=2pt, shorten <=2pt, every node/.style={fill=white,inner sep=2pt}]
		\draw (0,2+\H) to node{\footnotesize$b$} (\B,2+\H);
		\draw (\B,2+\H) to node{\footnotesize$\ell-b$} (\L,2+\H);
		\draw(0,-\H) to node{\footnotesize$a$} (\A,-\H);
		\draw(\A,-\H) to node{\footnotesize$k-(\ell-b)$} (\W,-\H);
		\end{scope}
		\foreach \x in {0,\A,\W}{\draw (\x,-\H-.3) to (\x,-\H+.3);}
		\foreach \x in {0,\B,\L}{\draw (\x,2+\H-.3) to (\x,2+\H+.3);}
		\foreach \x in {(\B,1),(\A,1)}{\node[bV, blue] at \x {};}
		\foreach \x in {(0,2), (\A+\B,0)}{\node[bV, red] at \x {};}
		}\\\text{if } a \ge \ell > b \text{~ or ~} \ell > a \ge b
	\end{matrix}\ .$$}

\noindent By \cite[(5.35)]{DR}, the 
corresponding configuration of boxes is 
$\kappa = \mathrm{rot}(\lambda/S^{(0)}_{\mathrm{max}}) \cup \lambda/S^{(0)}_{\mathrm{max}}$,
as pictured above in \eqref{kappapicture}.

To determine $(\cc, J)$, use the conditions ($\kappa$1)--($\kappa$4) of Section \ref{configurationsofboxes} which specify the relation between 
$\kappa$ and $(\cc, J)$.
First index the boxes of $\kappa$ with $-k, \ldots, -1, 1, \ldots, k$ by diagonals, left to right, and
NW to SE along diagonals.
The sequence $\cc = (c_1, \ldots, c_k)$ with $0\le c_1\le c_2\le \cdots \le c_k$
is the sequence of the absolute values of the shifted contents of 
boxes in the first row of $\kappa$.   Next, the set $J$ is determined as follows. 

\begin{enumerate}[1.]
\item By ($\kappa$4), the set $J$ contains $\varepsilon_i$ if $i>0$ and $\mathrm{box}_i$ is NW of the marker in the diagonal with shifted content $r_1$ or $r_2$ in $\kappa$. This occurs on diagonal $r_1$ whenever $\ell>b$  (marked in blue),
$$\vep_{\ell-b}\in J\ \hbox{if $a\ge \ell>b$}
\qquad\hbox{and}\qquad
\vep_{a-b}\in J\ \hbox{if $\ell>a\ge b$;}
$$
and $J$ contains no roots of the form $\vep_j$ when $a \ge b \ge \ell$.

\item By ($\kappa$3), the set $J$ contains $\varepsilon_j-\varepsilon_i$ 
if $j>i>0$ and $\mathrm{box}_i$ and $\mathrm{box}_j$ are in the same column of $\kappa$ 
(so that $\mathrm{box}_i$ and 
$\mathrm{box}_j$ are in adjacent diagonals and $\mathrm{box}_j$ is NW of $\mathrm{box}_i$).  
This occurs exactly when 
$0 \ge r_2 - \ell = \half(a+b+2) -\ell$.
If $\ell\ge \frac12(a+b+2)$ and $a+b$ is even then the
boxes indexed $1, 3, \dots, 1 + 2(\ell - \half(a + b + 2)) = 2\ell - (a + b + 1)$
are in the second row directly below boxes of index $2,4,\ldots, 2\ell-a-b$.
If $\ell\ge \frac12(a+b+2)$ and $a + b$ is odd then boxes $2, 4, \dots, 2(\ell - \half(a + b +1))$, directly below boxes of index $3,5, \ldots, 2\ell-a-b$:
$$
\begin{matrix}
\begin{tikzpicture}[xscale=.5,yscale=-.5]
	\draw (-1,0) rectangle (12,1) (-4,1) rectangle (9,2);
	\begin{scope}[thick,densely dotted, blue]
		\draw (0,0) to +(-.75,-.75) node[above left, inner sep=0pt]{\footnotesize$r_2 - \ell$};
		\draw (9,2) to +(1,1) node[below right, inner sep=0pt]{\footnotesize$\ell-r_2$};
			\draw (3,0) to +(-.75,-.75) node[above left, inner sep=0pt]{\footnotesize$0$};
	\end{scope}
		\foreach \x in {2, ..., 6, 8, 0}{\draw (\x,0) to  (\x,2);}
			\draw (-1,1) to (-1,2) (9,0) to (9,1) (11,0) to (11,1) (-3,1) to (-3,2);
		\node at (4.5, 1.5) {\scriptsize $1$};
		\node at (4.5, .5) {\scriptsize $2$};
		\node at (5.5, 1.5) {\scriptsize $3$};
		\node at (5.5, .5) {\scriptsize $4$};
		\node at (3.5, .5) {\scriptsize -$1$};
		\node at (3.5, 1.5) {\scriptsize -$2$};
		\node at (2.5, .5) {\scriptsize -$3$};
		\node at (2.5, 1.5) {\scriptsize -$4$};
		\node at (7, 1.5) {\scriptsize $\cdots$};
		\node at (7, .5) {\scriptsize $\cdots$};
		\node at (10, .5) {\scriptsize $\cdots$};
		\node at (8.5, .5) {\tiny $*\!+\!1$};
		\node at (11.5, .5) {\scriptsize $k$};
		\node at (-.5, 1.5) {\tiny -$*$-$1$};
		\node at (-.5, .5) {\scriptsize -$*$};
		\node at (-3.5, 1.5) {\scriptsize -$k$};
		\node at (1, 1.5) {\scriptsize $\cdots$};
		\node at (1, .5) {\scriptsize $\cdots$};
		\node at (-2, 1.5) {\scriptsize $\cdots$};
		\draw[->, shorten > = 3pt] (10.5,2) node[right] 
			{\scriptsize$2\ell - a - b - 1$} to 
				(8.5, 1.5) node {\scriptsize $*$};
\end{tikzpicture}
\qquad
&
\hbox{if $a+b$ is even,} 
\\
\begin{tikzpicture}[xscale=.5,yscale=-.5]
	\draw (0,0) rectangle (12,1) (-3,1) rectangle (9,2);
	\begin{scope}[thick,densely dotted, blue]
		\draw (0,0) to +(-.75,-.75) node[above left, inner sep=0pt]{\footnotesize$r_2 - \ell$};
		\draw (9,2) to +(1,1) node[below right, inner sep=0pt]{\footnotesize$\ell-r_2$};
		\draw (6,2) to +(1,1) node[below right, inner sep=0pt]{\footnotesize$\half$};		\draw (3,0) to +(-.75,-.75) node[above left, inner sep=0pt]{\footnotesize$-\half$};
	\end{scope}
		\foreach \x in {3, ..., 6, 8, 1}{\draw (\x,0) to  (\x,2);}
			\draw (0,1) to (0,2) (9,0) to (9,1) (11,0) to (11,1) (-2,1) to (-2,2);
		\node at (4.5, .5) {\scriptsize $1$};
		\node at (5.5, 1.5) {\scriptsize $2$};
		\node at (5.5, .5) {\scriptsize $3$};
		\node at (7, 1.5) {\scriptsize $\cdots$};
		\node at (7, .5) {\scriptsize $\cdots$};
		\node at (10, .5) {\scriptsize $\cdots$};
		\node at (8.5, .5) {\tiny $*\!+\!1$};
		\node at (11.5, .5) {\scriptsize $k$};
		\node at (4.5, 1.5) {\scriptsize -$1$};
		\node at (3.5, .5) {\scriptsize -$2$};
		\node at (3.5, 1.5) {\scriptsize -$3$};
		\node at (.5, 1.5) {\tiny -$*$-$1$};
		\node at (.5, .5) {\scriptsize -$*$};
		\node at (-2.5, 1.5) {\scriptsize -$k$};
		\node at (2, 1.5) {\scriptsize $\cdots$};
		\node at (2, .5) {\scriptsize $\cdots$};
		\node at (-1, 1.5) {\scriptsize $\cdots$};
		\draw[->, shorten > = 3pt] (10.5,2) node[right] {\scriptsize$2\ell - a - b - 1$} to 
				(8.5, 1.5) node {\scriptsize $*$};
\end{tikzpicture} 
&%\begin{array}{l}
%\hbox{$\ell\ge \frac12(a+b+2)$ and} \\
\hbox{if $a+b$  is odd.} 
%\end{array}
\end{matrix}$$
So $J$ contains 
$$\vep_2-\vep_1, \vep_4-\vep_3, \ldots, \vep_{2\ell-a-b}-\vep_{2\ell-a-b-1} \quad 
\text{if $\ell\ge \frac12(a+b+2)$ and $a+b$ is even, or}$$
$$\vep_3-\vep_2, \vep_5-\vep_4, \ldots, \vep_{2\ell-a-b}-\vep_{2\ell-a-b-1} \quad 
\text{if $\ell\ge \frac12(a+b+2)$ and $a+b$ is odd.}$$

\item Also by $(\kappa3)$, the set $J$ contains $\varepsilon_j+\varepsilon_i$ if $j>i>0$, and $\mathrm{box}_j$ is directly above $\mathrm{box}_{-i}$, which does not occur.
\end{enumerate}
In this way $\cc$ and $J$ are determined from $\kappa$.
Since all of these $H_k^{(z,\cc,J)}$
satisfy the conditions of Theorem \ref{TLconfigurations}, 
it follows that the $H^{\mathrm{ext}}_k$-action on $M\otimes N\otimes V^{\otimes k}$ factors through $TL^{\mathrm{ext}}_k$.
\end{proof}

\begin{remark}\label{dimformulas}
The dimension of $B_k^{(a+b+k-\ell, \ell)}$ is the number of paths in the 
Bratteli diagram from a shape on level $0$ to the shape $\lambda = (a+b+k-\ell,\ell)$ on level $k$.
Summing over the shapes on level $0$ for which there is a path to $\lambda$ gives
$$\dim(B^{(a+b+k-\ell,\ell)}) = \sum_{c = \max(0, \ell-k)}^{\min(b,\ell)}
f^{\lambda/(a+b-c,c)},$$
where $f^{\lambda/\mu}$ is the number of standard tableaux of skew shape $\lambda/\mu$.  If
$\ell \le a+b-c$ then the second row of $\lambda/(a+b-\ell,\ell)$ does not overlap the first row and thus
$$f^{\lambda/(a+b-c,c)} = \binom{k}{\ell-c}
\qquad\hbox{if $\ell \le a+b-c$.}$$
Since $c \le \min(b,\ell)$, the case $\ell>a+b-c$ can occur only when $\ell>a\ge b$, in which case 
{\def\A{8} \def\B{3} \def\C{3} \def\K{18} \def\H{.65} \def\L{12} 
$$(a+b+k-\ell,\ell)/(a+b-c,c) = \quad
\TikZ{[xscale=.4,yscale=-.4]
		\filldraw[black!20] (0,0) to (\A + \B - \C,0) to (\A + \B - \C,1) to (\C,1) to (\C,2) to (0,2) to (0,0);
		\draw (0,0) rectangle (\A + \B + \K - \L,1);
		\draw (0,1) to (0,2) to (\L,2) to (\L,1);
		\draw (\A + \B - \C,0) to (\A + \B - \C,2) (\C,2) to (\C,1); 
		\begin{scope}[<->, shorten >=2pt, shorten <=2pt, every node/.style={fill=white,inner sep=2pt}]
		\draw (0,2+\H) to node{\footnotesize$c$} (\C,2+\H);
		\draw (0,1.5) to node{\footnotesize$\ell$} (\L,1.5);
		\draw (\A+\B-\C,2+\H) to node[below, inner sep=1pt, outer sep=2pt]{\footnotesize$\ell - (a+b-c)$} (\L,2+\H);
		\draw(0,-\H) to node{\footnotesize$a+b-c$} (\A + \B - \C,-\H);
		\draw(\A + \B - \C,-\H) to node{\footnotesize$k-\ell+c$} (\A + \B + \K - \L,-\H);
		\end{scope}
		}\ ,$$}
so that
\begin{align*}
f^{(a+b+k-\ell,\ell)/(a+b-c,c)}
	= \sum_{j = \ell - (a+b-c)}^{k+\ell - c} f^{(k-j,j)}
	&= \sum_{j = \ell - (a+b-c)}^{\min(k-(\ell - c), \ell-c)} \binom{k}{j} - \binom{k}{j-1}\\
	&= \binom{k}{\ell-c} - \binom{k}{\ell - (a+b-c) - 1}.
\end{align*}
Namely, the first equality comes from the Pieri formula and the expansion of a 
skew Schur function by Littlewood-Richardson coefficients (see \cite[(5.16)]{Mac} for the 
Pieri formula and \cite[(5.2) and (5.3)]{Mac} for Littlewood-Richardson coefficients)
and the second equality comes from the number of standard tableaux of a two row shape 
as given, for example, in \cite[Theorem 2.8.5 and Lemma 2.8.4]{GHJ}.
\end{remark}

The following examples reference the node label styles in Figure \ref{Fig:2bdryTLBrat}.

\tikzstyle{FULL}=[fill=yellow!20, circle, inner sep=2pt, draw]
\tikzstyle{SumToL}=[fill=blue!15, draw]
\tikzstyle{SumToBKL}=[fill=cyan!15, diamond, inner sep=1.5pt, draw]
\tikzstyle{SumToA}=[fill=green!15,  regular polygon,regular polygon sides=3, inner sep=1pt, draw]
\tikzstyle{Overlapping}=[fill=red!15, regular polygon,regular polygon sides=8, inner sep=1pt, draw]
\tikzstyle{OverlappingB}=[fill=orange!15, regular polygon,regular polygon sides=6, inner sep=1pt, draw]
\begin{example}Let $a=7$ and $b=3$.
The markers are in the 
diagonals with shifted contents $\pm r_1$ and $\pm r_2$, where $r_1 = 2$ and $r_2 = 6$.
 An example where $\ell > a\ge b$: Let  $\ell = 11$ and $k = 14$, then 
$$
\OneTLNode{Overlapping}{11}
\qquad\qquad
\lambda = (13,11) = \TikZ{[xscale=.4, yscale=-.4] 
	\filldraw[black!20] (0,0) to (7,0) to (7,1) to (3,1) to (3,2) to (0,2) to (0,0);
	\Part{13,11}
	 \node[bV,red] at (10,0){}; \node[bV,blue] at (7,1){}; \node[bV,blue] at (3,1){}; \node[bV,red] at (0,2){};
}
\qquad\hbox{with $S^{(0)}_{\mathrm{max}} = (7,3)$.}
$$
The boxes of $\lambda/S^{(0)}_{\mathrm{max}}$ have
%$$\text{indexing: } 
%\TikZ{[xscale=.5,yscale=-.5]
%	\draw (7,2) to (7,0) to (13,0) to (13,1) to (3,1) to (3,2) to (11,2) to (11,0) 
%			(8,0) to (8,2) (9,0) to (9,2) (10,0) to (10,2)  
%			(4,2) to (4,1) (5,1) to (5,2) (6,1) to (6,2) 
%			(12,0) to (12,1);
%	\Cont{3.5,1.5}{$4$}
%	\Cont{4.5,1.5}{$2$}
%	\Cont{5.5,1.5}{$1$}
%	\Cont{6.5,1.5}{$3$}
%	\Cont{7.5,1.5}{$5$}
%	\Cont{7.5,.5}{$6$}
%	\Cont{8.5,1.5}{$7$}
%	\Cont{8.5,.5}{$8$}
%	\Cont{9.5,1.5}{$9$}
%	\Cont{9.5,.5}{$10$}
%	\Cont{10.5,1.5}{$11$}
%	\Cont{10.5,.5}{$12$}
%	\Cont{11.5,.5}{$13$}
%	\Cont{12.5,.5}{$14$}
%	\draw[->, rounded corners=5pt, black!20!white, shorten >= 3pt] (6, 2)  to ++(.5,.5) to ++(-1,0) to ++(-.5,-.5);
%	\begin{scope}[->, rounded corners=10pt, black!20!white, shorten >= 3pt]
%%	\draw (4, 1) to ++(-.5,-.5) to ++(1,0) to ++(.5,.5);
%	\draw (4, 1) to ++(-.5,-.5) to ++(2,0) to ++(.5,.5);
%	\draw (7, 2)  to ++(1,1) to ++(-3,0) to ++(-1,-1);
%	\draw (3, 1) to ++(-1,-1) to ++(4,0) to ++(1,1);
%	\draw (8, 2)  to ++(1.5,1.5) to ++(-5,0) to ++(-4,-4)
%			to ++(6,0) to ++(.5,.5);
%	\draw (9, 2)  to ++(2,2) to ++(-7,0) to ++(-5,-5)
%			to ++(8,0) to ++(1,1);
%	\draw (10, 2)  to ++(2.5,2.5) to ++(-9,0) to ++(-6,-6)
%			to ++(10,0) to ++(1.5,1.5);
%	\draw (11, 2)  to ++(3,3) to ++(-11,0) to ++(-7,-7)
%			to ++(12,0) to ++(2,2);
%	\draw (11, 1)  to ++(4.5,4.5) to ++(-13,0) to ++(-8,-8)
%			to ++(14,0) to ++(2.5,2.5);
%	\draw (12, 1)  to ++(5,5) to ++(-15,0) to ++(-9,-9)
%			to ++(16,0) to ++(3,3);
%	\end{scope}
%	 \node[bV,red] at (10,0){}; \node[bV,red] at (7,1){}; \node[bV,red] at (3,1){}; 
%}$$
$$\text{shifted contents: } 
\TikZ{[xscale=.5,yscale=-.5]
	\draw (7,2) to (7,0) to (13,0) to (13,1) to (3,1) to (3,2) to (11,2) to (11,0) 
			(8,0) to (8,2) (9,0) to (9,2) (10,0) to (10,2)  
			(4,2) to (4,1) (5,1) to (5,2) (6,1) to (6,2) 
			(12,0) to (12,1);
	\Cont{3.5,1.5}{-$2$}
	\Cont{4.5,1.5}{-$1$}
	\Cont{5.5,1.5}{$0$}
	\Cont{6.5,1.5}{$1$}
	\Cont{7.5,1.5}{$2$}
	\Cont{7.5,.5}{$3$}
	\Cont{8.5,1.5}{$3$}
	\Cont{8.5,.5}{$4$}
	\Cont{9.5,1.5}{$4$}
	\Cont{9.5,.5}{$5$}
	\Cont{10.5,1.5}{$5$}
	\Cont{10.5,.5}{$6$}
	\Cont{11.5,.5}{$7$}
	\Cont{12.5,.5}{$8$}
	 \node[bV,red] at (10,0){}; \node[bV,blue] at (7,1){}; \node[bV,blue] at (3,1){}; 
}
$$
Then $\cc$ is the rearrangement of the absolute values of 
$(-2,-1,0,1,2,3,3,4,4,5,5,6,7,8)$ into increasing order
and $J = \{\vep_4, \vep_2-\vep_1, \vep_4-\vep_3, \vep_6-\vep_5, \vep_8-\vep_7, \vep_{10}-\vep_9, \vep_{12}-\vep_{11}\}$.
The configuration of boxes $\kappa$ corresponding to $(\cc, J)$ has indexing of boxes
$$
\TikZ{[xscale=.5,yscale=-.5]
	\draw (-1,2) to (-1,0) to (13,0) to (13,1) to (-3,1) to (-3,2) to (11,2) to (11,0);
	\foreach \x in {0,1,...,10} {\draw (\x,0) to (\x,2);}
	\draw (12,0) to (12,1) (-2,1) to (-2,2);
	\foreach \x in {2, 4, ..., 12}{\Cont{4.5+.5*\x,.5}{$\x$}}
	\foreach \x in {1, 3, ..., 11}{\Cont{5+.5*\x,1.5}{$\x$}}
	\foreach \x in {2, 4, ..., 12}{\Cont{5.5-.5*\x,1.5}{-$\x$}}
	\foreach \x in {1, 3, ..., 11}{\Cont{5-.5*\x,.5}{-$\x$}}
	\Cont{11.5,.5}{$13$}
	\Cont{12.5,.5}{$14$}
	\Cont{-1.5,1.5}{-$13$}
	\Cont{-2.5,1.5}{-$14$}
	 \node[bV,red] at (10,0){}; \node[bV,blue] at (7,1){}; \node[bV,blue] at (3,1){};  \node[bV,red] at (0,2){}; 
}
$$

\end{example}

\begin{example}\label{ex:SW-modules}
Let $a=6$ and $b=3$ to take advantage of the setting and notation of Figure \ref{Fig:2bdryTLBrat}.
The markers are in the 
diagonals with shifted contents $\pm r_1$ and $\pm r_2$, where $r_1 = \frac32$ and $r_2 = \frac{11}{2}$.
\begin{enumerate}[(1)]
\item An example where $\ell > a\ge b$: Let  $\ell = 8$ and $k = 9$, then 
$$
\tikzstyle{FULL}=[fill=yellow!20, circle, inner sep=2pt, draw]
\tikzstyle{SumToL}=[fill=blue!15, draw]
\tikzstyle{SumToBKL}=[fill=cyan!15, diamond, inner sep=1.5pt, draw]
\tikzstyle{SumToA}=[fill=green!15,  regular polygon,regular polygon sides=3, inner sep=1pt, draw]
\tikzstyle{Overlapping}=[fill=red!15, regular polygon,regular polygon sides=8, inner sep=1pt, draw]
\OneTLNode{Overlapping}{8}
\qquad\qquad
\lambda = (10,8) = \TikZ{[xscale=.4, yscale=-.4] 
	\filldraw[black!20] (0,0) to (6,0) to (6,1) to (3,1) to (3,2) to (0,2) to (0,0);
	\Part{10,8}
	 \node[bV,red] at (9,0){}; \node[bV,blue] at (6,1){}; \node[bV,blue] at (3,1){}; \node[bV,red] at (0,2){};
}
\qquad\hbox{with $S^{(0)}_{\mathrm{max}} = (6,3)$.}
$$
The boxes of $\lambda/S^{(0)}_{\mathrm{max}}$ have
%$$\text{indexing: } 
%\TikZ{[xscale=.5,yscale=-.5]
%	\draw (6,2) to (6,0) to (10,0) to (10,1) to (3,1) to (3,2) to (8,2) to (8,0) (7,2) to (7,0) (4,2) to (4,1) (5,1) to (5,2) (9,0) to (9,1);
%	\Cont{3.5,1.5}{$3$}
%	\Cont{4.5,1.5}{$1$}
%	\Cont{5.5,1.5}{$2$}
%	\Cont{6.5,1.5}{$4$}
%	\Cont{6.5,.5}{$5$}
%	\Cont{7.5,1.5}{$6$}
%	\Cont{7.5,.5}{$7$}
%	\Cont{8.5,.5}{$8$}
%	\Cont{9.5,.5}{$9$}
%	\begin{scope}[->, rounded corners=10pt, black!20!white, shorten >= 3pt]
%	\draw (4, 1) to ++(-.5,-.5) to ++(1,0) to ++(.5,.5);
%	\draw (6, 2)  to ++(.5,.5) to ++(-2,0) to ++(-.5,-.5);
%	\draw (3, 1) to ++(-1,-1) to ++(3,0) to ++(1,1);
%	\draw (7, 2)  to ++(1,1) to ++(-4,0) to ++(-3.5,-3.5) to ++(5,0) to ++(.5,.5);
%	\draw (8, 2)  to ++(1.5,1.5) to ++(-6,0) to ++(-4.5,-4.5) to ++(7,0) to ++(1,1);
%	\draw (8, 1)  to ++(3,3) to ++(-8,0) to ++(-5.5,-5.5) to ++(9,0) to ++(1.5,1.5);
%	\draw (9, 1)  to ++(3.5,3.5) to ++(-10,0) to ++(-6.5,-6.5) to ++(11,0) to ++(2,2);
%	\end{scope}
%	  \node[bV,red] at (9,0){}; \node[bV,red] at (6,1){}; \node[bV,red] at (3,1){};  \node[bV,red] at (0,2){};
%}$$
$$\text{shifted contents: } 
\TikZ{[xscale=.5,yscale=-.5]
	\draw (6,2) to (6,0) to (10,0) to (10,1) to (3,1) to (3,2) to (8,2) to (8,0) (7,2) to (7,0) (4,2) to (4,1) (5,1) to (5,2) (9,0) to (9,1);
	\Cont{3.5,1.5}{-$\frac32$}
	\Cont{4.5,1.5}{-$\frac12$}
	\Cont{5.5,1.5}{$\frac12$}
	\Cont{6.5,1.5}{$\frac32$}
	\Cont{6.5,.5}{$\frac52$}
	\Cont{7.5,1.5}{$\frac52$}
	\Cont{7.5,.5}{$\frac72$}
	\Cont{8.5,.5}{$\frac92$}
	\Cont{9.5,.5}{$\frac{11}2$}
	 \node[bV,red] at (9,0){};  \node[bV,blue] at (6,1){}; \node[bV,blue] at (3,1){}; % \node[bV,red] at (0,2){};
}
$$
Then $\cc$ is the rearrangement of the absolute values of 
$(-\frac52, -\frac32, -\frac12, \frac12, \frac32, \frac52, \frac72, \frac92,\frac{11}{2})$ into increasing order
and $J= \{\vep_3,  \vep_3-\vep_2, \vep_5-\vep_4, \vep_7-\vep_6\}$.  
The configuration of boxes $\kappa$ corresponding to $(\cc, J)$ has indexing of boxes
$$
\TikZ{[xscale=.5,yscale=-.5]
	\draw (2,0) to (11,0) to (11,1) to (9,1) to (9,2) to (0,2) to (0,1) to (2,1) to (2,0)  (1,2) to (1,1) (2,2) to (2,0) (3,2) to (3,0) (4,2) to (4,0) (5,2) to (5,0) (6,2) to (6,0) (7,2) to (7,0) (8,2) to (8,0) (9,1) to (9,0) (10, 1) to (10,0) (9,1) to (2,1);%6,2) to (6,0) to (10,0) to (10,1) to (3,1) to (3,2) to (8,2) to (8,0) (7,2) to (7,0) (4,2) to (4,1) (5,1) to (5,2) (9,0) to (9,1);
	\Cont{0.5,1.5}{$-9$}
	\Cont{1.5,1.5}{$-8$}
	\Cont{2.5,1.5}{$-7$}
	\Cont{3.5,1.5}{$-5$}
	\Cont{4.5,1.5}{$-3$}
	\Cont{5.5,1.5}{$-1$}
	\Cont{6.5,1.5}{$2$}
	\Cont{7.5,1.5}{$4$}
	\Cont{8.5,1.5}{$6$}
	\Cont{2.5,.5}{$-6$}
	\Cont{3.5,.5}{$-4$}
	\Cont{4.5,.5}{$-2$}
	\Cont{5.5,.5}{$1$}
	\Cont{6.5,.5}{$3$}
	\Cont{7.5,.5}{$5$}
	\Cont{8.5,.5}{$7$}
	\Cont{9.5,.5}{$8$}
	\Cont{10.5,.5}{$9$}
	 \node[bV,red] at (10,0){}; \node[bV,blue] at (7,1){}; \node[bV,blue] at (4,1){}; \node[bV,red] at (1,2){};
}
\quad\hbox{with}\quad
P(\cc) = \left\{\begin{array}{l} \vep_3, \vep_9, \vep_2 + \vep_1, 
\vep_3-\vep_2, \vep_3-\vep_1 \\
\vep_5-\vep_4, \vep_5-\vep_3, \vep_6-\vep_4, \vep_6-\vep_3, \\
\vep_7-\vep_6, \vep_7-\vep_5, \vep_8-\vep_7, \vep_9-\vep_8
\end{array}\right\}.
$$

\item  An example with $a\ge \ell > b$: Let $k=3$ and $\ell=5$, so that $a+b+k-\ell = 7$. 
$$
\OneTLNode{SumToBKL}{5}	
\qquad\qquad
\lambda = (7,5) = \TikZ{[xscale=.4, yscale=-.4] 
	\filldraw[black!20] (0,0) to (6,0) to (6,1) to (3,1) to (3,2) to (0,2) to (0,0);
	\Part{7,5}
	\node[bV,blue] at (6,1){}; \node[bV,blue] at (3,1){};  \node[bV,red] at (0,2){};  %\node[bV,red] at (9,0){}; 
}
\qquad\hbox{with $S^{(0)}_{\mathrm{max}} = (6,3)$.}$$
The boxes of $\lambda/S^{(0)}_{\mathrm{max}}$ have
$$
%\text{indexing: }
%\TikZ{[xscale=.5,yscale=-.5]
%	\draw (6,0) rectangle (7,1)  (3,1) rectangle (5,2) (4,1) to (4,2);
%	\Cont{6.5,.5}{$3$}
%	\Cont{3.5,1.5}{$2$}
%	\Cont{4.5,1.5}{$1$}
%	\node[bV,red] at (6,1){}; \node[bV,red] at (3,1){};  \node[bV,red] at (9,0){};  \node[bV,red] at (0,2){};
%}
%\qquad
\text{shifted contents: }
\TikZ{[xscale=.5,yscale=-.5]
	\draw (6,0) rectangle (7,1)  (3,1) rectangle (5,2) (4,1) to (4,2);
	\Cont{6.5,.5}{$\frac52$}
	\Cont{3.5,1.5}{-$\frac32$}
	\Cont{4.5,1.5}{-$\frac{1}{2}$}
	\node[bV,blue] at (6,1){}; \node[bV,blue] at (3,1){}; % \node[bV,red] at (9,0){};  \node[bV,red] at (0,2){};
}
$$
Then $\cc$ is the rearrangement of the absolute values of 
$(\frac12, \frac32, \frac{5}{2})$ in increasing order and $J = \{\vep_2\}$.
The configuration of boxes $\kappa$ corresponding to $(\cc, J)$ is
$$
\TikZ{[xscale=.5,yscale=-.5]
	\draw (2,0) to (5,0) to (5,1) to (0,1) to (0,2) to (3,2) to (3,0)  (1,2) to (1,1)  (2,2) to (2,0) (4,1) to (4,0);%6,2) to (6,0) to (10,0) to (10,1) to (3,1) to (3,2) to (8,2) to (8,0) (7,2) to (7,0) (4,2) to (4,1) (5,1) to (5,2) (9,0) to (9,1);
	\Cont{0.5,1.5}{$-3$}
	\Cont{1.5,1.5}{$-2$}
	\Cont{2.5,1.5}{$-1$}
	\Cont{2.5,.5}{$1$}
	\Cont{3.5,.5}{$2$}
	\Cont{4.5,.5}{$3$}
	 \node[bV,blue] at (1,1){}; \node[bV,blue] at (4,1){};
}
\quad\hbox{with}\quad
P(\cc) = \{ \vep_2, \vep_2-\vep_1, \vep_3-\vep_2\}
$$

\item An example with $a\ge b\ge \ell$: Let $k=3$ and $\ell=2$, so that $a+b+k-\ell = 10$. Then 
$$
\OneTLNode{SumToL}{2}	
\qquad\qquad
\lambda = (10,2) = \TikZ{[xscale=.4, yscale=-.4] 
	\filldraw[black!20] (0,0) to (7,0) to (7,1) to (2,1) to (2,2) to (0,2) to (0,0);
	\Part{10,2}
	\node[bV,blue] at (6,1){}; \node[bV,blue] at (3,1){};  \node[bV,red] at (0,2){};  \node[bV,red] at (9,0){};
}
\qquad\hbox{with $S^{(0)}_{\mathrm{max}} = (7,2)$.}
$$
The boxes of $\lambda/S^{(0)}_{\mathrm{max}}$ have
$$
%\text{indexing: }
%\TikZ{[xscale=.5,yscale=-.5]
%	\draw (7,0) rectangle (10,1) (8,1) to (8,0)  (9,1) to (9,0);
%	\Cont{7.5,.5}{$1$}
%	\Cont{8.5,.5}{$2$}
%	\Cont{9.5,.5}{$3$}
%	\node[bV,red] at (6,1){}; \node[bV,red] at (3,1){};  \node[bV,red] at (0,2){};  \node[bV,red] at (9,0){};
%}
%\qquad
\text{shifted contents: }
\TikZ{[xscale=.5,yscale=-.5]
	\draw (7,0) rectangle (10,1) (8,1) to (8,0)  (9,1) to (9,0);
	\Cont{7.5,.5}{$\frac72$}
	\Cont{8.5,.5}{$\frac92$}
	\Cont{9.5,.5}{$\frac{11}{2}$}
%	\node[bV,red] at (6,1){}; \node[bV,red] at (3,1){};  \node[bV,red] at (0,2){};  
	\node[bV,red] at (9,0){};
}.
$$
Then $\cc$ is the rearrangement of the absolute values of $(\frac72, \frac92, \frac{11}{2})$ in increasing order and
$J =\emptyset$.
The configuration of boxes $\kappa$ corresponding to $(\cc, J)$ is
$$
\TikZ{[xscale=.5,yscale=-.5]
	\draw (8,0) to (11,0) to (11,1) to (8,1) to (8,0) (0,1) to (3,1) to (3,2) to (0,2) to (0,1)  (1,2) to (1,1)  (2,2) to (2,1) (9,1) to (9,0) (10,1) to (10,0);%6,2) to (6,0) to (10,0) to (10,1) to (3,1) to (3,2) to (8,2) to (8,0) (7,2) to (7,0) (4,2) to (4,1) (5,1) to (5,2) (9,0) to (9,1);
	\Cont{0.5,1.5}{$-3$}
	\Cont{1.5,1.5}{$-2$}
	\Cont{2.5,1.5}{$-1$}
	\Cont{8.5,.5}{$1$}
	\Cont{9.5,.5}{$2$}
	\Cont{10.5,.5}{$3$}
	 \node[bV,red] at (1,2){}; \node[bV,red] at (10,0){}; \node[bV,blue] at (4,1){}; \node[bV,blue] at (7,1){};
}
\quad\hbox{with}\quad
P(\cc) = \{ \vep_2-\vep_1, \vep_3-\vep_2\}.
$$

\end{enumerate}
\end{example}

\pagebreak \vfill

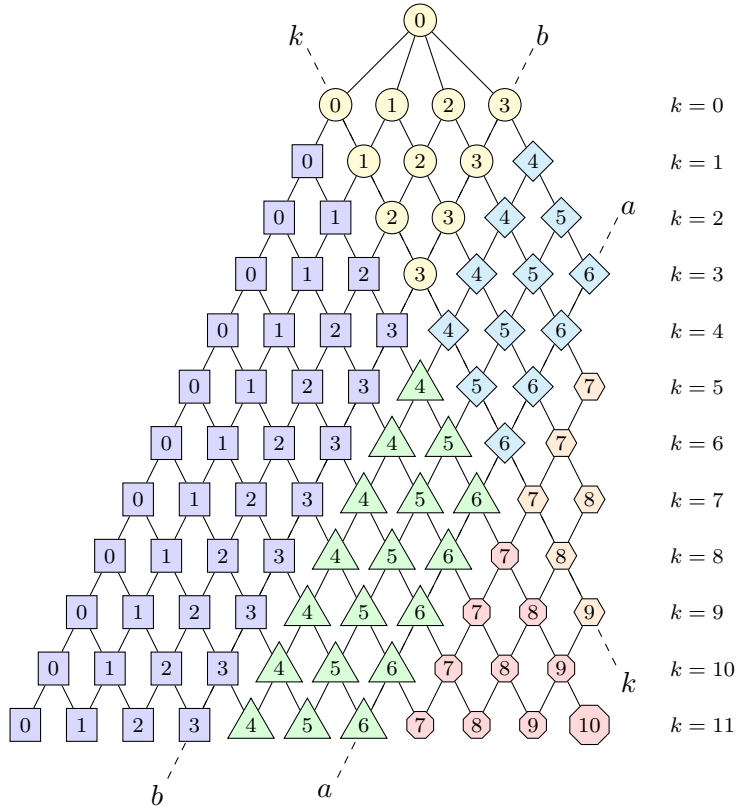
\begin{figure}
\caption{The Temperley-Lieb Bratteli diagram for $a=6$ and $b=3$, levels $0$--$9$.
Partitions $\lambda = (a+b+k-\ell, \ell)$ are labeled by $\ell$.  The dimension formulas are
consequences of Remark \ref{dimformulas}.}
\label{Fig:2bdryTLBrat}
\newcommand\TLNode[1]{ \node[#1] (\x\k) at (\x-.5*\k,\k){\scriptsize$\x$};}
\begin{center}
\begin{tikzpicture}[xscale=.75, yscale=-.75]
\draw [dashed] (5,2) node[above right, inner sep=1pt] {$a$} -- (0,12) node[below left, inner sep=1pt] {$a$};
\draw [dashed] (3.5,-1)	node[above right, inner sep=1pt] {$b$} -- (-3,12) node[below left, inner sep=1pt] {$b$};
\draw [dashed] (-.5,-1)	node[above left, inner sep=1pt] {$k$} -- (5,10) node[below right, inner sep=1pt] {$k$};
\foreach\k [evaluate=\k as \w using \k*.5+4.5,remember=\k as \lasty, count=\c from 3] in {0,1,...,11}{
	\foreach \x [remember=\x as \lastx] in {0,...,\w}{
	% Vertices: 
	\ifnum \x<\c \else \breakforeach \fi
	\ifnum \x<4
		\ifnum \x<\k \TLNode{SumToL}
			\else  \TLNode{FULL}
		\fi
	\else
		\ifnum \x > 6 
			\ifnum \x < \k \TLNode{Overlapping} \else \TLNode{OverlappingB} \fi
		\else
			\ifnum \x < \k  \TLNode{SumToA} \else 
				\ifnum \x> \c  \else	\TLNode{SumToBKL}	\fi
			\fi
		\fi
	\fi
	% Edges:
	\ifnum \k>0
		\ifnum \x > 0
			\draw (\x\k) to (\lastx\lasty);
			\fi
		\ifnum \x<\c 
			\ifnum \x < \w
				\draw (\x\k) to (\x\lasty);
			\fi
		\fi
	\fi
	}}
	\draw (64)--(63) (76)--(75) (88)--(87) (910)--(99);
\node[FULL] (a) at (1.5,-1.5){\scriptsize$0$};
\foreach \x in {0,1,...,3}{\draw (a)--(\x0);}
\foreach \k in {0,1,...,11}{\node[right] at (5.75,\k){\scriptsize$k=\k$}; }
\end{tikzpicture}
%\bigskip
%
%\begin{tikzpicture}[xscale=-.75, yscale=-.75]
%\node[FULL, , label=left:{dim$\Big($}, label=right:{$\Big)
%\displaystyle = \sum_{c=\ell-k}^\ell \binom{k}{\ell-c} = \sum_{i=0}^k \binom{k}{i} = 2^k
%\qquad 0\le \ell-k, \ \ a \ge b\ge \ell$}] at (-7.25,0) {\scriptsize$\ell$};
%\node[SumToL, label=left:{dim$\Big($}, label=right:{$\Big)\displaystyle =
%\sum_{c=0}^\ell \binom{k}{\ell-c} = \sum_{i=0}^{\ell}\binom{k}{i}
%\qquad 0> \ell-k, \ \ a \ge b\ge \ell$}] at (-7.25,2) {\scriptsize$\ell$};
%\node[SumToBKL, label=left:{dim$\Big($}, label=right:{$\Big)
%\displaystyle =\sum_{c=\ell-k}^b \binom{k}{\ell-c} = \sum_{i=0}^{b+k-\ell}\binom{k}{i} 
%\qquad 0\le \ell-k, \ \ a\ge \ell > b$}] at (-7.25,4) {\scriptsize$\ell$};
%\node[SumToA, label=left:{dim$\Big($}, label=right:{$\Big)
%\displaystyle =\sum_{c=0}^b \binom{k}{\ell-c} = \sum_{i=\ell-b}^{\ell}\binom{k}{i}
%\qquad 0 > \ell-k, \ \ a\ge \ell > b$}] at (-7.25,6) {\scriptsize$\ell$};
%\node[OverlappingB,label=left:{dim$\Big($}, label=right:{$\Big) 
%\displaystyle = \sum_{c=\ell-k}^{b} 
%		\binom{k}{\ell - c} - \binom{k}{\ell-(a+b-c)-1}
%\qquad 0\le \ell-k, \ \ \ell> a\ge b$}] at (-7.25,8) {\scriptsize$\ell$};
%\node[Overlapping,label=left:{dim$\Big($}, label=right:{$\Big) 
%\displaystyle = \sum_{c=0}^{b} 
%		\binom{k}{\ell - c} - \binom{k}{\ell-(a+b-c)-1}
%\qquad 0 > \ell-k, \ \ \ell> a\ge b$}] at (-7.25,10) {\scriptsize$\ell$};
%\end{tikzpicture}
\end{center}

\end{figure}

%%%%%%%%%%%%%%%%%%%%
%%%%% BIBLIOGRAPHY %%%%%%
%%%%%%%%%%%%%%%%%%%%
%


\begin{thebibliography}{99}
\bibitem[Dau]{Da} Z.\ Daugherty, \emph{Degenerate two-boundary centralizer algebras}, 
Pacific J.\ of Math.\ \textbf{258-1} (2012) 91--142, MR2972480, arXiv:1007.3950
 

\bibitem[DR]{DR} Z.\ Daugherty and A.\ Ram, \emph{Two boundary Hecke algebras and the combinatorics of type C},
arXiv:1804.10296.

\bibitem[GM]{GM} J.\ Gonz\'alez-Meneses, \emph{Basic results on braid groups}, Ann.\ Math.\ Blaise Pascal 18 (2011) 15--59, MR2830088, arXiv:1010.0321.

\bibitem[GHJ]{GHJ} F.\ Goodman, P.\ de la Harpe and V.F.R. Jones, {\sl  Coxeter graphs and towers of algebras}, Mathematical Sciences Research Institute Publications \textbf{14} Springer-Verlag, New York, 1989 ISBN: 0-387-96979-9, MR0999799 

\bibitem[GMP07]{GMP07} R.M.\ Green, P.P.\ Martin and A.E.\ Parker, 
\emph{Towers of recollement and bases for diagram algebras: planarity and beyond}, 
J.\ Algebra \textbf{316} (2007) 392-452, MR2354870, arXiv: 0610971.

\bibitem[GMP08]{GMP08} R.M.\ Green, P.P.\ Martin and A.E.\ Parker, 
\emph{On the non-generic representation theory of the symplectic blob algebra},
arXiv:0807.4101 %Seems not to have appeared in a journal %it is listed as a preprint on Alison's web page

\bibitem[GMP12]{GMP12} R.M.\ Green, P.P.\ Martin and A.E.\ Parker, 
\emph{A presentation for the symplectic blob algebra}, J.\ Algebra Appl.\ \textbf{11} (2012) 1250060, MR2928127
%seems not to have appeared on arXiv

\bibitem[GMP17]{GMP17}  R.M.\ Green, P.P.\ Martin and A.E.\ Parker, 
\emph{On quasi-heredity and cell module homomorphisms in the symplectic blob algebra},
arXiv:1707.06520. %not yet appeared in a journal

\bibitem[KMP16]{KMP16} O.H.\ King, P.P.\ Martin and A.E.\ Parker,
\emph{Decomposition matrices and blocks for the symplectic blob algebra over the complex field},
arXiv:1611.06968. %not yet appeared in a journal.

%\bibitem[KMP18]{KMP18} O.H.\ King, P.P.\ Martin and A.E.\ Parker, 
%\emph{On blocks for the symplectic blob algebra in the sub-critical case},
%in preparation.
    
\bibitem[GN]{GN} J.\ de Gier and A.\ Nichols,  \emph{
The two-boundary Temperley-Lieb algebra}, 
J.\ Algebra \textbf{321} (2009) 1132--1167, 
MR2489894, arXiv:math/0703338

\bibitem[GP]{GP} J.\ de Gier and P.\ Pyatov, \emph{Bethe Ansatz for the Temperley-Lieb loop model
with open boundaries}, J.\ Stat.\ Mech.\ (2004) P03002, hep-th/0312235.

\bibitem[GNPR]{GNPR} J.\ de Gier, A.\ Nichols, P.\ Pyatov, V.\ Rittenberg, \emph{Magic in the spectra
of the XXZ quantum spin chain with boundaries at $\Delta=0$ and $\Delta=-1/2$}, Nucl.\ Phys.\ B 729 (2005) 382, hep-th0505062.


\bibitem[Kat]{Kat} S.\ Kato, \emph{An exotic Deligne-Langlands correspondence for symplectic groups}, Duke Math.\ J.\ \textbf{148} (2009) 305--371, MR2524498, arXiv:0601155.


\bibitem[Mac]{Mac} I.\ G.\ Macdonald,  {\sl Symmetric functions and Hall polynomials},
Second Edition, Oxford University Press, 1995. ISBN: 0-19-853489-2 MR1354144 %(96h:05207)  

\bibitem[M03]{M03} I.\ G.\ Macdonald,  {\sl Affine Hecke algebras and orthogonal polynomials}, Cambridge Tracts in Mathematics \textbf{157} Cambridge University Press, Cambridge, 2003 ISBN: 0-521-82472-9 MR1976581.

\bibitem[Ree]{Re12} A.\ Reeves, \emph{Tilting modules for the symplectic blob algebra}, 
arXiv:1111.0146.  %seems not to have appeared in a journal.


\end{thebibliography}
\end{document}